\documentclass[11pt]{amsart}
\usepackage{amsmath,amsfonts,amsbsy,amsgen,amscd,mathrsfs,amssymb,amsthm}
\usepackage{enumerate,mathtools}
\usepackage{tabularx}
\usepackage{array}
\newcolumntype{P}[1]{>{\centering\arraybackslash}p{#1}}
\usepackage{makecell}

\usepackage{stmaryrd}

\usepackage{fullpage}
\usepackage{url}
\usepackage{bm}
\usepackage{mathrsfs}

\usepackage[colorlinks=true,linkcolor=blue]{hyperref} 
\makeatletter
\renewcommand*{\eqref}[1]{%
  \hyperref[{#1}]{\textup{\tagform@{\ref*{#1}}}}%
}
\makeatother

\usepackage{tikz}
\usetikzlibrary{shadings}

\numberwithin{equation}{section}

\newtheorem{theorem}{Theorem}[section]
\newtheorem{proposition}[theorem]{Proposition}

\newtheorem{corollary}[theorem]{Corollary}
\newtheorem{lemma}[theorem]{Lemma}
\newtheorem{example}[theorem]{Example}
\theoremstyle{definition}
\newtheorem{definition}[theorem]{Definition}
\newtheorem{claim}[theorem]{Claim}

\newtheorem{remark}[theorem]{Remark}
\newtheorem{assumption}[theorem]{Assumption}
\newtheorem*{claim*}{Claim}

\newcommand{\R}{\mathbb R}
\newcommand{\V}{\mathsf{V}}

\newcommand{\Vol}{\mathrm{Vol}}

\newcommand{\sspan}{\mathop{\mathrm{span}}}
\newcommand{\Id}{\operatorname{Id}}
\newcommand{\aint}{\alpha}
\newcommand{\abar}{\bar{\alpha}}
\newcommand{\bb}{\beta}
\newcommand{\incomp}{\nsim}
\newcommand{\poly}{K}
\newcommand{\bsg}{\beta}
\newcommand{\s}{\circ}
\newcommand{\lin}{\operatorname{Lin}}
\newcommand{\aff}{\mathop{\mathrm{aff}}}
\newcommand{\LL}{\ell}
\newcommand{\I}{\textnormal{I}}
\newcommand{\mixed}{y^{\sigma}_{\textnormal{crit}}}

\numberwithin{figure}{section}

\title{The extremals of Stanley's inequalities for partially ordered sets}

\author{Zhao Yu Ma}
\address{Department of Mathematics, Massachusetts Institute of Technology, 
Cambridge, MA, USA}
\email{zhaoyuma@mit.edu}

\author{Yair Shenfeld}
\address{Division of Applied Mathematics, Brown University, Providence, RI, USA}
\email{Yair\_Shenfeld@Brown.edu}

\begin{document}
\maketitle

\begin{abstract}
Stanley's inequalities for partially ordered sets establish important log-concavity relations for sequences of linear extensions counts. Their extremals however, i.e., the equality cases of these inequalities, were until now poorly understood with even conjectures lacking. In this work, we solve this problem by providing a complete characterization of the extremals of Stanley's inequalities. Our proof is based on building a new ``dictionary" between the combinatorics of partially ordered sets and the geometry of convex polytopes, which captures their extremal structures.
\end{abstract}

\tableofcontents

\section{Introduction} 
\label{sec:intro}
\subsection{Log-concave sequences}
Finite sequences of numbers $\{a_i\}_{i=1}^n$ often serve as a powerful way to encode properties of algebraic, geometric, and combinatorial  objects: $a_i$ can stand for the $i$th coefficient of a Schur polynomial, the dimension of the $i$th cohomology group of a toric variety, or the number of $i$-elements independent sets of a
matroid, etc. The properties and interrelations of the elements of the sequence  $\{a_i\}_{i=1}^n$ provide valuable information about the underlying mathematical objects. Here we focus on \emph{log-concavity} relations:
\begin{align*}
a_i^2\ge a_{i-1}a_{i+1}\quad\text{for all }i=2,\ldots, n-1,
\end{align*}
which are tied to notions of \emph{positivity}  and \emph{unimodality} \cite{Ber89,Ber94,Sta89,Stanley00, SW14, Bra15}. The question that motivates our work is the following: Suppose a log-concave sequence $\{a_i\}_{i=1}^n$, whose elements stand for some algebraic/geometric/combinatorial properties of a mathematical object, satisfies
\begin{align*}
a_j^2=a_{j-1}a_{j+1}\quad \text{for some \emph{fixed} index $j$.}
\end{align*}
What can we deduce about the underlying object? This question of identifying the \emph{extremals} of the sequence $\{a_i\}_{i=1}^n$ is fundamental for a number of reasons. At the very basic level, the structure of the extremals  is a basic property of the sequence which we ought to understand. More concretely, information about the extremals can provide information about the shape of the sequence which cannot be inferred from the log-concavity property alone: see Figure \ref{fig:stshape}.
Additionally, if one wishes to improve on the log-concavity property by having $a_i^2-a_{i-1}a_{i+1}\ge d_i$ for some non-trivial $d_i\ge 0$, then usually understanding the extremals of $\{a_i\}$, and hence the vanishing of $d_i$, is a necessary first step. From a different perspective, there are interesting questions related to combinatorial interpretations and computational complexity of the difference $a_i^2-a_{i-1}a_{i+1}$, where characterizing the vanishing condition $a_i^2=a_{i-1}a_{i+1}$ is a basic question \cite{P19,P22survey,chan2023equality}.

Establishing that a given sequence, which arises in an algebraic/geometric/combinatorial  setting, is  log-concave is a difficult problem, with many remaining open questions. In recent years, major advances were achieved on the fronts of proving log-concavity relations for various important sequences in combinatorics \cite{Huh18,kalai2022work,CP22a}.  These approaches  rely on building ``dictionaries"  between combinatorial and geometric-algebraic objects, and then using (or taking inspiration from) already-known log-concavity relations in the geometric-algebraic settings. What is missing, however, are the analogous dictionaries between the \emph{extremals} arising in the combinatorial and geometric-algebraic settings. In this work, we take a  step towards bridging this gap by focusing on  the  correspondence between combinatorics and \emph{convex} geometry due to R. Stanley in the context of partially ordered sets. We will build such a dictionary and, as a consequence, completely characterize the extremal structures arising in Stanley's inequalities  \cite{Sta81}. The question of the characterization of these  extremals was already raised by Stanley, but even conjectures on these extremals were lacking. As we will see, this is for a good reason since, surprisingly, the extremal structures of our combinatorial sequences will display the richness and subtle nature of their geometric counterparts.

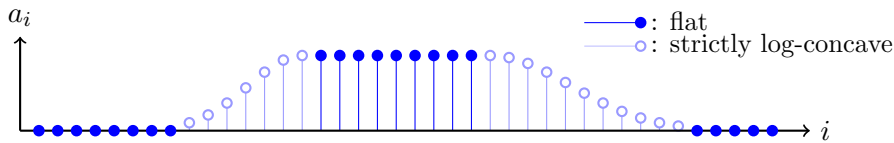
\begin{figure}[h]
\centering
\begin{tikzpicture}[scale=.25]

\foreach \i in {9,...,15}
{
	\draw[color=blue!40!white] (\i,{4*exp(-(\i-15)^2/16)}) -- (\i,0);
	\draw[color=blue!40!white,fill=white,thick] (\i,{4*exp(-(\i-15)^2/16)}) circle (0.25);
}

\foreach \i in {16,...,24}
{
	\draw[color=blue,fill=blue,thick] (\i,4) circle (0.25);
	\draw[color=blue] (\i,4) -- (\i,0);
}

\foreach \i in {25,...,35}
{
	\draw[color=blue!40!white] (\i,{4*exp(-(25-\i)^2/36)}) -- (\i,0);
	\draw[color=blue!40!white,fill=white,thick] (\i,{4*exp(-(25-\i)^2/36)}) circle (0.25);
}

\draw[thick,->] (0,0) -- (0,5) node[above] {$a_i$};
\draw[thick,->] (0,0) -- (42,0) node[right] {$i$};

\foreach \i in {1,...,8}
{
	\draw[color=blue,fill=blue,thick] (\i,0) circle (0.25);
}

\foreach \i in {36,...,40}
{
	\draw[color=blue,fill=blue,thick] (\i,0) circle (0.25);
}

\draw[color=blue] (30,5.75) -- (33,5.75);
\draw[color=blue,fill=blue,thick] (33,5.75) circle (0.25);
\draw (33,5.85) node[right] {\small : flat};

\draw[color=blue!25!white] (30,4.5) -- (33,4.5);
\draw[color=blue!40!white,fill=white,thick] (33,4.5) circle (0.25);
\draw (33,4.55) node[right] {\small : strictly log-concave};
\end{tikzpicture}
\caption{The extremals of this log-concave sequence (cf. \eqref{eq:Stanleyk=1}) are such that $a_j^2=a_{j-1}a_{j+1}\Rightarrow a_{j-1}=a_j=a_{j+1}$, corresponding to the flat parts of the sequence. The width of each of the flat parts can be characterized as well. This precise description of the shape of the sequence cannot be obtained from the log-concavity property alone.\label{fig:stshape}}
\end{figure}

\subsection{Stanley's inequalities}
\label{subsec:StanleyInq}
Let $\abar=\{y_1,\ldots,y_{n-k}\}\cup\{x_1,\ldots,x_k\}$ be a partially ordered set (poset) of $n$ elements with a fixed chain $x_1<\cdots<x_k$ of length $k$. The set of linear extensions of $\abar$ is the set of bijections of $\abar$ into $[n]:=\{1,\ldots,n\}$ which are order-preserving:
\[
\mathcal N:=\{\text{bijections }\sigma:\abar\to [n]: w\le z\Rightarrow \sigma(w)\le \sigma(z)~\forall ~w,z\in\abar\}.
\]
We are interested in linear extensions which send the elements in the  chain $x_1<\cdots<x_k$ into \emph{fixed} locations. Fix $1\le i_1<\cdots <i_k\le n$ and fix $\LL\in [k]$ such that $i_{\LL-1}+1<i_{\LL}<i_{\LL+1}-1$. For $\s\in\{-,=,+\}$, let
\begin{align*}
\mathcal N_{\s}:=\{\sigma\in\mathcal N:\sigma(x_j)=i_j~\forall\, j\in [k]\backslash \{\LL\}\text{ and }\sigma(x_{\LL})=i_{\LL}+ 1_{\s} \},
\end{align*}
where $1_{\s}:=1_{\{\s\text{ is }+\}}-1_{\{\s\text{ is }-\}}$. In words, whenever $j\neq \LL$,  $x_j$ is placed at $i_j$, and when $j=\LL$,  $x_{\LL}$ is placed at one of the locations in $\{i_{\LL}-1,i_{\LL},i_{\LL}+1\}$, depending on the sign of $\s\in\{-,=,+\}$; see Figure \ref{fig:linext}. 
\begin{figure}[h]
\begin{tikzpicture}[scale=.2]

\draw (-25,8) node[below]{1};
\draw (-17,8) node[below]{\textcolor{black}{$i_1$}};
\draw (-7.5,8) node[below]{\textcolor{black}{$i_{\LL-1}$}};
\draw (1,8) node[below]{\textcolor{red}{$i_{\LL}-1$}};
\draw (6,8) node[below]{\textcolor{red}{$i_{\LL}$}};
\draw (11,8) node[below]{\textcolor{red}{$i_{\LL}+1$}};
\draw (20,8) node[below]{\textcolor{black}{$i_k$}};
\draw (25,8) node[below]{$n$};

\draw[color=black,dashed] (-25,0)--(25,0);
\draw (-22.5,0) node[below]{\textcolor{black}{$\cdots$}};
\draw (-20,0) node[below]{\textcolor{black}{$y_3$}};
\draw[color=black,->] (-17,1)--(-17,5);
\draw (-17,0) node[below]{\textcolor{black}{$x_1$}};
\draw (-14,0) node[below]{\textcolor{black}{$y_2$}};
\draw (-11,0) node[below]{\textcolor{black}{$\cdots$}};
\draw (-7.5,0) node[below]{\textcolor{black}{$x_{\LL-1}$}};
\draw[color=black,->] (-7.5,1)--(-7.5,5);
\draw (-4,0) node[below]{\textcolor{black}{$\cdots$}};
\draw (-1,0) node[below]{\textcolor{black}{$y_{1}$}};
\draw (6,0) node[below]{\textcolor{red}{$x_{\LL}$}};
\draw (14,0) node[below]{\textcolor{black}{$y_4$}};
\draw (17,0) node[below]{\textcolor{black}{$\cdots$}};
\draw (20,0) node[below]{\textcolor{black}{$x_k$}};
\draw[color=black,->] (20,1)--(20,5);
\draw (24,0) node[below]{\textcolor{black}{$\cdots$}};

\draw[color=black,dashed,->] (5.5,0.5)--(1,5);
\draw (-0,1) node[right]{$\mathcal N_-$};
\draw[color=black,dashed,->] (5.5,0.5)--(5.5,5);
\draw (5,3.7) node[right]{$\mathcal N_=$};
\draw[color=black,dashed,->] (5.5,0.5)--(11,5);
\draw (7.5,1) node[right]{$\mathcal N_+$};
\end{tikzpicture}
\caption{Every linear extension sends $x_j$ to $i_j$ whenever $j\neq \LL$. But  \textcolor{red}{$x_{\LL}$} is sent to one of the locations \textcolor{red}{$i_{\LL}-1,~i_{\LL},~i_{\LL}+1$}, depending on whether the linear extension is in $\mathcal N_-,\mathcal N_=,\mathcal N_+$, respectively.\label{fig:linext}}
\end{figure}
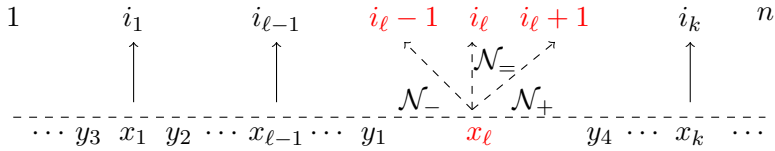

In \cite[Theorem 3.2]{Sta81}, Stanley showed that
\begin{align}
\label{eq:Stanley}
|\mathcal N_{=}|^2\ge |\mathcal N_{-}||\mathcal N_{+}|,
\end{align}
thus resolving a conjecture of Chung, Fishburn and Graham \cite{CFG80}. To see the relation to log-concave sequences consider the case $k=1$ and set 
\begin{align}
\label{eq:Stanleyk=1}
a_i:=|\{\sigma\in\mathcal N:\sigma(x_1)=i\}|,\quad i\in [n].
\end{align}
Then, \eqref{eq:Stanley} amounts to the statement that the sequence $\{a_i\}$ is log-concave. For the general case $k\ge 1$, \eqref{eq:Stanley} is a log-concavity statement about multi-index sequences. 

The goal of this work is to provide a complete characterization of the equality cases of \eqref{eq:Stanley} for any $k$. That is, we will answer the following question: If
\begin{align}
\label{eq:Stanleyeq}
|\mathcal N_{=}|^2=|\mathcal N_{-}||\mathcal N_{+}|,
\end{align}
what can we deduce about the poset $\abar$? 

To gain some intuition for the extremals of Stanley's inequalities \eqref{eq:Stanley} let us start with a trivial observation: If $\{y_1,\ldots,y_{n-k}\}$ are all incomparable to $x_{\LL}$, then $|\mathcal N_{-}|=|\mathcal N_{=}|=|\mathcal N_{+}|$, which yields equality in  \eqref{eq:Stanley}. In the same vein, consider the following example which is slightly less trivial.
\begin{example}
\label{ex:trivmech}
Suppose the poset $\abar$ satisfies
\begin{align}
\label{eq:triv}
\{z\in\abar: z<x_{\LL}\text{ and } z\not < x_{\LL-1}\}\cup \{z\in\abar: z>x_{\LL}\text{ and } z\not > x_{\LL+1}\}=\varnothing.
\end{align}
Then, given any $\sigma\in\cup_{\s\in\{-,=,+\}}\mathcal N_{\s}$, we can permute (some of) the locations of the elements $\{\sigma^{-1}(i_{\LL}-1),\, \sigma^{-1}(i_{\LL}),\, \sigma^{-1}(i_{\LL}+1)\}$ without violating any constraints. For example, given $\sigma\in\mathcal N_+$, the elements $\sigma^{-1}(i_{\LL}-1),\, \sigma^{-1}(i_{\LL})$ must be incomparable to $x_{\LL}=\sigma^{-1}(i_{\LL}+1)$ since, as $i_{\LL-1}+1<i_{\LL}<i_{\LL+1}-1$, the converse would violate \eqref{eq:triv}. Hence, we can exchange the locations of $\{\sigma^{-1}(i_{\LL}-1),\, \sigma^{-1}(i_{\LL}+1)\}$ or $\{\sigma^{-1}(i_{\LL}),\, \sigma^{-1}(i_{\LL}+1)\}$. It follows that
\begin{align}
\label{eq:trivii}
|\mathcal N_=|=|\mathcal N_-|=|\mathcal N_+|,
\end{align}
which in particular implies  \eqref{eq:Stanleyeq}. 
\end{example}
The mechanism \eqref{eq:triv} is wasteful since it is \emph{global} in nature. It controls all the elements between $x_{\LL-1}$ and $x_{\LL+1}$, even though we are concerned only with the elements which are close to $x_{\LL}$ in the sense that they are located in $i_{\LL}-1,\,i_{\LL},\,i_{\LL}+1$. Instead, we expect  \eqref{eq:Stanleyeq} to hold as soon as the mechanism \eqref{eq:triv} occurs only on a \emph{local} scale. To make this idea precise we make the following definition regarding elements that are close to $x_{\LL}$.

\begin{definition}
\label{def:neighbor}
Fix $\LL\in[k]$ such that $i_{\LL-1}+1<i_{\LL}<i_{\LL+1}-1$, and given $\s\in\{-,=,+\}$, fix $\sigma\in \mathcal N_{\s}$. The \emph{companions} of $x_{\LL}=\sigma^{-1}(i_{\LL}+1_{\s})$ are $\sigma^{-1}(i_j)$ for $i_j\in \{i_{\LL}-1,i_{\LL},i_{\LL}+1\}\backslash \{i_{\LL}+1_{\s}\}$, where $1_{\s}:=1_{\{\s\text{ is }+\}}-1_{\{\s\text{ is }-\}}$. The companion lower in ranking is the \emph{lower companion} and the companion higher in ranking is the  \emph{upper companion}.
\end{definition} 
For example, with $\s$ being $-$, the companions of $x_{\LL}=\sigma^{-1}(i_{\LL}-1)$ are $\sigma^{-1}(i_{\LL})$ and $\sigma^{-1}(i_{\LL}+1)$. The lower companion is $\sigma^{-1}(i_{\LL})$ and the upper companion is $\sigma^{-1}(i_{\LL}+1)$.

\subsection{The extremals of Stanley's inequalities}
\label{subsec:extremal_Stanley}

The characterization of the extremals of Stanley's inequalities will be in terms of the companions of $x_{\LL}$ as defined in Definition \ref{def:neighbor}. On a finer resolution, there are two distinct classes of posets which in turn have different types of extremals. The two classes of posets will be called \emph{supercritical} and \emph{critical}, a terminology which will become clear later. The precise definitions are deferred to Definition \ref{def:posetcrit}, but for now, we will simply note that a supercritical poset is always critical, but the converse is false. (There are further classes which reduce to the supercritical and critical classes. They  will be handled  in Section \ref{sec:subcrit}, see also Theorem \ref{thm:N_=>0intro}.)

\begin{theorem}{\textnormal{(\textbf{Supercritical extremals of Stanley's inequalities})}}
\label{thm:supcrit}
\vspace{0.05in}

Suppose the poset $\abar$ is supercritical. The following are equivalent:
\vspace{0.05in}

\begin{enumerate}[(i)]
\item $|\mathcal N_{=}|^2= |\mathcal N_{-}||\mathcal N_{+}|$.\\

\item $|\mathcal N_{-}|= |\mathcal N_{=}|=|\mathcal N_{+}|$.\\

\item  For every linear extension in $\mathcal N_{-}\cup\mathcal N_{=}\cup \mathcal N_{+}$, both companions of $x_{\LL}$ are incomparable to $x_{\LL}$.

\end{enumerate}
\end{theorem}

Theorem \ref{thm:supcrit} provides a number of insights into the extremals of \eqref{eq:Stanleyeq}. Part (ii) of the theorem (which held in \eqref{eq:trivii}) is non-trivial, and even surprising, since it puts heavy constraints on the ways in which $|\mathcal N_{=}|^2= |\mathcal N_{-}||\mathcal N_{+}|$ can occur. A priori, we could have  a geometric progression where $|\mathcal N_{-}|=ab^{c-1},~|\mathcal N_{=}|=ab^c,~|\mathcal N_{+}|=ab^{c+1}$, for some $a,b,c>0$, which would yield the equality 
\[
|\mathcal N_{=}|^2=a^2b^{2c}=(ab^{c-1})(ab^{c+1})=|\mathcal N_{-}||\mathcal N_{+}|.
\]
Theorem \ref{thm:supcrit}(ii) excludes this possibility. On the other hand, despite the information provided by  (ii), it sheds no light on the mechanism which yield equality in \eqref{eq:Stanley}. In contrast, Theorem \ref{thm:supcrit}(iii) provides the mechanism  behind the extremals: The companions of $x_{\LL}$, under any linear extension in $\bigcup_{\circ\in\{-,=,+\}}\mathcal N_{\circ}$, must be incomparable to $x_{\LL}$. Hence, the positions of $x_{\LL}$ and both of its companions can be swapped, which leads to part (ii). Note that (iii) is a \emph{local} condition which controls only the immediate companions of $x_{\LL}$, unlike \eqref{eq:triv}. The power of Theorem \ref{thm:supcrit} lies in the statement that this mechanism is the \emph{only} mechanism behind the extremals of Stanley's inequalities for supercritical posets.

The characterization of  Theorem \ref{thm:supcrit} is very clean and one might hope that it applies to every poset. This hope is quickly shattered:

\begin{example}
\label{ex:crit}
Let $\abar=\{y_1,y_2,y_3,y_4,x_1,x_2,x_3\}$ with the relations 
\[
x_1<x_2<x_3,\quad y_1<x_2,\quad x_2<y_2,\quad x_1<y_3<x_3.
\]
Set $\LL=2,$ and $i_1=2,~i_2=4,~i_3=6$. One can check that $|\mathcal N_-|= |\mathcal N_=|=|\mathcal N_+|=4$ so that Theorem \ref{thm:supcrit}(ii) holds. On the other hand, Theorem \ref{thm:supcrit}(iii) is false since $y_1,y_2$ are comparable to $x_2$ but can appear as companions of $x_2$ under linear extensions in $\mathcal N_-\cup\mathcal N_=\cup\mathcal N_+$. See Figure \ref{fig:excrit}.
\begin{figure}[h]
\begin{tikzpicture}[scale=1.5]
\begin{scope}
  \node at (-2,-1) {$\abar\quad =$};
  \node (y2) at (2,0) {$\color{blue}{y_2}$};
  \node (x3) at (0,0) {$\color{red}{x_3}$};
  \node (y3) at (-1,-1) {$\color{blue}{y_3}$};
  \node (x2) at (1,-1) {$\color{red}{x_2}$};
  \node (y1) at (2,-2) {$\color{blue}{y_1}$};
  \node (x1) at (0,-2) {$\color{red}{x_1}$};
  \node (y4) at (2,-1) {$y_4$};
  \draw [->, thick] (y3) -- (x1);
  \draw [->, thick] (x3) -- (y3);
  \draw [->, thick] (x3) -- (x2);
  \draw [->, thick] (x2) -- (x1);
  \draw  [->, thick] (y2) -- (x2);
  \draw  [->, thick] (x2) -- (y1);
  \end{scope} 
\end{tikzpicture}

\begin{align*}
&\mathcal N_- =\left\{ {\color{blue}y_1} {\color{red}x_1} {\color{red}x_2} {\color{blue}y_3} y_4  {\color{red}x_3} {\color{blue}y_2}, \quad{\color{blue}y_1} {\color{red}x_1} {\color{red}x_2} y_4 {\color{blue}y_3} {\color{red}x_3}  {\color{blue}y_2}, \quad{\color{blue}y_1} {\color{red}x_1} {\color{red}x_2} {\color{blue}y_2}{\color{blue}y_3} {\color{red}x_3} y_4, \quad {\color{blue}y_1} {\color{red}x_1} {\color{red}x_2} {\color{blue}y_3}{\color{blue}y_2} {\color{red}x_3} y_4\right\},\\
&\mathcal N_==\left\{ {\color{blue}y_1}  {\color{red}x_1} {\color{blue}y_3} {\color{red}x_2} y_4 {\color{red}x_3} {\color{blue}y_2}, \quad  {\color{blue}y_1}  {\color{red}x_1}  y_4 {\color{red}x_2} {\color{blue}y_3} {\color{red}x_3}  {\color{blue}y_2},\quad  {\color{blue}y_1}  {\color{red}x_1} {\color{blue}y_3} {\color{red}x_2}  {\color{blue}y_2} {\color{red}x_3} y_4,\quad y_4{\color{red}x_1}  {\color{blue}y_1} {\color{red}x_2}  {\color{blue}y_3} {\color{red}x_3} {\color{blue}y_2}\right\},\\
&\mathcal N_+=\left\{ {\color{blue}y_1} {\color{red}x_1} y_4 {\color{blue}y_3} {\color{red}x_2} {\color{red}x_3} {\color{blue}y_2},\quad  {\color{blue}y_1} {\color{red}x_1} {\color{blue}y_3} y_4{\color{red}x_2} {\color{red}x_3} {\color{blue}y_2},\quad y_4 {\color{red}x_1}   {\color{blue}y_1} {\color{blue}y_3} {\color{red}x_2}  {\color{red}x_3} {\color{blue}y_2},\quad y_4 {\color{red}x_1} {\color{blue}y_3} {\color{blue}y_1}  {\color{red}x_2} {\color{red}x_3} {\color{blue}y_2}\right\}.
  \end{align*}
 \caption{\textbf{Top}: Hasse diagram (arrows point from smaller to larger elements) of poset in Example \ref{ex:crit}. \textbf{Bottom}:  Collections of linear extensions of poset in Example \ref{ex:crit}.\label{fig:excrit}}
\end{figure}
\end{example}

Our next result goes beyond Theorem \ref{thm:supcrit} and characterizes the extremals of critical posets.
\begin{theorem}{\textnormal{(\textbf{Critical extremals of Stanley's inequalities})}}
\label{thm:crit}
\vspace{0.05in}

Suppose the poset $\abar$ is critical. The following are equivalent:
\vspace{0.05in}

\begin{enumerate}[(i)]
\item $|\mathcal N_{=}|^2= |\mathcal N_{-}||\mathcal N_{+}|$.\\

\item $|\mathcal N_{-}|= |\mathcal N_{=}|=|\mathcal N_{+}|$.\\

\item  For every linear extension in $\mathcal N_{-}\cup\mathcal N_{=}\cup \mathcal N_{+}$, at least one companion of $x_{\LL}$ is incomparable to $x_{\LL}$. In addition, there exist nonnegative numbers $\mathrm N_1, \mathrm N_2$ such that:
\begin{itemize}
\item For any fixed $\s \in \{-,=,+\}$, 
\begin{align*}
&|\{\sigma\in\mathcal N_{\circ}: \textnormal{only the lower companion of $x_{\LL}$ is incomparable to $x_{\LL}$} \}|\\
&=\mathrm N_1=\\
&|\{\sigma\in\mathcal N_{\circ}: \textnormal{only the upper companion of $x_{\LL}$ is incomparable to $x_{\LL}$} \}|.
\end{align*}
\item $|\{\sigma\in\mathcal N_{\circ}: \textnormal{both companions of $x_{\LL}$ are incomparable to $x_{\LL}$} \}|=\mathrm N_2\quad \forall ~ \s \in \{-,=,+\}$.
\end{itemize}
\end{enumerate}
\end{theorem}

Let us compare and contrast Theorem \ref{thm:supcrit} and Theorem \ref{thm:crit}. The conclusion in part (ii) that the equality \eqref{eq:Stanleyeq} necessitates $|\mathcal N_{-}|= |\mathcal N_{=}|=|\mathcal N_{+}|$  remains true for supercritical and critical posets. But the mechanisms, i.e., part (iii), for this phenomenon are different. Clearly, Theorem \ref{thm:supcrit}(iii) is a stronger condition since it trivially implies the condition in Theorem \ref{thm:crit}(iii). For critical posets, the conclusion that only 0 comparable companions are allowed (namely  Theorem \ref{thm:supcrit}(iii)) is relaxed into the statement that 0 or 1 comparable companions are allowed. But in order to get $|\mathcal N_{-}|= |\mathcal N_{=}|=|\mathcal N_{+}|$, there must be a balance between between those linear extensions with 1 comparable companion, which is the content of the second part of Theorem \ref{thm:crit}(iii).

Our formulation of Theorem \ref{thm:supcrit} and Theorem \ref{thm:crit} mirrors the analogous distinction in convex geometry between supercritical and critical (cf. Theorem \ref{thm:SvHsupcritIntro}). However, our proofs provide us with a stronger statement which encompass both Theorem \ref{thm:supcrit} and Theorem \ref{thm:crit}. 

\begin{theorem}{\textnormal{(\textbf{Extremals of Stanley's inequalities})}}
\label{thm:N_=>0intro}
Suppose $\abar$ is a poset such that $|\mathcal N_{=}|>0$. Then, the following hold:\\

\begin{itemize}
\item  The conclusions of Theorem \ref{thm:crit} remain true. In addition, given any $\sigma\in\mathcal N_{\s}$, for any $\s\in\{-,=,+\}$, where one of the companions is comparable to $x_{\LL}$, we have that the lower and upper companions are incomparable to each other.\\

\item If  $\abar$ is supercritical then the conclusions of  Theorem \ref{thm:supcrit} remain true. 
\end{itemize}
\end{theorem}

Theorem \ref{thm:N_=>0intro} improves upon  Theorem \ref{thm:supcrit} and Theorem \ref{thm:crit} by showing that the conclusions of Theorem \ref{thm:crit} hold even under the assumption $|\mathcal N_{=}|>0$. In addition, Theorem \ref{thm:N_=>0intro} provides further information on the structure of the linear extensions. The only case not covered by Theorem \ref{thm:N_=>0intro} is when $|\mathcal N_{=}|=0$, which is in fact trivial and will be characterized later (Theorem \ref{thm:trivsubcrit}). 

\begin{remark}{\textnormal{(\textbf{Poset characterization})}}
\label{rem:posetchar}
There is a way to reformulate Theorem \ref{thm:supcrit}(iii) so that the characterization of the extremals is given in terms of conditions on the poset itself rather than on the set of its linear extensions:
\begin{align}
\label{eq:posetchar}
\begin{split}
&\forall\,y<x_{\LL}: ~~\exists~ s(y)\in\{ 0,\ldots, k+1\} \text{ s.t. } y<x_{s(y)}\text{ and }| \{z\in\abar:  y<z<x_{s(y)}\}|>i_{s(y)}-i_{\LL},\\
&\forall\,y>x_{\LL}:~~\exists\, r(y)\in\{ 0,\ldots, k+1\} \text{ s.t. } y>x_{r(y)}\text{ and }|\{z\in\abar:  x_{r(y)}<z<y\}||>i_{\LL}-i_{r(y)};
\end{split}
\end{align}
see Proposition \ref{prop:IIIimpliesIV}. Here, $x_0$ (res. $x_{k+1}$) is the added element with the property that it is smaller (res. bigger) than any other element in $\abar$. The formulation \eqref{eq:posetchar} can be useful in practice since, given a standard description of a poset, \eqref{eq:posetchar} is easier to check. On the other hand, the formulation of Theorem \ref{thm:supcrit}(iii) is more compatible with our dictionary, which is more natural to formulate in terms of conditions on the linear extensions of the poset. In the first version of this manuscript we wrote that ``It is an interesting problem to find an analogue of \eqref{eq:posetchar} for critical posets." However, since the first version of our work was made public, Chan and Pak \cite[Theorem 1.3]{chan2023equality} proved a remarkable result on the computational complexity of the characterization of Stanley's inequalities, which in particular implies that a poset characterization of  Stanley's inequalities of the form \eqref{eq:posetchar} would contradict fundamental conjectures in computational complexity \cite[\S 3.5]{chan2023equality}. 
\end{remark}

\begin{remark}{\textnormal{(\textbf{$k=1$})}}
\label{rem:k=1}
The characterization of the extremals of Stanley's inequalities when $k=1$ was  done in \cite[\S 15]{SvH20}. It turns out that, when $k=1$, the poset must be supercritical and the characterization of \cite[\S 15]{SvH20} in this case is the same as Theorem \ref{thm:supcrit} and Remark \ref{rem:posetchar}. While our proofs take much inspiration from the work \cite{SvH20}, the new phenomena of critical posets necessitated the development of many new ideas (see Figure \ref{fig:dictionary}). For example, the dictionary constructed in  \cite[\S 15]{SvH20} was in terms of the poset itself (as in Remark \ref{rem:posetchar}), rather than its linear extensions. But when progressing to critical posets, the approach of \cite[\S 15]{SvH20} no longer works (especially in light of  \cite[Theorem 1.3]{chan2023equality}),  while our dictionary, which is in terms of a  linear extensions description, is suitable for these more subtle and rich extremals. 

Let us also mention that, when $k=1$, Chan and Pak, using their  \emph{combinatorial atlas} method \cite{CP21}, provided a \emph{linear-algebraic} proof of Stanley's inequalities and characterized their extremals, thus avoiding any use of convex geometry; see also the proof for width two posets by Chan, Pak, and Panova \cite{chan2021extensions}. However, their approach does not currently extend to the case $k>1$. 
\end{remark}

\begin{remark}{\textnormal{(\textbf{$k=2$})}}
\label{rem:k=2}
Using our techniques, Chan and Pak \cite[Lemma 9.1]{chan2023equality} showed that, in fact, the conclusion of Theorem \ref{thm:supcrit} remains true whenever $k=2$. Per Remark \ref{rem:k=1}, the same holds true for $k=1$. It follows that Example \ref{ex:crit}, where the conclusion of Theorem \ref{thm:supcrit} is no longer valid, is sharp in terms of $k$. 
\end{remark}

\subsection{Dictionaries between convex geometry and combinatorics}
\label{subsec:AFStanleyInq}
Stanley's proof of \eqref{eq:Stanley} relies on a remarkable correspondence that he found between mixed volumes of certain convex polytopes and linear extensions counts. Once this correspondence is established, the inequality \eqref{eq:Stanley} follows from a deep log-concavity result in convex geometry: The Alexandrov-Fenchel inequality. We will start this section by reviewing Stanley's proof of the inequality \eqref{eq:Stanley}, and then move to the discussion of its extremals.

\subsubsection{The Alexandrov-Fenchel inequality}
 We start with some preliminaries from convex geometry; our standard reference is \cite{Sch14}. Given convex bodies (non-empty compact convex sets)  $C,C'\subseteq \R^{n-k}$ and scalars $\lambda ,\lambda'\ge 0$, we define their sum as
\[
\lambda C+\lambda'C':=\{\lambda x+\lambda' y:x\in C, y\in C'\}.
\]
The volume of a sum of convex bodies behaves as a polynomial: Given a positive integer $p$, convex bodies $C_1,\ldots,C_p\subseteq\R^{n-k}$, and scalars  $\lambda_1,\ldots,\lambda_p\ge 0$, we have 
\[
\Vol_{n-k}(\lambda_1C_1+\cdots+\lambda_pC_p)=\sum_{ 1\le j_1,\ldots,j_{n-k}\le p}\V_{n-k}(C_{j_1},\ldots,C_{j_{n-k}})\lambda_{j_1}\cdots\lambda_{j_{n-k}},
\]
where the coefficients $\V_{n-k}(C_{j_1},\ldots,C_{j_{n-k}})$ are called \emph{mixed volumes}. These geometric objects generalize the notions of volume, surface area, mean width, etc. The \emph{Alexandrov-Fenchel inequality} \cite[\S 7.3]{Sch14} states that sequences of mixed volumes are log-concave: For any convex bodies $C_1,\ldots,C_{n-k}\subset \R^{n-k}$,
\begin{equation}
\label{eq:AFintro}
\V_{n-k}(C_1,C_2,C_3,\ldots,C_{n-k})^2\ge \V_{n-k}(C_1,C_1,C_3,\ldots,C_{n-k})\V_{n-k}(C_2,C_2,C_3,\ldots,C_{n-k}). 
\end{equation}
Stanley's proof of \eqref{eq:Stanley} relies on the identification of the poset $\abar$ with polytopes $\poly_0,\ldots,\poly_k$. We defer the explicit construction of these polytopes for later (Section \ref{sec:preliminaries}), and for now denote by $\mathcal\poly$ a certain collection of these polytopes containing $n-k-2$ of them. The key point are the identities
\begin{equation}
\label{eq:posetvolrepintro}
\begin{split}
&|\mathcal N_-|=(n-k)!\,\V_{n-k}(K_{\LL},K_{\LL},\mathcal \poly),\\
&|\mathcal N_=|=(n-k)!\,\V_{n-k}(K_{\LL-1},K_{\LL},\mathcal \poly),\\
&|\mathcal N_+|=(n-k)!\,\V_{n-k}(K_{\LL-1},K_{\LL-1},\mathcal \poly).
\end{split}
\end{equation}
With the representation \eqref{eq:posetvolrepintro} in hand, the inequality \eqref{eq:Stanley} is equivalent to
\begin{equation}
\label{eq:AFStanleyintro}
\V_{n-k}(K_{\LL-1},K_{\LL},\mathcal \poly)^2\ge \V_{n-k}(K_{\LL},K_{\LL},\mathcal \poly)\V_{n-k}(K_{\LL-1},K_{\LL-1},\mathcal \poly),
\end{equation}
which follows immediately from \eqref{eq:AFintro}. 

Stanely's proof of \eqref{eq:Stanley} is the only proof currently known. Hence, a natural route towards the characterization of the extremals of Stanley's inequalities would require:
\begin{itemize}
\item Characterization of the extremals of the Alexandrov-Fenchel inequality.
\item Dictionary between the extremals of the Alexandrov-Fenchel inequality and the extremals of Stanley's inequalities. 
\end{itemize}
For arbitrary convex bodies, the characterization of the extremals of \eqref{eq:AFintro} is a long-standing open problem \cite[\S 7.6]{Sch14}. But when the bodies are \emph{polytopes}, this problem was recently solved by the second-named author and Van Handel \cite{SvH20}. Thus, the work \cite{SvH20} takes care of the first item and our work here is dedicated to the second item.

To build intuition regarding the correspondence between the extremal structures of posets and polytopes, let us revisit Example \ref{ex:trivmech}. As will be evident  (see \eqref{eq:Ki}), the identity \eqref{eq:triv} holds if, and only if, $\poly_{\LL-1}=\poly_{\LL}$.  In this case it is clear that equality will be attained in \eqref{eq:AFStanleyintro}. But as we saw in Theorem \ref{thm:supcrit} and Theorem \ref{thm:crit}, equality can be attained in Stanley's inequalities under much weaker conditions than those captured by Example \ref{ex:trivmech}. It follows that  equality holds in \eqref{eq:AFStanleyintro} under conditions which are much weaker than $\poly_{\LL-1}=\poly_{\LL}$. The characterization of these conditions is the topic of the next section.

\subsubsection{The extremals of the Alexandrov-Fenchel inequality for convex polytopes}
The terminology of supercritical and critical posets comes in fact from the analogous terminology in the characterization of the extremals of the Alexandrov-Fenchel inequality for convex polytopes as introduced in \cite{SvH20}----the precise definitions of supercriticality and criticality is deferred to Definition \ref{def:ssc}. In the sequel, $B\subseteq \R^{n-k}$ always stands for the unit ball, and the notions of \emph{$(B,\mathcal \poly)$-extreme normal directions} and \emph{$\mathcal \poly$-degenerate pairs}, which will be used in the subsequent theorem, will be given in Definition \ref{def:extreme} and Definition \ref{def:deg}, respectively.

\begin{theorem}{\textnormal{(\textbf{Extremals of the Alexandrov-Fenchel inequality for convex polytopes}, \cite{SvH20})}}
\label{thm:SvHsupcritIntro}
\vspace{0.1in}
\begin{itemize}
\item Suppose $\mathcal \poly$ is supercritical. Then,
\[
\V_{n-k}(K_{\LL-1},K_{\LL},\mathcal \poly)^2=\V_{n-k}(K_{\LL},K_{\LL},\mathcal \poly)\V_{n-k}(K_{\LL-1},K_{\LL-1},\mathcal \poly),
\]
if, and only if, up to dilation and translation, the supporting hyperplanes of $\poly_{\LL-1}$ and $\poly_{\LL}$ agree in all $(B,\mathcal \poly)$-extreme normal directions.\\

\item Suppose $\mathcal \poly$ is critical. Then, 
\[
\V_{n-k}(K_{\LL-1},K_{\LL},\mathcal \poly)^2=\V_{n-k}(K_{\LL},K_{\LL},\mathcal \poly)\V_{n-k}(K_{\LL-1},K_{\LL-1},\mathcal \poly),
\]
if, and only if, there exist $0\le d<\infty$ $\mathcal \poly$-degenerate pairs $(P_1,Q_1),\ldots,(P_d,Q_d)$, such that, up to dilation and translation, the supporting hyperplanes of $\poly_{\LL-1}+\sum_{j=1}^dQ_j$ and $\poly_{\LL}+\sum_{j=1}^dP_j$ agree in all $(B,\mathcal \poly)$-extreme normal directions. 
\end{itemize}
\end{theorem}
The complicated structure of the $(B,\mathcal \poly)$-extreme normal directions (see Figure \ref{fig:graph}) is what gives rise to the richness of the extremals. If the supporting hyperplanes of $\poly_{\LL-1}$ and $\poly_{\LL}$ agree in \emph{every direction} on the sphere $S^{n-k-1}$, then, up to dilation and translation, $\poly_{\LL-1}$ and $\poly_{\LL}$ are identical. This is an example where a \emph{global} mechanism (supporting hyperplanes of $\poly_{\LL-1},\poly_{\LL}$ agree everywhere)  gives rise to equality in \eqref{eq:AFintro}. Theorem \ref{thm:SvHsupcritIntro} provides a \emph{local} mechanism for equality in \eqref{eq:AFintro} (supporting hyperplanes of $\poly_{\LL-1},\poly_{\LL}$ agree only in very few directions), and furthermore, establishes that this local mechanism is the \emph{only} mechanism for the extremal structures of the Alexandrov-Fenchel inequality. 

\subsubsection{Dictionary for extremals} A priori, it is not at all clear that the complications and richness of the extremals  of \eqref{eq:AFintro} would also arise in our very specific family of polytopes. Indeed, in the case $k=1$, only the supercritical extremals appear. Remarkably, not only does  this complexity arise,  but we can provide a clean and intuitive characterization of the extremals arising in Stanley's inequalities for critical posets. At the core of our work is a powerful dictionary which translates between the \emph{extremal} properties of convex polytopes and partially ordered sets. We discover \emph{new} extreme normal directions, and in addition, introduce numerous new key ideas: \emph{closure}, \emph{splitting pairs}, \emph{mixing}, \emph{critical subposet}, to name just a few. It will be best to introduce these ideas at the appropriate places in the paper; Section \ref{sec:outline} will contain a brief outline of our proof. We refer to Figure \ref{fig:dictionary} for a quick summary of the main components in our dictionary, and recommend that the reader revisit this table from time to time.

\begin{figure}[h]
\begin{center}
    \begin{tabular}{ | c | c | c |}
      \hline
      \thead{Geometry} & \thead{Dictionary} & \thead{Combinatorics} \\
      \hline
        \makecell{Criticality of polytopes \\ (Definition \ref{def:ssc})}  &  \makecell{ Section \ref{sec:notionsCrit} \\
        (Proposition \ref{prop:critequiv})}  & \makecell{Criticality of posets \\ (Definition \ref{def:posetcrit})}  \\
      \hline
             \makecell{Projection\\ (\cite[Theorem 5.3.1]{Sch14})}  & \makecell{ Section \ref{sec:subcrit} \\
        (Remark \ref{rem:splitproj})}  &   \makecell{Splitting \\ (Definition \ref{def:split})}\\
      \hline
 \makecell{Criticality of splitting pairs\\ (Definition \ref{def:splitpaircrit})}    &  Section \ref{sec:beyond} &  \makecell{Mixing of splitting pairs\\ (Figure \ref{fig:orgsec})} \\
      \hline

         \makecell{Maximal collection of polytopes\\ (\cite[section 9.1]{SvH20})}   &  \makecell{Section \ref{sec:beyond}\\
             (Proposition \ref{prop:max_notions})} &        
        \makecell{Maximal splitting pair\\ (Definition \ref  {def:rmaxsmin})} \\
      \hline

       \makecell{Extreme normal directions}  &  \makecell{Section \ref{sec:ext}} & \makecell{First- and second-neighbors}  \\
      \hline
            \makecell{Translation and dilation}  &  \makecell{Sections \ref{sec:supercrit}-\ref{sec:crit}} & \makecell{Chains of poset}  \\
      \hline
                  \makecell{Critical subspace\\
                  (Equation \eqref{eq:Eperp})}  &  \makecell{Section \ref{sec:crit}} & \makecell{Critical subposet\\
                           (Equation \eqref{eq:Eperp})}  \\
      \hline
    \end{tabular}
  \end{center}
  \caption{Dictionary between geometry of polytopes and combinatorics of posets.\label{fig:dictionary}}
  \end{figure}

\subsection{Organization of paper}
We start in Section \ref{sec:preliminaries} by reviewing the connection between partially ordered sets and convex geometry. In Section \ref{sec:found} we develop a number of tools (\emph{decompositions, closure}) that are used throughout the paper and also prove the sufficiency parts of Theorem \ref{thm:supcrit} and Theorem \ref{thm:crit}. Section \ref{sec:outline} provides a brief outline of the proofs of the main results.  Section \ref{sec:notionsCrit} sets the first building block of our dictionary by showing the equivalences between notions of criticality for posets and polytopes.  Section \ref{sec:subcrit} introduces the idea of \emph{splitting} and characterizes the extremals of the \emph{subcritical} posets. Section \ref{sec:beyond} introduces the idea of \emph{mixing} which is at the heart of our proofs and applies it to \emph{splitting pairs}. In Section \ref{sec:ext} we  add to our dictionary the combinatorial characterization of the extreme normal directions. We complete the proof of Theorem \ref{thm:supcrit}  in Section \ref{sec:supercrit} and the proofs of Theorem \ref{thm:crit} and Theorem \ref{thm:N_=>0intro}  in Section \ref{sec:crit}. At the end of the paper we include a Notation Appendix for the convenience of the reader.  

\section{Preliminaries} 
\label{sec:preliminaries}
In this section we review some basics about posets and convex geometry, as well as introduce the notation we use throughout the paper. We review the connection between posets and mixed volumes, and state the characterization of the extremals of the Alexandrov-Fenchel inequality for (convex) polytopes. In addition, we provide the criticality definitions for polytopes and posets. 

We use the notation $\le,<,=,\ge,>,\sim$ to describe the relations in a poset, where $\sim$ stands for the comparability relation\footnote{Note that $\sim$ is not a transitive property.}, and by $\not\le,\not<,\not=,\not\ge,\not>,\incomp$ to describe their negations. Given integers $p\le q$ we write
 \begin{equation}
 \label{eq:bracket}
\llbracket p,q \rrbracket :=\{p,p+1,\ldots,q-1,q\}.
\end{equation}
Fix positive integers $k,n$, with $k\le n$, and consider the poset $\abar$, of size $n$,
\[
\abar=\{y_1,\ldots,y_{n-k},x_1,\ldots, x_k\},
\]
where $x_1<x_2<\cdots<x_k$ is a chain. Let 
\[
\aint=\{y_1,\ldots,y_{n-k}\}
\]
be the induced poset of size $n-k$ obtained from $\abar$ by removing the chain. To simplify the notation we add two elements $x_0,x_{k+1}$ to $\abar$ with the property that $x_0$ is smaller than any element in $\abar$ while $x_{k+1}$ is bigger than any element in $\abar$. Note that this allows us to consider the case $k=0$.

Let $\mathcal N$ be the set of all linear extensions of $\abar$, that is,
\[
\mathcal N=\{\text{bijections }\sigma:\abar\to [n]: w\le z\Rightarrow \sigma(w)\le \sigma(z)~\forall ~w,z\in\abar\},
\]
with the convention that $\sigma(x_0)=0$ and $\sigma(x_{k+1})=n+1$ for any $\sigma\in\mathcal N$.
Fix $\LL\in [k]:=\{1,\ldots, k\}$ and fix $i_1<i_2<\cdots< i_k\in [n]$, with the property $i_{\LL-1}+1<i_{\LL}<i_{\LL+1}-1$, and let $i_0:=0,~ i_{k+1}:=n+1$. We define the following sets of linear extensions, $\mathcal N_{-},\mathcal N_{=}, \mathcal N_{+}\subseteq\mathcal N$,
\begin{align*}
&\mathcal N_{-}:=\{\sigma\in\mathcal N: \sigma(x_{\LL})=i_{\LL}-1 \quad\text{and}\quad \sigma(x_m)=i_m~\forall ~m\in [k]\backslash \{\LL\}\},\\
&\mathcal N_{=}:=\{\sigma\in\mathcal N: \sigma(x_{\LL})=i_{\LL} \quad\text{and}\quad \sigma(x_m)=i_m~\forall ~m\in [k]\backslash \{\LL\}\},\\
&\mathcal N_{+}:=\{\sigma\in\mathcal N: \sigma(x_{\LL})=i_{\LL}+1 \quad\text{and}\quad \sigma(x_m)=i_m~\forall ~m\in [k]\backslash \{\LL\}\},
\end{align*}
so Stanley's inequalities read
\begin{align}
\label{eq:eqmathcalN}
|\mathcal N_{=}|^2\ge |\mathcal N_{-}||\mathcal N_{+}|.
\end{align}

\subsection{Posets and polytopes}
Fundamental to our approach towards the extremals of \eqref{eq:eqmathcalN} is the connection, due to Stanley \cite{Sta81}, between posets and convex polytopes. We start with the definition of an \emph{order polytope}: Given $\bb\subseteq\aint$ we let $\R^{\bb}:=\{t\in \R^{n-k}:t_j=0\mbox{ for } y_j\notin\beta\}$ and define the order polytope $O_{\bb}\subseteq\R^{\bb}\subseteq \R^{\aint}$ by
\[
O_{\bb}:=\{t\in \R^{\bb}: t_j\in [0,1] ~\forall\, y_j\in\bsg, \mbox{ and } t_u\le t_v\mbox{ if }y_u\le y_v~\forall\, y_u,y_v\in\bsg\}. 
\]
The order polytope encodes important properties of the poset, e.g., the volume of $O_{\aint}$ is proportional to the number of linear extensions of $\aint$ \cite[Corollary 4.2]{Sta86}. Let us recall some basic facts about order polytopes, which will require the following poset notions. A \emph{maximal} (res. \emph{minimal}) element $y\in\aint$ is such that there exists no $z\in\aint$, different than $y$, satisfying $y<z$ (res. $z<y$). Given a set $\bb\subseteq\aint$ we define $\bb^{\uparrow}$ (res. $\bb^{\downarrow}$) to be the set of maximal (res. minimal) elements of $\bb$. Given a relation $\star\in\{\le,<,=,\ge,>,\sim,\not\le,\not<,\not=,\not\ge,\not>,\incomp\}$ and $y\in\bb$ we let
\[
\bb_{\star y}:=\{z\in\bb: z\star y\},
\]
and, similarly, given relations $\star,\ast\in\{\le,<,=,\ge,>,\sim,\not\le,\not<,\not=,\not\ge,\not>,\incomp\}$, and $y,y'\in\bb$, we write
\[
\bb_{\star y,\ast y'}:=\{z\in\bb: z\star y\mbox{ and }z\ast y'\}.
\]
An element $z\in\bb$ \emph{covers} $y\in\bb$ if $z\in \bb_{>y}^{\downarrow}$. We say that $\bb$ is an \emph{upper set} (res. \emph{lower set}) if $\aint_{>y}\subseteq\bb$ (res. $\aint_{<y}\subseteq\bb$), for every $y\in\bb$. 

The next result provides information about the face structure of order polytopes based on the poset notions just introduced.
\begin{lemma}{\textnormal{(\cite[\S 1]{Sta86})}}
\label{lem:dimOb}
For any $\bb\subseteq\aint$ we have $\dim O_{\bb}=|\bb|$. The \textnormal{($|\bb|-1$)}-dimensional faces of $O_{\bb}$ are precisely the following subsets of $O_{\bb}$:
\begin{enumerate}[(i)]
\item $O_{\bb}\cap\{t_j=0\}$ for $y_j\in\bb^{\downarrow}$.
\item $O_{\bb}\cap\{t_j=1\}$ for $y_j\in\bb^{\uparrow}$.
\item $O_{\bb}\cap\{t_u=t_v\}$ for $y_u,y_v\in\bb$ such that $y_v$ covers $y_u$ in $\bb$. 
\end{enumerate}
\end{lemma}
Hyperplane sections of order polytopes will play a crucial role for us: Given $i\in \llbracket 0, k\rrbracket$, define the polytopes in $\R^{n-k}$,
\begin{align}
\label{eq:Ki}
&\poly_i:=\{t\in O_{\aint}:t_j=0\mbox{ if }y_j<x_i,~ t_j=1\mbox{ if }y_j>x_{i+1},\mbox{ for all }y_j\in\aint\}.
\end{align}
While we defined the polytopes $\{\poly_i\}$ as hyperplane sections of order polytopes, they are in fact nothing but translations of certain order polytopes. To see this relation we start with the next lemma whose proof is a matter of checking the definitions. In the sequel, given $\bb\subseteq \aint$ let $1_{\bb}:=\sum_{y_j\in\bb}e_j$, with $\{e_j\}_{j\in\bb}$ denoting the standard basis of $\R^{\bb}$.
\begin{lemma}
\label{lem:Onot}
Let $\bb,\bb'\subseteq\aint$ be disjoint sets where $\bb$ is an upper set and $\bb'$ is a lower set. Then,
\[
O_{\aint\backslash(\bb\cup\bb')}+1_{\bb}=\{t\in O_{\aint}:t_j=0 \mbox{ if }y_j\in\bb' \mbox{ and } t_j=1\mbox{ if }y_j\in \bb\},
\]
where we view $\R^{\aint\backslash(\bb\cup\bb')}$ as a subset of $\R^{\aint}	\cong\R^{n-k}$. 
\end{lemma}
We can now write $\{\poly_i\}_{i\in\llbracket 0,k\rrbracket}$ as translates of order polytopes. For $i\in\llbracket 0,k\rrbracket$ define
\begin{align}
\label{eq:betai}
\bsg_i:=\aint\backslash(\aint_{<x_i}\cup \aint_{>x_{i+1}}),
\end{align}
with the convention that $\bsg_i=\varnothing$ if $i<0$ or $i>k$; for $S\subseteq \llbracket 0,k\rrbracket$ set $\bsg_S:=\cup_{i\in S}\bsg_i$. The interpretation of $\bsg_i$ is as the set of elements which can potentially be ordered between $x_i$ and $x_{i+1}$. Then, applying Lemma \ref{lem:Onot}, with the disjoint upper and lower sets $\beta=\aint_{>x_{i+1}},~ \beta'=\aint_{<x_i}$, shows that
\begin{equation}
\label{eq:Kiordpoly}
\poly_i=O_{\bsg_i}+1_{\aint_{>x_{i+1}}}\mbox{ for }i\in\llbracket 0,k\rrbracket.
\end{equation} 
As an example of $\bsg_S$, which will be useful later,  the following result handles the set $S:=\llbracket 0, r\rrbracket\cup \llbracket s, k\rrbracket$.
\begin{lemma}
\label{lem:betaintv}
For any $r\le s$,
\[
 \bsg_{\llbracket 0, r\rrbracket\cup \llbracket s, k\rrbracket}=\aint\backslash\aint_{>x_{r+1},<x_s}=\bsg_r\cup\bsg_s\cup\aint_{<x_{r+1}}\cup \aint_{>{x_s}}.
 \]
\end{lemma}
\begin{proof}
The second identity is clear so we focus on the first identity. Let $j_0:=-1,~ 0\le j_1<\cdots<j_p\le k,~j_{p+1}:=k+1$. We claim that
\begin{equation}
\label{eq:capcup}
\bigcap_{q=1}^p(\aint_{<x_{j_q}}\cup \aint_{>x_{j_q+1}})=\bigcup_{q=0}^p\aint_{>x_{j_q+1},<x_{j_{(q+1)}}}.
\end{equation}
$\subseteq$: Let $y\in \bigcap_{q=1}^p(\aint_{<x_{j_q}}\cup \aint_{>x_{j_q+1}})$ so that, for each $q\in \llbracket 1,p\rrbracket$, either $y<x_{j_q}$ or $y>x_{j_q+1}$. Let $q'$ be the largest $q$ such that $y>x_{j_q+1}$. Then, $y$ is not bigger than $x_{j_{(q'+1)}+1}$, which means that $y<x_{j_{(q'+1)}}$, as $y\in \aint_{<x_{j_{(q'+1)}}}\cup \aint_{>x_{j_{(q'+1)}}+1}$ (this is trivially true if $q'=p$). Hence, $y\in \aint_{>x_{j_{q'}+1},<x_{j_{(q'+1)}}}$.

$\supseteq$: Let $y\in\aint_{>x_{j_q+1},<x_{j_{(q+1)}}}$ for some $q\in \llbracket 0,p\rrbracket$. Then, for any $q'\le q$, $y>x_{j_{q'}+1}$ and,  for any $q'> q$, $y<x_{j_{q'}}$. Hence, $y\in \bigcap_{q=0}^{p+1}(\aint_{<x_{j_{q}}}\cup\aint_{>x_{j_{q}+1}})= \bigcap_{q=1}^p(\aint_{<x_{j_{q}}}\cup\aint_{>x_{j_{q}+1}})$.\\

We now to turn to the proof of the lemma. Let $j_0:=-1, j_{p+1}:=k+1$, and $\{j_1,\ldots,j_p\}:=\{0,\ldots,r,s,\ldots,k\}$. We have 
\begin{align*}
\bsg_{\llbracket 0, r\rrbracket\cup \llbracket s, k\rrbracket}&=\bigcup_{q=1}^p\beta_{j_q}=\bigcup_{q=1}^p\left(\aint\bigg \backslash\left(\aint_{<x_{j_q}}\cup \aint_{>x_{j_q+1}}\right)\right)=\aint\bigg \backslash\bigcap_{q=1}^p\left(\aint_{<x_{j_q}}\cup \aint_{>x_{j_q+1}}\right)\\
&\underset{\eqref{eq:capcup}}{=}\aint\bigg \backslash \bigcup_{q=0}^p\aint_{>x_{j_q+1},<x_{j_{(q+1)}}}.
\end{align*}
Whenever $j_q\neq r$, $j_q+1=j_{q+1}$, so $\aint_{>x_{j_q+1},<x_{j_{(q+1)}}}=\varnothing$. It follows that $\bigcup_{q=0}^p\aint_{>x_{j_q+1},<x_{j_{(q+1)}}}=\aint_{>x_{r+1},<x_{s}}$, which completes the proof.
\end{proof}

\subsection{Posets and mixed volumes}
The connection between the polytopes $\{K_i\}_{i\in\llbracket 0,k\rrbracket}$ and \newline $|\mathcal N_-|, |\mathcal N_=|, |\mathcal N_+|$, which leads to Stanley's proof of \eqref{eq:eqmathcalN}, goes through the notion of mixed volumes; we refer to \cite{Sch14} as the standard reference for the theory of convex bodies. Given convex bodies (nonempty compact convex sets) $C,C'\subseteq \R^{n-k}$, and scalars $\lambda,\lambda'\ge 0$, we define their sum as
\[
\lambda C+\lambda' C':=\{\lambda x+\lambda' y:x\in C, y\in C'\}.
\]
The volume of a sum of convex bodies behaves as a polynomial: Given convex bodies $C_1,\ldots,C_p\subseteq\R^{n-k}$, and scalars  $\lambda_1,\ldots,\lambda_p\ge 0$, we have \cite[Theorem 5.1.7]{Sch14},
\[
\Vol_{n-k}(\lambda_1C_1+\cdots+\lambda_pC_p)=\sum_{ j_1,\ldots,j_{n-k}\in\llbracket 1, p\rrbracket}\V_{n-k}(C_{j_1},\ldots,C_{j_{n-k}})\lambda_{j_1}\cdots\lambda_{j_{n-k}}.
\]
The coefficients $\V_{n-k}(C_{j_1},\ldots,C_{j_{n-k}})$, which are nonnegative, symmetric, and multilinear in their in their arguments,  are called \emph{mixed volumes}. Stanley's proof of \eqref{eq:eqmathcalN} relies on the following identification of $|\mathcal N_-|,|\mathcal N_=|,|\mathcal N_+|$ with mixed volumes \cite[Theorem 3.2]{Sta81}. For $m\in\llbracket 0,k\rrbracket$ let
\[
\mathcal \poly_m:=(\underbrace{\poly_m,\ldots,\poly_m}_{i_{m+1}-i_m-1}).
\]
Then,
\begin{align*}
&|\mathcal N_-|=(n-k)!\V_{n-k}(\mathcal\poly_0,\mathcal\poly_1,\ldots,\underbrace{\poly_{\LL-1},\ldots,\poly_{\LL-1}}_{i_\LL-1-i_{\LL-1}-1},\underbrace{\poly_\LL,\ldots,\poly_\LL}_{i_{\LL+1}-(i_\LL-1)-1},\mathcal\poly_{\LL+1},\ldots ,\mathcal\poly_k),\\
&|\mathcal N_=|=(n-k)!\V_{n-k}(\mathcal\poly_0,\mathcal\poly_1,\ldots,\underbrace{\poly_{\LL-1},\ldots,\poly_{\LL-1}}_{i_{\LL}-i_{\LL-1}-1},\underbrace{\poly_{\LL},\ldots,\poly_{\LL}}_{i_{\LL+1}-i_{\LL}-1},\mathcal\poly_{\LL+1},\ldots ,\mathcal\poly_k),\\
&|\mathcal N_+|=(n-k)!\V_{n-k}(\mathcal\poly_0,\mathcal\poly_1,\ldots,\underbrace{\poly_{\LL-1},\ldots,\poly_{\LL-1}}_{i_{\LL}+1-i_{\LL-1}-1},\underbrace{\poly_{\LL},\ldots,\poly_{\LL}}_{i_{\LL+1}-(i_{\LL}+1)-1},\mathcal\poly_{\LL+1},\ldots, \mathcal\poly_k).
\end{align*}
To shorten the notation, let
\begin{equation*}
 \mathcal \poly:=(\mathcal \poly_0,\mathcal \poly_1,\ldots,\underbrace{\poly_{\LL-1},\ldots,\poly_{\LL-1}}_{i_{\LL}-i_{\LL-1}-2}, \underbrace{\poly_{\LL},\ldots,\poly_{\LL}}_{i_{\LL+1}-i_{\LL}-2},\mathcal\poly_{\LL+1},\ldots,\mathcal \poly_k),
\end{equation*}
to get
\begin{equation}
\label{eq:posetvolrep}
\begin{split}
&|\mathcal N_-|=(n-k)!\V_{n-k}(K_{\LL},K_{\LL},\mathcal \poly),\\
&|\mathcal N_=|=(n-k)!\V_{n-k}(K_{\LL-1},K_{\LL},\mathcal \poly),\\
&|\mathcal N_+|=(n-k)!\V_{n-k}(K_{\LL-1},K_{\LL-1},\mathcal \poly).
\end{split}
\end{equation}
With the representation \eqref{eq:posetvolrep} in hand, we get that the inequality \eqref{eq:eqmathcalN} is equivalent to
\[
\V_{n-k}(K_{\LL-1},K_{\LL},\mathcal \poly)^2\ge  \V_{n-k}(K_{\LL-1},K_{\LL-1},\mathcal \poly)\V_{n-k}(K_{\LL},K_{\LL},\mathcal \poly).
\]
The latter inequality follows immediately from the Alexandrov-Fenchel inequality  \cite[Theorem 7.3.1]{Sch14}: For any convex bodies $C_1,\ldots,C_{n-k}\subseteq\R^{n-k}$ we have
\begin{equation}
\label{eq:AF}
\tag{AF}
\V_{n-k}(C_1,C_2,C_3,\ldots,C_{n-k})^2\ge \V_{n-k}(C_1,C_1,C_3,\ldots,C_{n-k})\V_{n-k}(C_2,C_2,C_3,\ldots,C_{n-k}). 
\end{equation}
This completes Stanley's proof of \eqref{eq:eqmathcalN}. Since our goal in this paper is to understand the equality cases of \eqref{eq:eqmathcalN}, the above discussion naturally leads to the investigation of the equality cases of the Alexandrov-Fenchel inequality itself. 

\subsection{The extremals of the Alexandrov-Fenchel inequality for convex polytopes}
We start with the \emph{support function} associated to a convex body: Given a convex body $C\subseteq \R^{n-k}$ we define $h_C:S^{n-k-1}\to \R$ by
\[
h_C(\mathrm u):=\sup_{x\in C}\langle \mathrm u,x\rangle,\quad \mbox{for }\mathrm u\in S^{n-k-1}.
\]
The support function evaluated at $\mathrm u$ gives the distance to the origin of the  hyperplane orthogonal to $\mathrm u$ supporting $C$. The support function respects the summation of convex bodies in the sense that
\[
h_{\lambda C+\lambda' C'}=\lambda h_C+\lambda'h_{C'},
\]
for any convex bodies $C,C'\subseteq\R^{n-k}$ and scalars $\lambda,\lambda'\ge 0$. The function $h_C$ completely describes $C$ in the sense that two convex bodies are the same if their support functions are identical. That is, $C=C'$ if $h_C(\mathrm u)=h_{C'}(\mathrm u)$ for every $\mathrm u\in S^{n-k-1}$. Since mixed volumes are invariant under translations, and scale proportionally with dilations, it is clear that equality holds in \eqref{eq:AF} whenever there exist $a\ge 0$ and $\mathrm v\in\R^{n-k}$ such that $h_{C_1}(\mathrm u)=h_{aC_2+\mathrm v}(\mathrm u)$ for every $\mathrm u\in S^{n-k-1}$. However, the difficulty in characterizing the extremals of  the Alexandrov-Fenchel inequality stems from the fact that equality can be attained in \eqref{eq:AF} even if $h_{C_1}$ and $h_{aC_2+\mathrm v}$ agree on a very small subset of $S^{n-k-1}$. The complete characterization of the extremals of \eqref{eq:AF} has been open for decades. But in the case of \emph{polytopes}, which is the setting relevant to Stanley's inequalities, the problem was completely settled in \cite{SvH20}. In order to present the results of \cite{SvH20} we need some definitions. In the sequel, $B\subseteq \R^{n-k}$ always stands for the unit ball. Given a polytope $C\subseteq \R^{n-k}$ and $\mathrm u\in S^{n-k-1}$ we write
\[
F(C,\mathrm u):=\{x\in C:\langle \mathrm u,x\rangle=h_C(\mathrm u)\},
\] 
for the face of $C$ in the direction $\mathrm u$. We recall \cite[Theorem 1.7.2]{Sch14} that
\begin{align}
\label{eq:linface}
F(C+C',\mathrm u)=F(C,\mathrm u)+F(C',\mathrm u),
\end{align}
for any convex bodies $C,C'$ and $\mathrm u\in S^{n-k-1}$.
\begin{definition}
\label{def:extreme}
Let $\mathcal C:=(C_3,\ldots,C_{n-k})$ be a nonempty collection of polytopes in $\R^{n-k}$. A vector $\mathrm u\in S^{n-k-1}$ is a $(B,\mathcal C)$\emph{-extreme normal direction} if, for any $\mathcal C'\subseteq\mathcal C$,
\[
\dim\left(\sum_{C\in\mathcal C'}F(C,\mathrm u)\right)\ge |\mathcal C'|.
\]
\end{definition}
One example of $(B,\mathcal C)$-extreme normal directions can be found in Figure \ref{fig:graph}. The definition of $(B,\mathcal C)$-extreme normal directions plays a crucial role in the characterization of the extremals of the Alexandrov-Fenchel inequality for convex polytopes. For example, it follows from \cite{SvH20} that if $C_1,\ldots,C_{n-k}$ are \emph{full-dimensional} polytopes in $\R^{n-k}$, then, equality holds in \eqref{eq:AF} if, and only if, there exist $a\ge 0$ and $ \mathrm v\in\R^{n-k}$ such that
\[
h_{C_1}(\mathrm u)=h_{aC_2+\mathrm v}(\mathrm u) \quad \mbox{for every }(B,\mathcal C)\mbox{-extreme normal directions } \mathrm u. 
\]
\begin{figure}[h]
\centering
\begin{tikzpicture}[scale=.6]

\shade[ball color = blue, opacity = 0.15] (1,0) circle [radius=2];

\draw[thick, densely dashed] (3,0) arc [start angle = 0, end angle = 180, 
x radius = 2, y radius = .5];

\draw[thick] (3,0) arc [start angle = 0, end angle = -180, x radius = 2, 
y radius = .5];

\draw[thick] (1,2) [rotate=90] arc [start angle = 0, end angle = 
-180, x radius = 2, y radius = -1];

\draw[thick, densely dashed] (1,2) [rotate=90] arc [start angle 
= 0, end angle = 
-180, x radius = 2, y radius = 1];

\draw[thick] (1,2) [rotate=90] arc [start angle = 0, end angle = 
-180, x radius = 2, y radius = 1.75];

\draw[thick, densely dashed] (1,2) [rotate=90] arc [start angle 
= 0, end angle = 
-180, x radius = 2, y radius = -1.75];

\draw[fill=black] (1,2) circle [radius=.07];
\draw[fill=black] (1,-2) circle [radius=.07];
\draw[fill=black] (2.74,-.24) circle [radius=.09] node[below right] 
{$\mathrm v_{F}$};
\draw[fill=black] (-0.74,.24) circle [radius=.07];
\draw[fill=black] (.03,-.45) circle [radius=.09];
\draw[fill=black] (1.97,.45) circle [radius=.07];

\draw (.60,-.44) node[below] {$\mathrm v_{F'}$};

\draw (1.2,-.20) node {$e_{F,F'}$};


\begin{scope}[scale=.9]

\fill[blue!15] (-5.95,.55) -- (-4.03,1.55) -- 
(-4.03,-1.25) -- (-5.95,-2.25) -- (-5.95,.55);

\fill[blue!25] (-5.95,.55) -- (-8.47,.89) -- 
(-8.47,-1.91) -- (-5.95,-2.25) -- (-5.95,.55);

\fill[blue!5] (-8.47,.89) -- (-6.53,1.79) -- 
(-4.03,1.55) -- (-5.95,.55) -- (-8.47,.89);

\draw[thick] (-5.95,.55) -- (-4.03,1.55) -- (-4.03,-1.25) --
(-5.95,-2.25) -- (-5.95,.55);

\draw[thick] (-5.95,.55) -- (-8.47,.89) -- (-8.47,-1.91) --
(-5.95,-2.25);

\draw[thick] (-8.47,.89) -- (-6.53,1.79) -- (-4.03,1.55);

\draw[thick,densely dashed] (-6.53,1.79) -- (-6.53,-1.01) --
(-4.03,-1.25);

\draw[thick,densely dashed] (-6.53,-1.01) -- (-8.47,-1.91);
\end{scope}

\draw (-6.5,-.6) node {$F'$};

\draw (-4.4,-.3) node {$F$};

\end{tikzpicture}
\caption{\footnotesize Extreme normal directions associated to the cube. The vectors $\mathrm v_F,\mathrm v_{F'}\in S^2$ are the unit normals of the facets $F,F'$, and the line $e_{F,F'}$ is the shortest geodesic between the nodes $\mathrm v_F,\mathrm v_{F'}$. The $(\textnormal{Ball},\textnormal{Cube})$-extreme normal directions comprises of the nodes and arcs in this embedded graph on the sphere $S^2$.\label{fig:graph}}
\end{figure}
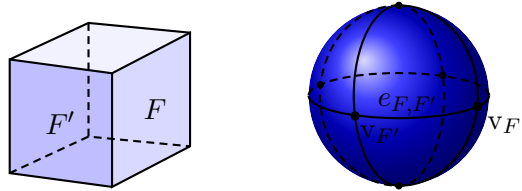

In the setting of Stanley's inequalities, the full-dimensionality assumption does not hold so we need the full power of the results of \cite{SvH20}. This requires a few definitions. 

\begin{definition}
\label{def:ssc}
Let $\mathcal C$  be a nonempty collection of polytopes in $\R^{n-k}$. 
\begin{itemize}
\item The collection $\mathcal C$ is \emph{subcritical} if, for any collection $\mathcal C'\subseteq\mathcal C$, $\dim\left(\sum_{C\in \mathcal C'}C\right)\ge |\mathcal C'|$. A collection $\mathcal C' \subseteq\mathcal C$ is \emph{sharp-subcritical} if $\dim\left(\sum_{C\in \mathcal C'}C\right)= |\mathcal C'|$.
\item The collection $\mathcal C$ is \emph{critical} if, for any nonempty collection $\mathcal C'\subseteq\mathcal C$, $\dim\left(\sum_{C\in \mathcal C'}C\right)\ge |\mathcal C'|+1$. A collection $\mathcal C' \subseteq\mathcal C$ is \emph{sharp-critical} if $\dim\left(\sum_{C\in \mathcal C'}C\right)= |\mathcal C'|+1$.
\item The collection $\mathcal C$ is \emph{supercritical} if, for any nonempty collection $\mathcal C'\subseteq\mathcal C$, $\dim\left(\sum_{C\in \mathcal C'}C\right)\ge |\mathcal C'|+2$.
\end{itemize}
\end{definition}
The origin of the above definition is the following lemma, which characterizes the conditions under which mixed volumes are positive \cite[Theorem 5.1.8]{Sch14}.
\begin{lemma}{\textnormal{(\textbf{Positivity of mixed volumes})}}
\label{lem:mvpos}
Let $C_1,\ldots, C_{n-k}$ be convex bodies in $\R^{n-k}$. Then, $\V_{n-k}(C_1,\ldots, C_{n-k})>0$ if, and only if, 
\[
\dim\left(\sum_{C\in\mathcal C'}C\right)\ge |\mathcal C'|\quad\mbox{for every collection } \mathcal C'\subseteq \{C_i\}_{i\in\llbracket 1,n-k\rrbracket}.
\]
\end{lemma}
For example, if the collection of polytopes $\mathcal C:=(C_3,\ldots,C_{n-k})$ in \eqref{eq:AF} is not subcritical, then Lemma \ref{lem:mvpos} shows that equality holds in \eqref{eq:AF} for trivial reasons: both sides of the inequality are zero. If $\mathcal C$ is subcritical with a sharp-subcritical collection, then the equality cases of \eqref{eq:AF} can be reduced to the equality cases of the Alexandrov-Fenchel inequality in a lower dimension; we refer to \cite{SvH20} for details. The difficult equality cases of \eqref{eq:AF}  are the supercritical and, to a much larger degree, the critical collections. The following definition is needed for the characterization of the critical extremals of \eqref{eq:AF}. 

\begin{definition}
\label{def:deg}
Let $\mathcal C=(C_3,\ldots,C_{n-k})$  be a collection of polytopes in $\R^{n-k}$ and let $(P,Q)$ be a pair of convex bodies in $\R^{n-k}$. The pair $(P,Q)$ is a $\mathcal C$\emph{-degenerate pair} if $P$ is not a translate of $Q$,
\begin{align*}
\V_{n-k}(P,Q,\mathcal C)=0,\quad\mbox{and}\quad \V_{n-k}(P,B,\mathcal C)=\V_{n-k}(Q,B,\mathcal C).
\end{align*}
\end{definition}

\begin{theorem}{\textnormal{(\cite[Theorem 2.13, Corollary 2.16]{SvH20})}}
\label{thm:SvH}
Let $C_1,\ldots, C_{n-k}$ be  polytopes in $\R^{n-k}$ and let $\mathcal C:=(C_3,\ldots, C_{n-k})$. 
\begin{itemize}
\item Suppose $\mathcal C$ is supercritical. Then, equality holds in \eqref{eq:AF} if, and only if, there exist $a\ge 0$ and $\mathrm v\in \R^{n-k}$ such that
\[
h_{C_1}(\mathrm u)=h_{aC_2+\mathrm v}(\mathrm u) \quad \mbox{for all }(B,\mathcal C)\textnormal{-extreme normal directions } \mathrm u. 
\]
\item Suppose $\mathcal C$ is critical. Then, equality holds in \eqref{eq:AF} if, and only if, there exist $a\ge 0,\, \mathrm v\in \R^{n-k}$, and a number $0\le d<\infty$ of $\mathcal C$-degenerate pairs $(P_1,Q_1),\ldots,(P_d,Q_d)$, such that 
\[
h_{C_1+\sum_{j=1}^dQ_j}(\mathrm u)=h_{aC_2+\mathrm v+\sum_{j=1}^dP_j}(\mathrm u) \quad \mbox{for all }(B,\mathcal C)\textnormal{-extreme normal directions } \mathrm u. 
\]
\end{itemize}
\end{theorem}

\subsubsection{The extremals of Stanley's inequalities} The crux of our work lies in understanding how to apply Theorem \ref{thm:SvH} in our setting in order to get a \emph{combinatorial} characterization of the equality cases of \eqref{eq:eqmathcalN}. For convenience and future reference, let us explicitly write Theorem \ref{thm:SvH} in our setting. 
\begin{theorem}
\label{thm:SvHComb}
$~$
\begin{itemize}
\item Suppose $\mathcal \poly$ is supercritical. Then, $|\mathcal N_{=}|^2=|\mathcal N_{-}||\mathcal N_{+}|$ holds, if, and only if, there exist $a\ge 0$ and $\mathrm v\in \R^{n-k}$ such that
\[
h_{\poly_{\LL-1}}(\mathrm u)=h_{a\poly_{\LL}+\mathrm v}(\mathrm u) \quad \mbox{for all }(B,\mathcal \poly)\textnormal{-extreme normal directions } \mathrm u. 
\]
\item Suppose $\mathcal \poly$ is critical. Then, $|\mathcal N_{=}|^2=|\mathcal N_{-}||\mathcal N_{+}|$ holds, if, and only if, there exist $a\ge 0,\,\mathrm v\in \R^{n-k}$, and a number $0\le d<\infty$ of $\mathcal \poly$-degenerate pairs $(P_1,Q_1),\ldots,(P_d,Q_d)$, such that 
\[
h_{\poly_{\LL-1}+\sum_{j=1}^dQ_j}(\mathrm u)=h_{a\poly_{\LL}+\mathrm v+\sum_{j=1}^dP_j}(\mathrm u) \quad \mbox{for all }(B,\mathcal \poly)\textnormal{-extreme normal directions } \mathrm u. 
\]
\end{itemize}
\end{theorem}
Our proof proceeds by induction on $k$. The base case $k=0$ is trivial as equality in \eqref{eq:eqmathcalN} cannot occur because $|\mathcal N_-|=|\mathcal N_+|=0$ while $|\mathcal N_=|=|\mathcal N|$. Hence, Theorem \ref{thm:supcrit} and Theorem \ref{thm:crit} hold trivially when $k=0$. From here on we assume that $k\ge 1$ and that equality holds in \eqref{eq:eqmathcalN}:
\begin{equation*}
|\mathcal N_{=}|^2=|\mathcal N_{+}||\mathcal N_{-}|\quad\Longleftrightarrow \quad\V_{n-k}(K_{\LL-1},K_{\LL},\mathcal \poly)^2=\V_{n-k}(K_{\LL-1},K_{\LL-1},\mathcal \poly)\V_{n-k}(K_{\LL},K_{\LL},\mathcal \poly).
\end{equation*}
\begin{assumption}
\label{ass:induct}
Theorem \ref{thm:supcrit} and Theorem \ref{thm:crit}  hold true for $k-1$.
\end{assumption}
We conclude this section by introducing the notions of criticality for posets. The relations between the criticality notions of Definition \ref{def:ssc} and the following Definition \ref{def:posetcrit} is given in Section \ref{sec:notionsCrit}.
\begin{definition}
\label{def:posetcrit}
Let $\abar=\{y_1,\ldots,y_{n-k}\}\cup\{x_1,\ldots,x_k\}$ be a poset, with a fixed chain $x_1<\cdots<x_k$, and fix $1\le i_1<\cdots <i_k\le n$ such that $i_{\LL-1}+1<i_{\LL}<i_{\LL+1}-1$ for some fixed $\LL\in [k]$. Suppose that $|\mathcal N_=|>0$.
\begin{itemize}
\item The poset $\abar$ is \emph{supercritical} if, for any integer $p\ge 1$ and $j_0:=-1< j_1<\cdots<j_p<k+1=:j_{p+1}$, such that  $ i_{j_q+1}-i_{j_q}-1-1_{j_q\in\{\LL-1,\LL\}}$ are positive for any $q\in [p]$, we have
\[
\sum_{q=0}^p1_{\{j_q+1<j_{(q+1)}\}}|\abar_{>x_{j_q+1},<x_{j_{(q+1)}}}|\le |\{q\in [p]:j_q\in\{\LL-1,\LL\}\}|-2+\sum_{q=0}^p1_{\{j_q+1<j_{(q+1)}\}}(i_{j_{(q+1)}}-i_{j_q+1}-1).
\]
\item The poset $\abar$ is \emph{critical} if, for any integer $p\ge 1$ and $j_0:=-1< j_1<\cdots<j_p<k+1=:j_{p+1}$, such that  $ i_{j_q+1}-i_{j_q}-1-1_{j_q\in\{\LL-1,\LL\}}$ are positive for any $q\in [p]$, we have
\[
\sum_{q=0}^p1_{\{j_q+1<j_{(q+1)}\}}|\abar_{>x_{j_q+1},<x_{j_{(q+1)}}}|\le |\{q\in [p]:j_q\in\{\LL-1,\LL\}\}|-1+\sum_{q=0}^p1_{\{j_q+1<j_{(q+1)}\}}(i_{j_{(q+1)}}-i_{j_q+1}-1).
\]
\end{itemize}
To get some intuition for Definition \ref{def:posetcrit} note that when $|\mathcal N_=|>0$ we have  
\[
|\abar_{>x_{r+1},<x_s}|\le i_s-i_{r+1}-1\quad \forall  ~ r\le s.
\]
Hence, criticality is captured in Definition \ref{def:posetcrit} by checking the tightness of the above bound. (Equivalent and more transparent definitions of  (super)criticality of posets are given in \cite[\S 10.7]{chanpakequiv2023}.)

Finally, let us remark that the case $k=1$ is always supercritical, where we use that $|\mathcal N_-|,|\mathcal N_=|,|\mathcal N_+|$ are positive, as $|\mathcal N_=|>0$ and $|\mathcal N_{=}|^2=|\mathcal N_{+}||\mathcal N_{-}|$.
\end{definition}

\section{Linear extensions}
\label{sec:found}
In this section we introduce a number of ideas and tools that will simplify the proofs of our main results. Section \ref{subsec:decomp} presents a \emph{decompositions} of $\mathcal N_{-}, \mathcal N_{=},\mathcal N_{+}$. Section \ref{subsec:suff} uses the above decompositions to prove the sufficiency part of Theorem \ref{thm:supcrit} and Theorem \ref{thm:critsec} (Proposition \ref{prop:suff}), and introduces conditions which are equivalent to Theorem \ref{thm:supcrit} and Theorem \ref{thm:critsec} (Lemma \ref{lem:suff}). Finally, Section \ref{subsec:cl} introduces the  technical tool of \emph{closure} where relations are added to the poset $\abar$ based on linear extensions.

\subsection{Decompositions of linear extensions}
\label{subsec:decomp}
Fix $\s\in\{-,=,+\}$ and $\star,	\ast\in\{\incomp,\sim\}$. Recall Definition \ref{def:neighbor} and let
\[
\mathcal N_{\s}(\star,\ast):=\{\sigma\in\mathcal N_{\s}:~\text{lower companion  $\star ~ x_{\LL}$ and upper companion  $\ast ~x_{\LL}$}\}.
\]
It is clear that we have the disjoint decompositions, 
\begin{align}
\label{eq:decomps_sum}
|\mathcal N_-|&=|\mathcal N_-(\incomp,\incomp)|+|\mathcal N_-(\incomp,\sim)|+|\mathcal N_-(\sim,\incomp)|+|\mathcal N_-(\sim,\sim)|,\nonumber\\
|\mathcal N_=|&=|\mathcal N_=(\incomp,\incomp)|+|\mathcal N_=(\incomp,\sim)|+|\mathcal N_=(\sim,\incomp)|+|\mathcal N_=(\sim,\sim)|,\\
|\mathcal N_+|&=|\mathcal N_+(\incomp,\incomp)|+ |\mathcal N_+(\incomp,\sim)|+|\mathcal N_+(\sim,\incomp)|+|\mathcal N_+(\sim,\sim)|\nonumber.
\end{align}
The next result shows that, regardless of whether equality holds in \eqref{eq:eqmathcalN}, certain relations between terms in \eqref{eq:decomps_sum} always hold.
\begin{lemma}
\label{lem:decompsimpl}
For any poset $\abar$ the following hold:
\begin{enumerate}[(i)]
\item $|\mathcal N_-(\incomp,\incomp)|=|\mathcal N_=(\incomp,\incomp)|=| \mathcal N_+(\incomp,\incomp)|$.
\item $|\mathcal N_-(\incomp,\sim)|=|\mathcal N_=(\incomp,\sim)|$.
\item $|\mathcal N_=(\sim,\incomp)|=|\mathcal N_+(\sim,\incomp)|$.
\item $|\mathcal N_-(\sim,\incomp)|\le |\mathcal N_-(\incomp,\sim)|$.
\item $|\mathcal N_+(\incomp,\sim)|\le |\mathcal N_+(\sim,\incomp)|$.
\end{enumerate}
\end{lemma}

\begin{proof}
$~$

\begin{enumerate}[(i)]
\item We show $|\mathcal N_-(\incomp,\incomp)|=|\mathcal N_=(\incomp,\incomp)|$; the argument for $|\mathcal N_=(\incomp,\incomp)|=| \mathcal N_+(\incomp,\incomp)|$ is analogous. Let $\pi_{i_{\LL-1},i_\LL}:[n]\to [n]$ be the permutation that swaps the positions of $i_{\LL-1}$ and $i_{\LL}$. We claim that defining $\pi_{i_{\LL-1},i_{\LL}}(\sigma):=\pi_{i_{\LL-1},i_{\LL}}\circ \sigma$, for $\sigma \in \mathcal N_-(\incomp,\incomp)$, yields a bijection $\pi_{i_{\LL-1},i_{\LL}}:\mathcal N_-(\incomp,\incomp)\to\mathcal N_=(\incomp,\incomp)$. That $\pi_{i_{\LL-1},i_{\LL}}(\mathcal N_-(\incomp,\incomp))\subseteq\mathcal N_=(\incomp,\incomp)$ follows from the fact that $x_{\LL}$ is incomparable to the element placed in $i_{\LL}$ so their positions can be swapped. Hence, to conclude that $\pi_{i_{\LL-1},i_{\LL}}$ is a bijection it suffices to show that $\pi_{i_{\LL-1},i_{\LL}}$ is invertible and that its inverse $\pi_{i_{\LL-1},i_{\LL}}^{-1}$ satisfies $\pi_{i_{\LL-1},i_{\LL}}^{-1}(\mathcal N_=(\incomp,\incomp))\subseteq \mathcal N_-(\incomp,\incomp)$. The inverse $\pi_{i_{\LL-1},i_{\LL}}^{-1}$ exists since $\pi_{i_{\LL-1},i_{\LL}}^{-1}=\pi_{i_{{\LL}-1},i_{\LL}}$. That  $\pi_{i_{\LL-1},i_{\LL}}(\mathcal N_=(\incomp,\incomp))\subseteq \mathcal N_-(\incomp,\incomp)$ is clear. 
\item Analogous argument to (i).
\item Analogous argument to (i).
\item  Let $\pi_{i_{\LL},i_\LL+1}:[n]\to [n]$ be the permutation that swaps the positions of $i_{\LL}$ and $i_{\LL+1}$. We claim that defining $\pi_{i_{\LL},i_{\LL+1}}(\sigma):=\pi_{i_{\LL},i_{\LL+1}}\circ \sigma$, for $\sigma \in \mathcal N_-(\sim,\incomp)$, yields an injection $\pi_{i_{\LL},i_{\LL+1}}:\mathcal N_-(\sim,\incomp)\to\mathcal N_-(\incomp,\sim)$. Indeed, fix $\sigma \in \mathcal N_-(\sim,\incomp)$, so $\sigma(x_{\LL})=i_{\LL}-1$, and let $y_u:=\sigma^{-1}(i_{\LL}),y_v:=\sigma^{-1}(i_{\LL}+1)$ so that, by the definition of $ \mathcal N_-(\sim,\incomp)$, $x_{\LL}<y_u$ and $y_v\incomp x_{\LL}$. We cannot have $y_u<y_v$ since that would imply $x_{\LL}<y_u<y_v$ contradicting $y_v\incomp x_{\LL}$. Since $y_u=\sigma^{-1}(i_{\LL}),y_v=\sigma^{-1}(i_{\LL}+1)$, we cannot have $y_v<y_u$ so we must have $y_u\incomp y_v$. It follows that swapping the positions of $y_u$ and $y_v$ in $\sigma$ yields the linear extension $\pi_{i_{\LL},i_{\LL+1}}(\sigma)\in \mathcal N_-(\incomp,\sim)$.
\item  Analogous argument to (iv).
\end{enumerate}
\end{proof}

\subsection{Sufficiency}
\label{subsec:suff}
The decompositions given in Section \ref{subsec:decomp} help us prove the sufficiency of the conditions of Theorem \ref{thm:supcrit}(iii)  and Theorem \ref{thm:crit}(iii). 
\begin{proposition}{\textnormal{(\textbf{Sufficient conditions})}}
\label{prop:suff}
$~$
\begin{enumerate}[(a)]
\item  Theorem \ref{thm:supcrit}(ii) $\Longrightarrow$ Theorem \ref{thm:supcrit}(i) and Theorem \ref{thm:crit}(ii) $\Longrightarrow$ Theorem \ref{thm:crit}(i).
\item Theorem \ref{thm:crit}(iii) $\Longrightarrow$ Theorem \ref{thm:crit}(ii).
\item Theorem \ref{thm:supcrit}(iii) $\Longrightarrow$ Theorem \ref{thm:crit}(iii) $\Longrightarrow$ Theorem \ref{thm:supcrit}(ii).
\end{enumerate}
\end{proposition}

\begin{proof}
$~$

\begin{enumerate}[(a)]
\item Immediate.\\

\item The conditions in Theorem \ref{thm:crit}(iii) read
\begin{align*}
&|\mathcal N_-(\sim,\sim)|=|\mathcal N_=(\sim,\sim)|=| \mathcal N_+(\sim,\sim)|=0,\\
&|\mathcal N_-(\incomp,\sim)|=|\mathcal N_-(\sim,\incomp)|=\mathrm N_1,\\
&|\mathcal N_=(\incomp,\sim)|=|\mathcal N_=(\sim,\incomp)|=\mathrm N_1,\\
&|\mathcal N_+(\incomp,\sim)|=|\mathcal N_+(\sim,\incomp)|=\mathrm N_1,\\
&|\mathcal N_-(\incomp,\incomp)|=|\mathcal N_=(\incomp,\incomp)|=| \mathcal N_+(\incomp,\incomp)|=\mathrm N_2.
\end{align*}
Hence, \eqref{eq:decomps_sum} reads
\begin{align*}
&|\mathcal N_-|=\mathrm N_2+\mathrm N_1+\mathrm N_1+0=\mathrm N_2+2\mathrm N_1,\\
&|\mathcal N_=|=\mathrm N_2+\mathrm N_1+\mathrm N_1+0=\mathrm N_2+2\mathrm N_1,\\
&|\mathcal N_+|=\mathrm N_2+\mathrm N_1+\mathrm N_1+0=\mathrm N_2+2\mathrm N_1,
\end{align*}
which is the statement in Theorem \ref{thm:crit}(ii).\\

\item The first implication is immediate and the second implication follows from (b).
\end{enumerate}
\end{proof}

In order to prove Theorem \ref{thm:supcrit} and Theorem \ref{thm:crit} it remains to show that Theorem \ref{thm:supcrit}(i) $\Longrightarrow$ Theorem \ref{thm:supcrit}(iii) and  Theorem \ref{thm:crit}(i) $\Longrightarrow$ Theorem \ref{thm:crit}(iii). To this end,  the following conditions will suffice.

\begin{lemma}
\label{lem:suff}
$~$

\begin{enumerate}[(a)]
\item The conditions in Theorem \ref{thm:supcrit}(iii) hold if, and only if, 
\[
|\mathcal N_=(\incomp,\sim)|=|\mathcal N_=(\sim,\incomp)|=|\mathcal N_=(\sim,\sim)|=0.
\]
\item Suppose $|\mathcal N_=|^2=|\mathcal N_-||\mathcal N_+|$. The conditions in Theorem \ref{thm:crit}(iii) hold if, and only if, 
\[
|\mathcal N_-(\sim,\sim)|=|\mathcal N_+(\sim,\sim)|=0.
\]
\end{enumerate}
\end{lemma}
\begin{proof}
We start with proof of (a).  The ``only if" part is clear. To prove the ``if" part, assume that 
\[
|\mathcal N_=(\incomp,\sim)|=|\mathcal N_=(\sim,\incomp)|=|\mathcal N_=(\sim,\sim)|=0,
\]
which by  \eqref{eq:decomps_sum} implies
\[
|\mathcal N_=|=|\mathcal N_=(\incomp,\incomp)|.
\]
On the other hand, Lemma \ref{lem:decompsimpl}(i) yields
\[
N':=|\mathcal N_-(\incomp,\incomp)|=|\mathcal N_=(\incomp,\incomp)|=|\mathcal N_+(\incomp,\incomp)|,
\]
so \eqref{eq:decomps_sum} reads
\begin{align*}
|\mathcal N_-|&=N'+|\mathcal N_-(\incomp,\sim)|+|\mathcal N_-(\sim,\incomp)|+|\mathcal N_-(\sim,\sim)|,\\
|\mathcal N_=|&=N',\\
|\mathcal N_+|&=N'+ |\mathcal N_+(\incomp,\sim)|+|\mathcal N_+(\sim,\incomp)|+|\mathcal N_+(\sim,\sim)|.
\end{align*}
Stanley's inequality \eqref{eq:eqmathcalN},
\[
|\mathcal N_{=}|^2\ge |\mathcal N_{-}||\mathcal N_{+}|,
\]
implies that all the terms other than $N'$ must vanish, which completes the proof.\\

We now prove (b). The `only if" part is clear. To prove the ``if" part, assume that
\[
|\mathcal N_-(\sim,\sim)|=| \mathcal N_+(\sim,\sim)|=0.
\]
Using Lemma \ref{lem:decompsimpl}(i-iii), set
\begin{align*}
&N':=|\mathcal N_-(\incomp,\incomp)|=|\mathcal N_=(\incomp,\incomp)|=| \mathcal N_+(\incomp,\incomp)|,\\
&N_a':=|\mathcal N_-(\incomp,\sim)|=|\mathcal N_=(\incomp,\sim)|,\\
&N_b':=|\mathcal N_=(\sim,\incomp)|=|\mathcal N_+(\sim,\incomp)|,
\end{align*}
so \eqref{eq:decomps_sum} reads
\begin{align*}
|\mathcal N_-|&=N'+N_a'+|\mathcal N_-(\sim,\incomp)|,\\
|\mathcal N_=|&=N'+N_a'+N_b'+|\mathcal N_=(\sim,\sim)|,\\
|\mathcal N_+|&=N'+N_b'+ |\mathcal N_+(\incomp,\sim)|.
\end{align*}
By Lemma \ref{lem:decompsimpl}(iv-v),
\[
|\mathcal N_-(\sim,\incomp)|\le N_a'\quad\text{and}\quad  |\mathcal N_+(\incomp,\sim)|\le N_b'
\]
so
\begin{align*}
|\mathcal N_-|&=N'+N'_a+|\mathcal N_-(\sim,\incomp)|\le N'+2N_a',\\
|\mathcal N_=|&=N'+N'_a+N'_b+|\mathcal N_=(\sim,\sim)|\ge N'+N_a'+N_b',\\
|\mathcal N_+|&=N'+N_b'+ |\mathcal N_+(\incomp,\sim)|\le N'+2N_b'.
\end{align*}
Hence,
\begin{align*}
&(N'+2N_a')(N'+2N_b')=(N'+N_a'+N_b')^2-(N_a'-N_b')^2\le (N'+N_a'+N_b')^2\\
&\le |\mathcal N_=|^2=|\mathcal N_-||\mathcal N_+|\le  (N'+2N_a')(N'+2N_b').
\end{align*}
It follows that all of the above inequalities are in fact equalities. In particular,
\begin{align}
&|\mathcal N_=(\sim,\sim)|=0,\label{N-=0}\\
&N_a'=N_b',\label{Na'=Nb'}\\
&|\mathcal N_-(\sim,\incomp)|=N_a',\label{Na'=}\\
&|\mathcal N_+(\incomp,\sim)|=N_b'.\label{Nb'=}
\end{align}
The identity \eqref{N-=0}, together with the assumption  $|\mathcal N_-(\sim,\sim)|=|\mathcal N_+(\sim,\sim)|=0$, imply  that every linear extension in $\mathcal N_{\s}$, for any $\s\in\{-,=,+\}$, has either 0 or 1 comparable companions to $x_{\LL}$. It remains to show that there exist  nonnegative numbers $\mathrm N_1, \mathrm N_2$ such that
\begin{align*}
&|\mathcal N_-(\sim,\incomp)|=|\mathcal N_=(\sim,\incomp)|=|\mathcal N_+(\sim,\incomp)|=|\mathcal N_-(\incomp,\sim)|=|\mathcal N_=(\incomp,\sim)|=|\mathcal N_+(\incomp,\sim)|=\mathrm N_1,\\
&|\mathcal N_-(\incomp,\incomp)|=|\mathcal N_=(\incomp,\incomp)|=|\mathcal N_+(\incomp,\incomp)|=\mathrm N_2.
\end{align*}
The first part follows since 
\begin{align*}
&|\mathcal N_-(\sim,\incomp)|\underset{\eqref{Na'=}}{=}N_a':=|\mathcal N_-(\incomp,\sim)|\underset{\text{Lemma \ref{lem:decompsimpl}(ii)}}{=}|\mathcal N_=(\incomp,\sim)|\underset{\eqref{Na'=Nb'}}{=}N_b':=|\mathcal N_=(\sim,\incomp)|\\
&\underset{\text{Lemma \ref{lem:decompsimpl}(iii)}}{=}|\mathcal N_+(\sim,\incomp)|\underset{\eqref{Nb'=}}{=}|\mathcal N_+(\incomp,\sim)|=:\mathrm N_1,
\end{align*}
and the second part follows by Lemma \ref{lem:decompsimpl}(i). 
\end{proof}

We conclude the section with a corollary of the above lemmas, which will be needed for the proof of Theorem \ref{thm:N_=>0intro}. (Note that the assumption in the following result that $\abar$ is critical can be  relaxed to $|\mathcal N_=|>0$, cf. Section \ref{sec:crit}.)

\begin{corollary}
\label{cor:referee}
Let $\abar$ be a critical poset such that  $|\mathcal N_=|^2=|\mathcal N_-||\mathcal N_+|$, and assume that Theorem \ref{thm:supcrit} and Theorem \ref{thm:crit} hold true. Fix $\s\in \{-,=,+\}$ and $\sigma\in \mathcal N_{\circ}(\sim,\incomp)\cup \mathcal N_{\circ}(\incomp,\sim)$. Then, the upper and lower companions are incomparable to each other. 
\end{corollary}
\begin{proof}
We start by establishing the claim in the case where $\s$ is equal to $-$. Fix $\sigma\in \mathcal N_{-}(\sim,\incomp)$.  If the  upper and lower companions are comparable to each other, then, by transitivity, $\sigma\in \mathcal N_{-}(\sim,\sim)$, which is a contradiction. On the other hand, the proof of Lemma \ref{lem:suff} shows that $|\mathcal N_{-}(\incomp,\sim)|=|\mathcal N_{-}(\sim,\incomp)|$. Hence, the map $\pi_{i_{\LL},i_\LL+1}:\mathcal N_{-}(\sim,\incomp)\to \mathcal N_{-}(\incomp,\sim)$ defined in the proof of Lemma \ref{lem:decompsimpl}(iv) is a bijection. It follows that the  upper and lower companions in any $\sigma\in \mathcal N_{-}(\incomp,\sim)$ cannot be comparable to each other, or else they will also be comparable to each other in $\pi_{i_{\LL},i_\LL+1}(\sigma)\in  \mathcal N_{-}(\sim,\incomp)$, which is a contradiction. 

Analogous argument works when $\s$ is equal to $+$. In the case when $\s$ is equal to $=$, we note that Lemma \ref {lem:decompsimpl}(ii-iii) gives bijections $\mathcal N_{-}(\incomp,\sim)\to \mathcal N_{=}(\incomp,\sim)$ and $\mathcal N_{+}(\sim, \incomp)\to \mathcal N_{=}(\sim, \incomp)$, so we can argue as above to conclude that  the  upper and lower companions are incomparable. 
\end{proof}

\subsection{Closure}
\label{subsec:cl}
Since we are interested in the extremals of \eqref{eq:eqmathcalN}, it is beneficial to add relations to $\abar$ which are compatible with $\mathcal N_{-},\mathcal N_{=},\mathcal N_{+}$, while leaving these sets invariant.
\begin{definition}
\label{def:clposet}
Denote by $\textnormal{Cl}(\abar)$ (the \emph{closure} of $\abar$) the poset with  the same elements as $\abar$ and with the partial order on $\textnormal{Cl}(\abar)$ given by
\[
w<z\quad\textnormal{if and only if}\quad \sigma(w)<\sigma(z)~\forall ~\sigma\in \mathcal N_-\cup \mathcal N_=\cup \mathcal N_+.
\]
Let 
\[
\mathcal N^{\text{cl}}:=\{\textnormal{bijections }\sigma:\textnormal{Cl}(\abar)\to [n]: w\le z\Rightarrow \sigma(w)\le \sigma(z)~\forall ~w,z\in\textnormal{Cl}(\abar)\},
\]
with the analogous $\mathcal N_{\s}^{\text{cl}}(\star,\ast)$ for $\s \in\{-,=,+\}$ and $\star,\ast\in\{\incomp,\sim\}$.
\end{definition}

We first need to check that Definition \ref{def:clposet} is well-defined. Indeed, if $z_1,z_2,z_3\in  \textnormal{Cl}(\abar)$ are such that $z_1<z_2$ and $z_2<z_3$ in $\textnormal{Cl}(\abar)$, then, by definition, $\sigma(z_1)<\sigma(z_2)$ and $\sigma(z_2)<\sigma(z_3)$ for every $\sigma\in \bigcup_{\s \in\{-,=,+\}} \mathcal N_{\s}$, so $\sigma(z_1)<\sigma(z_2)<\sigma(z_3)$. It follows that $z_1<z_3$ in $\textnormal{Cl}(\abar)$.

Let us now show that the relations in $\textnormal{Cl}(\abar)$ are compatible with the relations in $\abar$.

\begin{lemma}
\label{lem:posetcl}
If $z_1<z_2$ in $\abar$ then $z_1<z_2$ in $\textnormal{Cl}(\abar)$. If $z_1\incomp z_2$ in $\textnormal{Cl}(\abar)$ then $z_1\incomp z_2$ in $\abar$. 
\end{lemma}
\begin{proof}
If $z_1<z_2$ in $\abar$, then $\sigma(z_1)<\sigma(z_2)$ for every $\sigma\in \bigcup_{\s \in\{-,=,+\}} \mathcal N_{\s}$, so $z_1<z_2$ in $\textnormal{Cl}(\abar)$. The contrapositive of this statement is that if $z_1\incomp z_2$ in $\textnormal{Cl}(\abar)$ then $z_1\incomp z_2$ in $\abar$. 
\end{proof}

While the closure operation is compatible with the relations in $\abar$, it can introduce new relations as the following example demonstrates.
\begin{example}
\label{ex:cl}
 Let $\abar=\{x_1,x_2,y_1,y_2,y_3\}$, so $k=2$ and $n=5$, and suppose that the only relations are $x_1<x_2$ and $y_1<x_1$. Let $i_1=2,i_2=4$ and $l=2$, and note that $i_{\LL-1}+1=i_1+1=3<4=i_2=i_{\LL}< 5=(n+1)-1=i_{\LL+1}-1$. Let us show that, in $\textnormal{Cl}(\abar)$, $x_1<y_2$ and $x_1<y_3$, relations which do not hold in $\abar$. Indeed, take any $\sigma\in \mathcal N_-\cup \mathcal N_=\cup \mathcal N_+$ and note that $\sigma(x_1)=i_1=2$ so, since $y_1<x_1$, we must have $\sigma(y_1)=1$. Thus, $\sigma(y_2),\sigma(y_3)>2$, and hence,  in $\textnormal{Cl}(\abar)$, $x_1<y_2$ and $x_1<y_3$. See Figure \ref{fig:excl}.
 
 \begin{figure}
\begin{tikzpicture}[scale=1]
\begin{scope}
  \node (x2) at (0,0) {$\color{red}{x_2}$};
  \node (x1) at (0,-1) {$\color{red}{x_1}$};
  \node (y3) at (2,-2) {$\color{blue}{y_3}$};
  \node (y2) at (1,-2) {$\color{blue}{y_2}$};
  \node (y1) at (0,-2) {$\color{blue}{y_1}$};
   \draw[->,thick] (x2) -- (x1);
    \draw[->,thick] (x1) -- (y1);
  \node at (1,-2.7){$\abar$};
\end{scope}

\begin{scope}[xshift=5cm]
  \node (x2) at (0,0) {$\color{red}{x_2}$};
  \node (x1) at (0,-1) {$\color{red}{x_1}$};
  \node (y3) at (1,0) {$\color{blue}{y_3}$};
  \node (y2) at (-1,0) {$\color{blue}{y_2}$};
  \node (y1) at (0,-2) {$\color{blue}{y_1}$};
   \draw[->,thick] (x2) -- (x1);
    \draw[->,thick] (x1) -- (y1);
  \draw[->,dashed,thick] (y2) -- (x1);
   \draw[->,dashed,thick] (y3) -- (x1);
\node at (0,-2.7){$\textnormal{Cl}(\abar)$};
\end{scope}

\begin{scope}[xshift=6cm, yshift=-0.3cm]
  \node at (5,0) {$\mathcal N_- =\{{\color{blue}{y_1}}{\color{red}{x_1x_2}}{\color{blue}{y_2y_3}},~ {\color{blue}{y_1}}{\color{red}{x_1x_2}}{\color{blue}{y_3y_2}}\}$};
  \node at (5,-1) {$\mathcal N_=  =\{{\color{blue}{y_1}}{\color{red}{x_1}}{\color{blue}{y_2}}{\color{red}{x_2}}{\color{blue}{y_3}}, ~{\color{blue}{y_1}}{\color{red}{x_1}}{\color{blue}{y_3}}{\color{red}{x_2}}{\color{blue}{y_2}}\}$};
  \node at (5,-2) {$\mathcal N_+ :=\{ {\color{blue}{y_1}}{\color{red}{x_1}}{\color{blue}{y_2y_3}}{\color{red}{x_2}}, ~{\color{blue}{y_1}}{\color{red}{x_1}}{\color{blue}{y_3y_2}}{\color{red}{x_2}}\}$};
 \end{scope}
\end{tikzpicture}

\caption{Hasse diagram (arrows point from smaller to larger elements) of posets in Example \ref{ex:cl}, together with their (identical) sets of linear extensions, showing that new relations can occur under the closure operation.}
\label{fig:excl}
\end{figure}
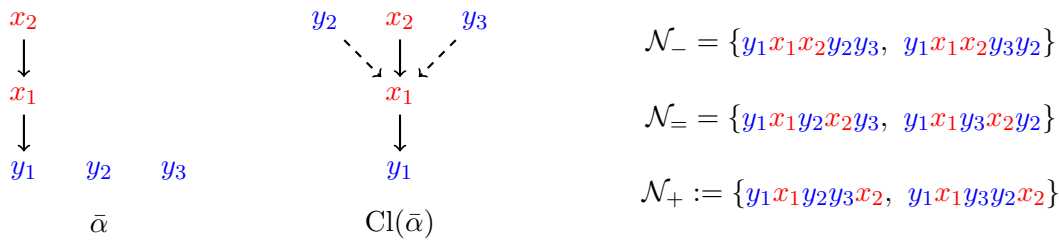 
\end{example}

The next result shows that our basic objects of interest remain more-or-less invariant under the closure operation. To simplify the notation, let $(\textnormal{i})$ (res. $(\textnormal{ii})$) stand for the conditions in Theorem \ref{thm:supcrit}(i) and Theorem \ref{thm:crit}(i) (res. Theorem \ref{thm:supcrit}(ii) and Theorem \ref{thm:crit}(ii)), and let $(\textnormal{iii}_{\text{supcrit}})$ (res. $(\textnormal{iii}_{\text{crit}})$) stand for the conditions in Theorem \ref{thm:supcrit}(iii) (res. Theorem \ref{thm:crit}(iii)). We use an upper script ``cl"  for the corresponding notation when $\textnormal{Cl}(\abar)$, rather than $\abar$, is used.
\begin{proposition}
\label{prop:clposet}
The set $\textnormal{Cl}(\abar)$ is a poset satisfying
\begin{enumerate}[(a)]
\item $\mathcal N_{\s}^{\textnormal{cl}}=\mathcal N_{\s}$ for every $\s\in\{-,=,+\}$.
\item $(\textnormal{i}^{\textnormal{cl}})\Longleftrightarrow (\textnormal{i})$,
\item $(\textnormal{ii}^{\textnormal{cl}})\Longleftrightarrow (\textnormal{ii})$,
\item $(\textnormal{iii}_{\textnormal{supcrit}}^{\textnormal{cl}})\Longrightarrow (\textnormal{iii}_{\textnormal{supcrit}})$ and  $(\textnormal{iii}_{\textnormal{crit}}^{\textnormal{cl}})\Longrightarrow (\textnormal{iii}_{\textnormal{crit}})$.
\end{enumerate}
\end{proposition}
\begin{proof}
$\textnormal{Cl}(\abar)$ is indeed a poset since  irreflexivity is immediate and transitivity was checked after Definition \ref{def:clposet}.\\
 
\begin{enumerate}[(a)]
\item We show that $\mathcal N_=^{\text{cl}}=\mathcal N_=$; the proof that $\mathcal N_-^{\text{cl}}=\mathcal N_-$ and $\mathcal N_+^{\text{cl}}=\mathcal N_+$ is analogous. We start by observing that since Lemma \ref{lem:posetcl} yields ``$w<z$ in $\abar$ implies  $w<z$ in $\textnormal{Cl}(\abar)$", it follows that ``$\sigma\in \mathcal N_=^{\text{cl}}$ implies $\sigma\in \mathcal N_=$". Conversely, let $\sigma\in\mathcal N_=$ so it suffices to show that $\sigma\in\mathcal N^{\text{cl}}$. The latter holds since if $w<z$ in $\textnormal{Cl}(\abar)$, then it must be, by the definition of $\textnormal{Cl}(\abar)$, that $\sigma(w)<\sigma(z)$, and hence $\sigma\in\mathcal N^{\text{cl}}$.\\

\item Follows trivially from (a).\\

\item Follows trivially from (a).\\

\item  We show that 
\begin{align}
\label{eq:clsupcrit}
|\mathcal N_{=}^{\text{cl}}(\incomp,\sim)|=|\mathcal N_{=}^{\text{cl}}(\sim,\incomp)|=|\mathcal N_{=}^{\text{cl}}(\sim,\sim)|=0\quad\Longrightarrow \quad |\mathcal N_{=}(\incomp,\sim)|=|\mathcal N_{=}(\sim,\incomp)|=|\mathcal N_{=}(\sim,\sim)|=0,
\end{align}
which proves $(\textnormal{iii}_{\textnormal{supcrit}}^{\textnormal{cl}})\Longrightarrow (\textnormal{iii}_{\textnormal{supcrit}})$ by Lemma \ref{lem:suff}(a). To establish \eqref{eq:clsupcrit} we show $|\mathcal N_{=}(\incomp,\sim)|=0$; the proof of $|\mathcal N_{=}(\sim,\incomp)|=0$ and $|\mathcal N_{=}(\sim,\sim)|=0$ is analogous. Suppose $|\mathcal N_{=}(\incomp,\sim)|>0$ so there exists $\sigma\in\mathcal N_=$ such that $\sigma(x_{\LL})=i_{\LL}$ and $x_{\LL}<\sigma^{-1}(i_{\LL}+1)$ in $\abar$. By (a), $\sigma \in\mathcal N_{=}^{\text{cl}}$, and by Lemma \ref{lem:posetcl}, $x_{\LL}<\sigma^{-1}(i_{\LL}+1)$ in $\textnormal{Cl}(\abar)$. It follows that $\sigma\in\mathcal N_{=}^{\text{cl}}(\incomp,\sim)\cup \mathcal N_{=}^{\text{cl}}(\sim,\sim)$, which is a contradiction.  

Next we show
\begin{align}
\label{eq:clto}
|\mathcal N_{-}^{\text{cl}}(\sim,\sim)|=|\mathcal N_{+}^{\text{cl}}(\sim,\sim)|=0\quad\Longrightarrow \quad |\mathcal N_{-}(\sim,\sim)|=|\mathcal N_{+}(\sim,\sim)|=0.
\end{align}
Since $(\textnormal{iii}_{\textnormal{crit}}^{\textnormal{cl}})\Longrightarrow (\textnormal{i}_{\textnormal{crit}}^{\textnormal{cl}})$ by Proposition \ref{prop:suff}(a--b), and since $(\textnormal{i}^{\textnormal{cl}})\Longleftrightarrow (\textnormal{i})$ by part (a), the proof will be complete by Lemma \ref{lem:suff}(b). 

To establish \eqref{eq:clto} we show that $|\mathcal N_{+}^{\text{cl}}(\sim,\sim)|=0\Rightarrow |\mathcal N_{+}(\sim,\sim)|=0$; the proof of $|\mathcal N_{-}^{\text{cl}}(\sim,\sim)|=0\Rightarrow |\mathcal N_{-}(\sim,\sim)|=0$ is analogous. Indeed, if $|\mathcal N_{+}(\sim,\sim)|>0$ then there exists $\sigma\in \mathcal N_{+}$ such that $\sigma(x_{\LL})=i_{\LL}+1$ and $\sigma^{-1}(i_{\LL}-1),\sigma^{-1}(i_{\LL})$ are both smaller than $x_{\LL}$ in $\abar$. By (a), $\sigma\in \mathcal N_{+}^{\text{cl}}$, and by Lemma \ref{lem:posetcl}, $\sigma^{-1}(i_{\LL}-1),\sigma^{-1}(i_{\LL})$ are both smaller than $x_{\LL}$ in $\textnormal{Cl}(\abar)$. It follows that $\sigma\in \mathcal N_{+}^{\text{cl}}(\sim,\sim)$. In other words, $|\mathcal N_{+}(\sim,\sim)|>0\Rightarrow |\mathcal N_{+}^{\text{cl}}(\sim,\sim)|>0$, which is the contrapositive of what we want to show. 
\end{enumerate}
\end{proof}

\section{Proof outline}
\label{sec:outline}
In this section we outline the proof of the characterization of the extremals of Stanley's inequalities. The first step is to understand how we use the closure procedure. We have the following equivalences:
\begin{align*}
&(\textnormal{i}^{\text{cl}})\quad\underset{\text{Thms. \ref{thm:supcritsec}, \ref{thm:critsec} + Lem. \ref{lem:suff}}}{\Longrightarrow}\hspace{.05in} (\textnormal{iii}^{\text{cl}})\quad\underset{\text{Prop. \ref{prop:suff}(b-c) }}{\Longrightarrow}\quad (\textnormal{ii}^{\text{cl}})\quad\underset{\text{trivial}}{\Longrightarrow}\quad (\textnormal{i}^{\text{cl}})\\
&\Updownarrow\tiny\text{Prop. \ref{prop:clposet}(b)}\hspace{1.23in}\Downarrow \tiny\text{Prop. \ref{prop:clposet}(d)}\hspace{.69in}\Updownarrow \tiny\text{Prop. \ref{prop:clposet}(c)}\hspace{0.23in}\Updownarrow\tiny\text{Prop. \ref{prop:clposet}(b)}\\
&\hspace{0.01in}(\textnormal{i}) \hspace{1.8in}(\textnormal{iii}) \hspace{0.3in}\underset{\text{Prop. \ref{prop:suff}(b-c)}}{\Longrightarrow}\hspace{0.15in}(\textnormal{ii})\quad\hspace{0.1in}\underset{\text{trivial}}{\Longrightarrow}\quad \hspace{0.02in}(\textnormal{i})
\end{align*}
The only implication that has not been proven thus far is $(\textnormal{i}^{\text{cl}})\Longrightarrow (\textnormal{iii}^{\text{cl}})$, which will follow from Theorem \ref{thm:supcritsec}, Theorem \ref{thm:critsec}, and  Lemma \ref{lem:suff}. Hence, from here on we may assume:
\begin{assumption}
\label{ass:cl}
\[
\abar=\textnormal{Cl}(\abar).
\]
\end{assumption}
Note that Remark \ref{rem:posetchar}, which is proven in Proposition \ref{prop:IIIimpliesIV}, does not require Assumption \ref{ass:cl}. The first extremals we need to characterize are those arising in the trivial case $|\mathcal N_=|=0$, which we dispose of in Theorem \ref{thm:trivsubcrit}. Assuming that $|\mathcal N_=|>0$, the characterization of \eqref{eq:Stanleyeq} is divided to three types of classes, \emph{subcritical}, \emph{supercritical}, and \emph{critical}. By subcritical we mean that $\mathcal \poly$ is subcritical. The supercritical and critical settings were defined in Definition \ref{def:ssc} and Definition \ref{def:posetcrit}.\\

The characterization of the \textbf{subcritical extremals} relies on the \emph{splitting mechanism} (Definition \ref{def:split} and Proposition \ref{prop:split}). The idea is that if $\mathcal \poly$ is truly subcritical, rather than critical,  we can reduce the problem to the extremals of a  poset with a shorter chain $\{x_i\}$. Arguing by induction, we then characterize the subcritical extremals (Theorem \ref{thm:subcrit}).\\

For the \textbf{supercritical extremals}, the starting point is Theorem \ref{thm:SvHComb} which yields that
 $|\mathcal N_{=}|^2=|\mathcal N_{-}||\mathcal N_{+}|$ holds, if, and only if, there exist $a\ge 0$ and $\mathrm v\in \R^{n-k}$ such that
\begin{align}
\label{eq:hsup}
h_{\poly_{\LL-1}}(\mathrm u)=h_{a\poly_{\LL}+\mathrm v}(\mathrm u) \quad \mbox{for all }(B,\mathcal \poly)\textnormal{-extreme normal directions } \mathrm u. 
\end{align}
The identity \eqref{eq:hsup} constitutes a system of equations (one equation for each $\mathrm u$) and the goal is to interpret these equations as combinatorial constraints on the poset $\abar$. Hence, the first important step is to find enough $(B,\mathcal \poly)$--extreme normal directions which can be described combinatorially. This is achieved in Section \ref{sec:ext} (Proposition \ref{prop:dir}(a-d)) by using the \emph{mixing} phenomenon (Section \ref{subsec:guide}). Once these directions are found in Section \ref{sec:ext}, Section \ref{sec:supercrit} is dedicated to plugging these directions back into \eqref{eq:hsup} and analyzing the outcomes. The second important step is to show that the scalar $a$ and the vector $\mathrm v$ in \eqref{eq:hsup} satisfy $a=1$ and $\mathrm v_j=0$ for certain $j$'s. The identity \eqref{eq:hsup} then further simplifies and provides the bulk of the desired characterization of the extremals (Theorem \ref{thm:supcritsec}). We explain in Section \ref{sec:supercrit} how to control $a$ and $\mathrm v$.\\

The starting point for the \textbf{critical extremals} is again Theorem \ref{thm:SvHComb}, but now we need to use its second part which states that
 $|\mathcal N_{=}|^2=|\mathcal N_{-}||\mathcal N_{+}|$ holds, if, and only if, there exist $a\ge 0,\,\mathrm v\in \R^{n-k}$, and a number $0\le d<\infty$ of $\mathcal \poly$-degenerate pairs $(P_1,Q_1),\ldots,(P_d,Q_d)$, such that 
\[
h_{\poly_{\LL-1}+\sum_{j=1}^dQ_j}(\mathrm u)=h_{a\poly_{\LL}+\mathrm v+\sum_{j=1}^dP_j}(\mathrm u) \quad \mbox{for all }(B,\mathcal \poly)\textnormal{-extreme normal directions } \mathrm u. 
\]
The presence of the degenerate pairs causes great difficulties (which are not just technical since, as we saw, new extremals do indeed arise for critical posets). The first key idea to resolve these problems is to find a sub-poset of $\abar$ on which we have more-or-less a supercritical behavior. From a geometric standpoint, this corresponds to finding a subspace $E^{\perp}$ such that 
\begin{align}
\label{eq:hcrit}
h_{\poly_{\LL-1}}(\mathrm u)=h_{a\poly_{\LL}+\mathrm v}(\mathrm u) \quad \mbox{for all }(B,\mathcal \poly)\textnormal{-extreme normal directions } \mathrm u\text{ which are contained in }E^{\perp}.
\end{align}
The identification of $E^{\perp}$ and its properties relies on the mixing properties of the \emph{maximal splitting pair} (Section \ref{lem:splitequiv}). Even after identifying $E^{\perp}$ we face the problem that \eqref{eq:hcrit} provides less constraints than \eqref{eq:hsup} due to the restriction to the subspace $E^{\perp}$. Hence, we cannot derive enough combinatorial constraints on $\abar$. The solution is to find even more  $(B,\mathcal \poly)$-extreme normal directions which were not needed for supercritical posets (Proposition \ref{prop:dir}(e-h)). With these new directions in hand, Section \ref{sec:crit} proceeds roughly as Section \ref{sec:supercrit} to show that $a=1$ and $\mathrm v_j=0$ for certain $j$'s. This description is an oversimplification since the situation is in fact much more delicate. It is precisely this delicacy which leads to the new extremals for critical posets.

\section{Notions of criticality}
\label{sec:notionsCrit}
In this section we start building our dictionary between convex geometry and combinatorics. The first building block is a correspondence between geometric and combinatorial notions of criticality, which will be used throughout this work. Section \ref{subsec:triv} starts with the easiest correspondence (Lemma \ref{lem:spancollec}), which connects the linear spans of polytopes in $\mathcal\poly$ with subsets of $\abar$. Consequently, we characterize the trivial extremals which appear when $|\mathcal N_=|=0$ (Theorem \ref{thm:trivsubcrit}). Section \ref{subsec:equivcrit} is dedicated to the equivalences between geometric and combinatorial notions of criticality (Proposition \ref{prop:critequiv}), and  their consequences on sharp-subcritical and sharp-critical collections (Lemmas \ref{lem:sharsubcrit}, \ref{lem:sharcrit}).

\subsection{The trivial extremals}
\label{subsec:triv}
We start with some notation. Given a convex body $C$  let $\aff(C)$ stand for the affine hull of $C$, and let $\lin(C)$ stand for the vector space obtained by the translation of $\aff(C)$ to the origin, i.e., $\lin(C):=\aff(C)-c_0=\sspan(C-c_0)$, for any $c_0\in C$. Given a collection $\mathcal C$ of convex bodies, it is immediate to see that
\begin{equation}
\label{eq:linspan}
\lin\left(\sum_{C\in \mathcal C}C\right)=\sspan\left((\lin(C_1),\ldots,\lin(C_{|\mathcal C|}\right)).
\end{equation}
The following lemma relates the combinatorics of subsets of $\aint$ to the linear spans of the polytopes in $\{\poly_i\}$. 
\begin{lemma}
\label{lem:spancollec}
Let $j_0:=-1< j_1<\cdots<j_p<k+1=:j_{p+1}$ and set
\[
\mathcal \poly':=(\underbrace{\poly_{j_1},\ldots,\poly_{j_1}}_{\kappa_1},\ldots,\underbrace{\poly_{j_p},\ldots,\poly_{j_p}}_{\kappa_p}),
\]
where $\kappa_1,\ldots, \kappa_p$ are positive integers. 
Then,
\[
\lin\left(\sum_{\poly\in \mathcal \poly'}\poly\right)=\R^{\bsg_{\{j_1,\ldots,j_p\}}},
\]
and, consequently,
\[
\dim\left(\sum_{\poly\in \mathcal \poly' }\poly\right)=n-k-\sum_{q=0}^p|\aint_{>x_{j_q+1},<x_{j_{(q+1)}}}|.
\]
\end{lemma}
\begin{proof}
Combining \eqref{eq:linspan} and \eqref{eq:Kiordpoly} shows that $\lin\left(\sum_{\poly\in \mathcal \poly'}\poly\right)=\R^{\bsg_{\{j_1,\ldots,j_p\}}}$. It follows that
\[
\dim\left(\sum_{\poly\in \mathcal \poly' }\poly\right)=\left|\bigcup_{q=1}^p\bsg_{j_q}\right|=\left|\bigcup_{q=1}^p\aint\backslash(\aint_{<x_{j_q}}\cup \aint_{>x_{j_q+1}})\right|=\left|\aint\Big\backslash\bigcap_{q=1}^p(\aint_{<x_{j_q}}\cup \aint_{>x_{j_q+1}})\right|.
\]
The proof is complete by \eqref{eq:capcup}, and by noting that the sets $\{\aint_{>x_{j_q+1},<x_{j_{(q+1)}}}\}_{q\in \llbracket 0,p\rrbracket}$ are disjoint.
\end{proof}
As a first application of Lemma \ref{lem:spancollec}, we dispose of the trivial extremals. Before doing so, we present the following definition which will be used throughout the paper.
\begin{definition}
\label{def:splitpair}
A pair $(r,s)$ is \emph{splitting} if $0\le r+1< s\le k+1$ and $(r+1,s)\neq (0,k+1)$. A splitting pair $(r,s)$ is an \emph{$\LL$-splitting pair} if $r+1<\LL<s$. 
\end{definition}

\begin{theorem}{\textnormal{(\textbf{Trivial extremals})}}
\label{thm:trivsubcrit}
We have $|\mathcal N_=|=0$ if, and only if, there exists a splitting pair $(r,s)$ such that
\[
|\abar_{>x_{r+1},<x_s}|> i_s-i_{r+1}-1.
\]
\end{theorem}

\begin{proof}
$~$

$\Longleftarrow$: 
Suppose there exists a splitting pair $(r,s)$  such that
\[
|\abar_{>x_{r+1},<x_s}|> i_s-i_{r+1}-1.
\]
Every $\sigma\in\mathcal N_=$ must satisfy $\sigma(z)\in \llbracket i_{r+1}+1, i_s-1\rrbracket$ for every $z\in \abar_{>x_{r+1},<x_s}$. Since $|\llbracket i_{r+1}+1, i_s-1\rrbracket|=(i_s-1)-(i_{r+1}+1)+1=i_s-i_{r+1}-1<|\abar_{>x_{r+1},<x_s}|$, we see that no such $\sigma$ can exist.\\

$\Longrightarrow$: If $|\mathcal N_{=}|=0$ then, by \eqref{eq:posetvolrep} and Lemma \ref{lem:mvpos}, there exist $0\le j_1<\cdots<j_p\le k$, and positive integers $\kappa_1,\ldots, \kappa_p$, with $\kappa_q\le i_{j_q+1}-i_{j_q}-1$  for $q\in [p]$, such that, with
\[
\mathcal \poly'=(\underbrace{\poly_{j_1},\ldots,\poly_{j_1}}_{\kappa_1},\ldots,\underbrace{\poly_{j_p},\ldots,\poly_{j_p}}_{\kappa_p})\subseteq\mathcal (\poly_{\LL-1},\poly_{\LL},\mathcal \poly),
\]
we have
\[
\dim\left(\sum_{\poly\in \mathcal \poly' }\poly\right)<|\mathcal \poly'|.
\]
Let $j_0:=-1,~ j_{p+1}:=k+1$ and use Lemma \ref{lem:spancollec} to get 
\[
\dim\left(\sum_{\poly\in \mathcal \poly' }\poly\right)=n-k-\sum_{q=0}^p|\aint_{>x_{j_q+1},<x_{j_{(q+1)}}}|.
\]
On the other hand,  
\begin{align*}
|\mathcal \poly'|&=\sum_{q=1}^p\kappa_q\le \sum_{q=1}^p[i_{j_q+1}-i_{j_q}-1]=n-k-i_{j_{p+1}}+i_{j_0+1}+k+1+\sum_{q=1}^p(i_{j_q+1}-i_{j_q}-1)\\
&=n-k-\left(\sum_{q=1}^{p+1}i_{j_q}\right)+\left(\sum_{q=0}^pi_{j_q+1}\right)+j_{p+1}-j_0-(p+1)\\
&=n-k-\left(\sum_{q=0}^pi_{j_{(q+1)}}\right)+\left(\sum_{q=0}^pi_{j_q+1}\right)+j_{p+1}-j_0-(p+1)\\
&=n-k-\sum_{q=0}^p(i_{j_{(q+1)}}-i_{j_q+1}-j_{q+1}+j_q+1).
\end{align*}
It follows that
\begin{align}
\label{eq:thm:trivsubcrittemp}
\sum_{q=0}^p(i_{j_{(q+1)}}-i_{j_q+1}-j_{q+1}+j_q+1)<\sum_{q=0}^p|\aint_{>x_{j_q+1},<x_{j_{(q+1)}}}|.
\end{align}
Since
\begin{align*}
|\aint_{>x_{j_q+1},<x_{j_{(q+1)}}}|&=1_{\{j_q+1<j_{(q+1)}\}}\,|\aint_{>x_{j_q+1},<x_{j_{(q+1)}}}|,\\
i_{j_{(q+1)}}-i_{j_q+1}-j_{q+1}+j_q+1&=1_{\{j_q+1<j_{(q+1)}\}}\,(i_{j_{(q+1)}}-i_{j_q+1}-j_{(q+1)}+j_q+1),
\end{align*}
the inequality \eqref{eq:thm:trivsubcrittemp} is equivalent to
\[
\sum_{q=0}^p1_{\{j_q+1<j_{(q+1)}\}}(i_{j_{(q+1)}}-i_{j_q+1}-j_{q+1}+j_q+1)<\sum_{q=0}^p1_{\{j_q+1<j_{(q+1)}\}}|\aint_{>x_{j_q+1},<x_{j_{(q+1)}}}|.
\]
Using
\[
1_{\{j_q+1<j_{(q+1)}\}}\,|\abar_{>x_{j_q+1},<x_{j_{(q+1)}}}\backslash \aint_{>x_{j_q+1},<x_{j_{(q+1)}}}|=j_{(q+1)}-j_q-2,
\]
we get that \eqref{eq:thm:trivsubcrittemp} is equivalent to
\[
\sum_{q=0}^p1_{\{j_q+1<j_{(q+1)}\}}(i_{j_{(q+1)}}-i_{j_q+1}-1)<\sum_{q=0}^p1_{\{j_q+1<j_{(q+1)}\}}|\abar_{>x_{j_q+1},<x_{j_{(q+1)}}}|.
\]
Hence, there must exist a pair $(j_q+1,j_{(q+1)})$, with $j_q+1<j_{(q+1)}$, such that 
\[
|\abar_{>x_{j_q+1},<x_{j_{(q+1)}}}|>i_{j_{(q+1)}}-i_{j_q+1}-1.
\]
Since $(j_q+1,j_{(q+1)})\neq (0,k+1)$, because $j_q+1=0\Rightarrow q=0$ so $j_{(q+1)}=j_1\le j_p<j_{p+1}=k+1$, we conclude that there exists a splitting pair $(r,s)$ such that
\[
|\abar_{>x_{r+1},<x_s}|> i_s-i_{r+1}-1.
\]
\end{proof}

\begin{remark}
 Theorem \ref{thm:trivsubcrit} is the same as the result of Chan, Pak, and Panova in \cite[Theorem 1.12]{CP22}, where it was proved using purely combinatorial arguments. 
\end{remark}

In light of Theorem \ref{thm:trivsubcrit} we assume from here on that $|\mathcal N_=|>0$. Note  that $|\mathcal N_=|>0$ implies,  by \eqref{eq:posetvolrep} and Lemma \ref{lem:mvpos}, that $\mathcal\poly$ is subcritical. To summarize:
\begin{assumption}
\label{ass:N=>0}
\[
|\mathcal N_=|^2=|\mathcal N_-||\mathcal N_+|,\quad|\mathcal N_=|>0,\quad\text{and}\quad \mathcal\poly\text{ is subcritical}.
\]
\end{assumption}
\begin{remark}
\label{rem:totorder}
For future reference, we note that under Assumption \ref{ass:N=>0}, $\abar$ cannot be totally ordered. Indeed, if  $\abar$ is totally ordered, then at least two elements in $\{|\mathcal N_=|,|\mathcal N_-|,|\mathcal N_+|\}$ are zero. But since $|\mathcal N_=|^2=|\mathcal N_-||\mathcal N_+|$, that would imply that $|\mathcal N_=|=0$.
\end{remark}

\subsection{Equivalences of criticality notions}
\label{subsec:equivcrit}
The next result is at the base of the correspondence between criticality notions in our geometric and  combinatorial settings, namely, the equivalence between Definition \ref{def:ssc} and Definition \ref{def:posetcrit}.
\begin{proposition}
\label{prop:critequiv}
Fix a nonnegative integer $c$. The following are equivalent.
\begin{enumerate}[(1)]
\item For any integer $p\ge 1$ and $j_0:=-1< j_1<\cdots<j_p<k+1=:j_{p+1}$ such that  $ i_{j_q+1}-i_{j_q}-1-1_{j_q\in\{\LL-1,\LL\}}$ are positive for any $q\in [p]$, it holds that with any
\[
\mathcal \poly':=(\underbrace{\poly_{j_1},\ldots,\poly_{j_1}}_{\kappa_1},\ldots,\underbrace{\poly_{j_p},\ldots,\poly_{j_p}}_{\kappa_p})\subseteq \mathcal \poly,
\]
where $\kappa_q\le i_{j_q+1}-i_{j_q}-1-1_{j_q\in\{\LL-1,\LL\}}$ are positive integers, we have
\[
\dim\left(\sum_{\poly\in \mathcal \poly' }\poly\right)\ge |\mathcal\poly'|+c.
\]

\item For any integer $p\ge 1$, $j_0:=-1< j_1<\cdots<j_p<k+1=:j_{p+1}$ such that $ i_{j_q+1}-i_{j_q}-1-1_{j_q\in\{\LL-1,\LL\}}$ are positive for any $q\in [p]$, it holds that
\begin{align*}
\sum_{q=0}^p1_{\{j_q+1<j_{(q+1)}\}}\,|\abar_{>x_{j_q+1},<x_{j_{(q+1)}}}|\le |\{q\in [p]:j_q\in\{\LL-1,\LL\}\}|-c+\sum_{q=0}^p1_{\{j_q+1<j_{(q+1)}\}}\,(i_{j_{(q+1)}}-i_{j_q+1}-1).
\end{align*}
\end{enumerate}

\end{proposition}
The proof of Proposition \ref{prop:critequiv} follows the logic of the proof of Theorem \ref{thm:trivsubcrit}, but it is more complicated since we now work with collections $\mathcal\poly'\subseteq\mathcal\poly$, rather than $\mathcal\poly'\subseteq(\poly_{\LL-1},\poly_{\LL},\mathcal\poly)$. This leads to the presence of the term $1_{j_{(q+1)}=\LL}+ 1_{j_q+1=\LL}$ in the proof below.
\begin{proof}[Proof of Proposition \ref{prop:critequiv}]
 Fix 
\[
\mathcal \poly':=(\underbrace{\poly_{j_1},\ldots,\poly_{j_1}}_{\kappa_1},\ldots,\underbrace{\poly_{j_p},\ldots,\poly_{j_p}}_{\kappa_p})\subseteq \mathcal \poly,
\]
where $p\ge 1$, $j_0:=-1< j_1<\cdots<j_p<k+1=:j_{p+1}$, and $0<\kappa_q\le i_{j_q+1}-i_{j_q}-1-1_{j_q\in\{\LL-1,\LL\}}$. By Lemma \ref{lem:spancollec},
\[
\dim\left(\sum_{\poly\in \mathcal \poly' }\poly\right)=n-k-\sum_{q=0}^p|\aint_{>x_{j_q+1},<x_{j_{(q+1)}}}|.
\]
On the other hand,  using $\LL\notin \{0,k+1\}$, and arguing as in the proof of Theorem \ref{thm:trivsubcrit}, 
\begin{align*}
|\mathcal \poly'|&=\sum_{q=1}^p\kappa_q\le \sum_{q=1}^p\left(i_{j_q+1}-i_{j_q}-1- 1_{j_q\in\{\LL-1,\LL\}}\right)\\
&=n-k-\sum_{q=0}^p\left(i_{j_{(q+1)}}-i_{j_q+1}-j_{q+1}+j_q+1+ 1_{j_{(q+1)}=\LL}+ 1_{j_q+1=\LL}\right).
\end{align*}
Hence, given $c$, we have that 
\[
\dim\left(\sum_{\poly\in \mathcal \poly' }\poly\right)< |\mathcal\poly'|+c,
\]
if and only if
\begin{align}
\label{eq:cdiminq}
\sum_{q=0}^p|\aint_{>x_{j_q+1},<x_{j_{(q+1)}}}|> -c+\sum_{q=0}^p\left(i_{j_{(q+1)}}-i_{j_q+1}-j_{q+1}+j_q+1+ 1_{j_{(q+1)}=\LL}+ 1_{j_q+1=\LL}\right).
\end{align}
Conversely, if \eqref{eq:cdiminq} holds, then we may take $\mathcal \poly'$ to be such that $\kappa_q=i_{j_q+1}-i_{j_q}-1- 1_{j_q\in\{\LL-1,\LL\}}$ for every $q$, to get $|\mathcal \poly'|=n-k-\sum_{q=0}^p\left(i_{j_{(q+1)}}-i_{j_q+1}-j_{q+1}+j_q+1+ 1_{j_{(q+1)}=\LL}+ 1_{j_q+1=\LL}\right)$. We may then conclude that $\dim\left(\sum_{\poly\in \mathcal \poly' }\poly\right)< |\mathcal\poly'|+c$. Hence, we get
\[
\dim\left(\sum_{\poly\in \mathcal \poly' }\poly\right)\ge |\mathcal\poly'|+c\quad \Longleftrightarrow \quad\eqref{eq:cdiminqconverse},
\]
where
\begin{align}
\label{eq:cdiminqconverse}
\sum_{q=0}^p|\aint_{>x_{j_q+1},<x_{j_{(q+1)}}}|\le -c+\sum_{q=0}^p\left(i_{j_{(q+1)}}-i_{j_q+1}-j_{q+1}+j_q+1+ 1_{j_{(q+1)}=\LL}+ 1_{j_q+1=\LL}\right).
\end{align}
Since 
\[
\sum_{q=0}^p(1_{j_{(q+1)}=\LL}+ 1_{j_q+1=\LL})=|\{q\in [p]:j_q\in\{\LL-1,\LL\}\}|,
\]
and
\begin{align*}
|\aint_{>x_{j_q+1},<x_{j_{(q+1)}}}|&=1_{\{j_q+1<j_{(q+1)}\}}\,|\aint_{>x_{j_q+1},<x_{j_{(q+1)}}}|,\\
i_{j_{(q+1)}}-i_{j_q+1}-j_{q+1}+j_q+1&=1_{\{j_q+1<j_{(q+1)}\}}(i_{j_{(q+1)}}-i_{j_q+1}-j_{(q+1)}+j_q+1),
\end{align*}
the inequality \eqref{eq:cdiminqconverse} is  equivalent to
\begin{align*}
&\sum_{q=0}^p1_{\{j_q+1<j_{(q+1)}\}}\,|\aint_{>x_{j_q+1},<x_{j_{(q+1)}}}|\\
&\le |\{q\in [p]:j_q\in\{\LL-1,\LL\}\}|-c
+\sum_{q=0}^p1_{\{j_q+1<j_{(q+1)}\}}\,\left(i_{j_{(q+1)}}-i_{j_q+1}-j_{q+1}+j_q+1\right),\\
&=|\{q\in [p]:j_q\in\{\LL-1,\LL\}\}|-c+\sum_{q=0}^p1_{\{j_q+1<j_{(q+1)}\}}\,\left(i_{j_{(q+1)}}-i_{j_q+1}-1-j_{(q+1)}+j_q+2\right).
\end{align*}
Using
\[
1_{\{j_q+1<j_{(q+1)}\}}\,|\abar_{>x_{j_q+1},<x_{j_{(q+1)}}}\backslash \aint_{>x_{j_q+1},<x_{j_{(q+1)}}}|=j_{(q+1)}-j_q-2,
\]
we find that \eqref{eq:cdiminqconverse} is equivalent to
\begin{align*}
&\sum_{q=0}^p1_{\{j_q+1<j_{(q+1)}\}}\,|\abar_{>x_{j_q+1},<x_{j_{(q+1)}}}|\le |\{q\in [p]:j_q\in\{\LL-1,\LL\}\}|-c+\sum_{q=0}^p1_{\{j_q+1<j_{(q+1)}\}}\,(i_{j_{(q+1)}}-i_{j_q+1}-1).
\end{align*}
\end{proof}

In contrast to Proposition \ref{prop:critequiv}, the next lemma, which treats the opposite inequality of Proposition \ref{prop:critequiv}, holds for a \emph{fixed} $\mathcal \poly'$.
\begin{lemma}
\label{lem:le}
Fix an integer $p\ge 1$ and $j_0:=-1< j_1<\cdots<j_p<k+1=:j_{p+1}$ such that  $ i_{j_q+1}-i_{j_q}-1-1_{j_q\in\{\LL-1,\LL\}}$ are positive for any $q\in [p]$. Let

\[
\mathcal \poly':=(\underbrace{\poly_{j_1},\ldots,\poly_{j_1}}_{\kappa_1},\ldots,\underbrace{\poly_{j_p},\ldots,\poly_{j_p}}_{\kappa_p})\subseteq \mathcal \poly,
\]
where $\kappa_q$ are integers such that $0<\kappa_q\le i_{j_q+1}-i_{j_q}-1-1_{j_q\in\{\LL-1,\LL\}}$ for all $q\in [p]$, be such that
\[
\dim\left(\sum_{\poly\in \mathcal \poly' }\poly\right)\le |\mathcal\poly'|+c.
\]
Then,
\begin{align*}
&\sum_{q=0}^p1_{\{j_q+1<j_{(q+1)}\}}\,|\abar_{>x_{j_q+1},<x_{j_{(q+1)}}}|\ge |\{q\in [p]:j_q\in\{\LL-1,\LL\}\}|-c+\sum_{q=0}^p1_{\{j_q+1<j_{(q+1)}\}}\,(i_{j_{(q+1)}}-i_{j_q+1}-1).
\end{align*}
\end{lemma}
\begin{proof}
We proceed as in the proof of Proposition \ref{prop:critequiv} and use 
\[
|\mathcal \poly'|=\sum_{q=1}^p\kappa_q\le \sum_{q=1}^p\left(i_{j_q+1}-i_{j_q}-1- 1_{j_q\in\{\LL-1,\LL\}}\right),
\]
to reason about a fixed collection $\mathcal \poly'$.
\end{proof}
As a consequence of Lemma \ref{lem:le}, we get the following combinatorial information about \emph{sharp} collections.
\begin{lemma}
\label{lem:critequiveq}
Fix $c\ge 0$, an integer $p\ge 1$, and $j_0:=-1< j_1<\cdots<j_p<k+1=:j_{p+1}$ such that  $ i_{j_q+1}-i_{j_q}-1-1_{j_q\in\{\LL-1,\LL\}}$ are positive for any $q\in [p]$. Suppose there exist
\[
\mathcal \poly':=(\underbrace{\poly_{j_1},\ldots,\poly_{j_1}}_{\kappa_1},\ldots,\underbrace{\poly_{j_p},\ldots,\poly_{j_p}}_{\kappa_p})\subseteq \mathcal \poly,
\]
where $\kappa_q$ are integers such that $0<\kappa_q\le i_{j_q+1}-i_{j_q}-1-1_{j_q\in\{\LL-1,\LL\}}$ for all $q\in [p]$, such that
\[
\dim\left(\sum_{\poly\in \mathcal \poly' }\poly\right)= |\mathcal\poly'|+c.
\]
Then,
\[
|\{q\in [p]:j_q\in\{\LL-1,\LL\}\}|\le c.
\]
\end{lemma}
\begin{proof}
The assumption $\dim\left(\sum_{\poly\in \mathcal \poly' }\poly\right)= |\mathcal\poly'|+c$ implies  $\dim\left(\sum_{\poly\in \mathcal \poly' }\poly\right)\le |\mathcal\poly'|+c$, so by Lemma \ref{lem:le},
\begin{align}
\label{eq:dimtemp1}
&\sum_{q=0}^p1_{\{j_q+1<j_{(q+1)}\}}\,|\abar_{>x_{j_q+1},<x_{j_{(q+1)}}}|\ge |\{q\in [p]:j_q\in\{\LL-1,\LL\}\}|-c+\sum_{q=0}^p1_{\{j_q+1<j_{(q+1)}\}}\,(i_{j_{(q+1)}}-i_{j_q+1}-1).
\end{align}
On the other hand, since $|\mathcal N_=|>0$, we have 
\begin{align}
\label{eq:dimtemp1.5}
1_{\{j_q+1<j_{(q+1)}\}}\,|\abar_{>x_{j_q+1},<x_{j_{(q+1)}}}|\le 1_{\{j_q+1<j_{(q+1)}\}}\,(i_{j_{(q+1)}}-i_{j_q+1}-1)
\end{align}
(because $|\llbracket i_{j_q+1}+1,i_{j_{(q+1)}}-1\rrbracket|\le i_{j_{(q+1)}}-i_{j_q+1}-1$),
so
\begin{align}
\label{eq:dimtemp2}
\sum_{q=0}^p1_{\{j_q+1<j_{(q+1)}\}}\,|\abar_{>x_{j_q+1},<x_{j_{(q+1)}}}|\le \sum_{q=0}^p1_{\{j_q+1<j_{(q+1)}\}}\,(i_{j_{(q+1)}}-i_{j_q+1}-1).
\end{align}
Combining \eqref{eq:dimtemp1} and  \eqref{eq:dimtemp2} we get
\[
|\{q\in [p]:j_q\in\{\LL-1,\LL\}\}|\le c.
\]
\end{proof}

We are now ready to characterize the sharp-(sub)critical collections. We start with the sharp-subcritical collections. 

\begin{lemma}{\textnormal{(\textbf{Sharp-subcritical collections})}}
\label{lem:sharsubcrit}
Fix an integer $p\ge 1$, and $j_0:=-1< j_1<\cdots<j_p<k+1=:j_{p+1}$ such that  $ i_{j_q+1}-i_{j_q}-1-1_{j_q\in\{\LL-1,\LL\}}$ are positive for any $q\in [p]$. Suppose that
\[
\mathcal \poly':=(\underbrace{\poly_{j_1},\ldots,\poly_{j_1}}_{\kappa_1},\ldots,\underbrace{\poly_{j_p},\ldots,\poly_{j_p}}_{\kappa_p})\subseteq \mathcal \poly,
\]
where $\kappa_q$ are integers such that $0<\kappa_q\le i_{j_q+1}-i_{j_q}-1-1_{j_q\in\{\LL-1,\LL\}}$ for all $q\in [p]$, is sharp-subcritical. Then,
\[
\forall ~ q\in [p]:\quad j_q\notin\{\LL-1,\LL\}\quad\text{and}\quad1_{\{j_q+1<j_{(q+1)}\}}\,|\abar_{>x_{j_q+1},<x_{j_{(q+1)}}}|=1_{\{j_q+1<j_{(q+1)}\}}\,(i_{j_{(q+1)}}-i_{j_q+1}-1).
\]
\end{lemma}
\begin{proof}
Take $c=0$ in Lemma \ref{lem:critequiveq} to get 
\begin{align}
\label{eq:c=0}
|\{q\in [p]:j_q\in\{\LL-1,\LL\}\}|=0.
\end{align}
Since $
\dim\left(\sum_{\poly\in \mathcal \poly' }\poly\right)\le |\mathcal\poly'|$, and $\mathcal \poly$ is subcritical, applying Lemma \ref{lem:le} and Proposition \ref{prop:critequiv}, with $c=0$, yields
\[
\sum_{q=0}^p1_{\{j_q+1<j_{(q+1)}\}}\,|\abar_{>x_{j_q+1},<x_{j_{(q+1)}}}|=\sum_{q=0}^p1_{\{j_q+1<j_{(q+1)}\}}\,(i_{j_{(q+1)}}-i_{j_q+1}-1).
\]
By \eqref{eq:dimtemp1.5}, it follows that, for every  $0\le j_q\le k+1$, 
\[
1_{\{j_q+1<j_{(q+1)}\}}\,|\abar_{>x_{j_q+1},<x_{j_{(q+1)}}}|=1_{\{j_q+1<j_{(q+1)}\}}\,(i_{j_{(q+1)}}-i_{j_q+1}-1).
\]
\end{proof}

We now turn to the sharp-critical collections. The assumption made in the following lemma does not follow automatically from the fact that $\mathcal \poly$ is sharp-critical.  Rather, we will be able to make this assumption only after Section \ref{sec:subcrit}, and the motivation behind this assumption can be found in Theorem \ref{thm:subcrit}. The proof, however, is similar in spirit to the rest of this section so it is included here.
\begin{lemma}{\textnormal{(\textbf{Sharp-critical collections})}}
\label{lem:sharcrit}
Suppose $|\abar_{>x_{r+1},<x_s}| \le i_s-i_{r+1}-2$ for every splitting pair $(r,s)$. Fix an integer $p\ge 1$, and $j_0:=-1< j_1<\cdots<j_p<k+1=:j_{p+1}$ such that $ i_{j_q+1}-i_{j_q}-1-1_{j_q\in\{\LL-1,\LL\}}$ are positive for any $q\in [p]$.
 Then, every
\[
\mathcal \poly':=(\underbrace{\poly_{j_1},\ldots,\poly_{j_1}}_{\kappa_1},\ldots,\underbrace{\poly_{j_p},\ldots,\poly_{j_p}}_{\kappa_p})\subseteq \mathcal \poly,
\]
 where $\kappa_q$ are integers such that $0<\kappa_q\le i_{j_q+1}-i_{j_q}-1-1_{j_q\in\{\LL-1,\LL\}}$ for all $q\in [p]$, satisfying
\[
\dim\left(\sum_{\poly\in \mathcal \poly' }\poly\right)=|\mathcal\poly'|+1,
\]
must be of the form
\[
\mathcal \poly'=(\mathcal \poly_0,\mathcal \poly_1,\ldots,\mathcal \poly_{r-1},\mathcal \poly_r,\mathcal \poly_s,\mathcal \poly_{s+1},\ldots,\mathcal \poly_k),
\]
where $(r,s)$ is an $\LL$-splitting pair satisfying 
\[
|\abar_{>x_{r+1},<x_s}|= i_s-i_{r+1}-2.
\]
\end{lemma} 
\begin{proof}
First note that $(j_q+1,j_{(q+1)})\neq (0,k+1)$ because $j_q+1=0\Rightarrow q=0$ so $j_{(q+1)}=j_1\le j_p<j_{p+1}=k+1$. The assumption $|\abar_{>x_{r+1},<x_s}| \le i_s-i_{r+1}-2$ for every splitting pair $(r,s)$ implies that 
\begin{align*}
&\sum_{q=0}^p1_{\{j_q+1<j_{(q+1)}\}}|\abar_{>x_{j_q+1},<x_{j_{(q+1)}}}|\le \sum_{q=0}^p1_{\{j_q+1<j_{(q+1)}\}}\,(i_{j_{(q+1)}}-i_{j_q+1}-2)\\
=&-|\{q\in [p]: j_q+1<j_{(q+1)}\}|+\sum_{q=0}^p1_{\{j_q+1<j_{(q+1)}\}}(i_{j_{(q+1)}}-i_{j_q+1}-1).
\end{align*}
On the other hand, since $
\dim\left(\sum_{\poly\in \mathcal \poly' }\poly\right)\le |\mathcal\poly'|+1$, applying Lemma \ref{lem:le} with $c=1$ yields
\begin{align}
\label{eq:lower}
\sum_{q=0}^p1_{\{j_q+1<j_{(q+1)}\}}\,|\abar_{>x_{j_q+1},<x_{j_{(q+1)}}}|\ge |\{q\in [p]:j_q\in\{\LL-1,\LL\}\}|-1+\sum_{q=0}^p1_{\{j_q+1<j_{(q+1)}\}}\,(i_{j_{(q+1)}}-i_{j_q+1}-1).
\end{align}
We conclude that
\begin{align}
\label{eq:twoterms}
 |\{q\in [p]:j_q\in\{\LL-1,\LL\}\}|+|\{q\in [p]: j_q+1<j_{(q+1)}\}|\le 1.
\end{align}
Since
\[
|\{q\in [p]: j_q+1<j_{(q+1)}\}|=0\quad\Longrightarrow \quad\{j_1,\ldots,j_p\}=\{1,\ldots, k\},
\]
we get 
\[
|\{q\in [p]: j_q+1<j_{(q+1)}\}|=0\quad \Longrightarrow \quad |\{q\in [p]:j_q\in\{\LL-1,\LL\}\}|=2.
\]
Hence, \eqref{eq:twoterms} can hold if, and only if, 
\[
 |\{q\in [p]:j_q\in\{\LL-1,\LL\}\}|=0\quad\text{and}\quad|\{q\in [p]: j_q+1<j_{(q+1)}\}|=1.
\]
It follows that
\[
\mathcal \poly'=(\mathcal \poly_0,\mathcal \poly_1,\ldots,\mathcal \poly_{r-1},\mathcal \poly_r,\mathcal \poly_s,\mathcal \poly_{s+1},\ldots,\mathcal \poly_k),
\]
where $(r,s)$ is an $\LL$-splitting pair. Finally, plugging in $ |\{q\in [p]:j_q\in\{\LL-1,\LL\}\}|=0$ into \eqref{eq:lower}, and using that $(r,s)$ is the only pair $(j_q,j_{(q+1)})$ satisfying  $j_q+1<j_{(q+1)}$, yields
\[
|\abar_{>x_{r+1},<x_s}|\ge -1+[i_s-i_{r+1}-1]=i_s-i_{r+1}-2.
\]
On the other hand, by assumption, $|\abar_{>x_{r+1},<x_s}| \le i_s-i_{r+1}-2$, so we conclude
\[
|\abar_{>x_{r+1},<x_s}|=i_s-i_{r+1}-2.
\]
\end{proof}
\begin{remark}
In the proof of Lemma \ref{lem:sharcrit} we only used the condition $\dim\left(\sum_{\poly\in \mathcal \poly' }\poly\right)\le|\mathcal\poly'|+1$, so the reader might wonder why we assume that $\mathcal \poly'$ is sharp-critical. By Assumption \ref{ass:N=>0}, the only other possibility would be for $\mathcal \poly'$ to be sharp-subcritical, but this is impossible by Lemma \ref{lem:sharsubcrit} and the assumption $|\abar_{>x_{r+1},<x_s}| \le i_s-i_{r+1}-2$ for every splitting pair $(r,s)$.
\end{remark}

\section{Splitting and the subcritical extremals}
\label{sec:subcrit}
In this section we introduce the \emph{splitting} mechanism for posets, which is connected to a reduction to lower dimensional extremals. Consequently, we characterize the subcritical extremals (Theorem \ref{thm:subcrit}).
To motivate the splitting mechanism recall that,  by Lemma \ref{lem:sharsubcrit}, we know that every sharp-subcritical collection
 \[
\mathcal \poly':=(\poly_{j_1},\ldots,\poly_{j_1},\ldots,\poly_{j_p},\ldots,\poly_{j_p})\subseteq \mathcal \poly,
\]
must satisfy
\[
\forall ~ q\in [p]:\quad j_q\notin\{\LL-1,\LL\}\quad\text{and}\quad1_{\{j_q+1<j_{(q+1)}\}}\,|\abar_{>x_{j_q+1},<x_{j_{(q+1)}}}|=1_{\{j_q+1<j_{(q+1)}\}}\,(i_{j_{(q+1)}}-i_{j_q+1}-1).
\]
Fix an index $j_q$ such that  $j_q\notin\{\LL-1,\LL\}$ and $j_q+1<j_{(q+1)}$, so that
\[
|\abar_{>x_{j_q+1},<x_{j_{(q+1)}}}|=i_{j_{(q+1)}}-i_{j_q+1}-1.
\]
Since $|\llbracket  i_{j_q+1}+1,i_{j_{(q+1)}}-1\rrbracket|=i_{j_{(q+1)}}-i_{j_q+1}-1$, we must have
\[
\abar_{\ge x_{j_q+1},\le x_{j_{(q+1)}}}\quad\overset{\text{bijection}}{\mapsto} \quad \llbracket  i_{j_q+1},i_{j_{(q+1)}}\rrbracket
\]
under any linear extension.  This means that the poset $\abar$ can be \emph{split} by factoring out the poset $\abar_{\ge x_{j_q+1},\le x_{j_{(q+1)}}}$, so that we are left with a poset with a shorter chain. We will show that $|\mathcal N_=|^2=|\mathcal N_-||\mathcal N_+|$ implies that equality holds in Stanley's inequalities \emph{also for the poset with the shorter chain}. We may then resort to our induction hypothesis that the extremals in the case where the chain size is $<k$ were already characterized.

\begin{remark}
\label{rem:splitproj}
The splitting mechanism described in this section can be viewed as a combinatorial equivalence of the projection formula for mixed volumes \cite[Theorem 5.3.1]{Sch14}. This is another building block of our dictionary between geometry and combinatorics.
\end{remark}

We now proceed to formalize the above splitting mechanism. 

\begin{definition}
\label{def:split}
The \emph{split} of $\abar$, based on a splitting pair $(r,s)$, is given by defining posets $\abar_1,\abar_2$ as 
\begin{align*}
&\abar_1:= \abar_{\ge x_{r+1},\le x_s}\quad\mbox{and}\quad \abar_2:=(\abar\backslash\abar_1)\cup \{x\},
\end{align*}
where the relations for $x$ are defined via $x\ast z$, for $\ast\in\{<,>\}$ and $z\in \abar\backslash\abar_1$, if, and only if, there exists $w\in\abar_1$ such that $w\ast z$.\footnote{The new element $x$ should be thought of as a compression of $\abar_1$ into one element, namely $x$. The relations for $x$ are consistent since we cannot have $w_1<z<w_2$ for $w_1,w_2\in \abar_1, z\in \abar\backslash\abar_1$ because this would imply that $x_r\le z\le x_s$, and hence $z\in  \abar_1$, which is a contradiction.}
\end{definition}

Let $(r,s)$ be a splitting pair satisfying $\LL\notin\{r+1,s\}$. We will define the analogues of $\mathcal N_-,\mathcal N_=,\mathcal N_+$ associated with the posets $\abar_1,\abar_2$. This requires distinguishing between two cases: (1) $x_{\LL}\in\{x_{r+2},\ldots,x_{s-1}\}$ and (2) $x_{\LL}\in\{x_1,\ldots,x_r\}\cup\{x_{s+1},\ldots,x_k\}$; note that $x_{\LL}\notin\{x_{r+1},x_s\}$ by assumption.\footnote{We use the convention $\{x_a,\ldots, x_z\}=\varnothing$ when $z<a$; e.g., $\{x_1,\ldots, x_{r-2}\}=\varnothing$ when $r=0$.} For $\iota=1,2$ let
\[
\mathcal N^{\iota}:=\{\mbox{bijections }\sigma:\abar_{\iota}\to [|\abar_{\iota}|]: w\le z\Rightarrow \sigma(w)\le\sigma(z)~\forall~ w,z\in\abar_i\},
\]
and, given  $\s\in\{-,=,+\}$, let $1_{\s}:=1_{\{\s\text{ is }+\}}-1_{\{\s\text{ is }-\}}$.\\

\emph{Case (1).} For $\s\in\{-,=,+\}$ set
\begin{align*}
&\mathcal N^1_{\s}:=\{\sigma\in \mathcal N^1:\sigma(x_j)=i_j-i_{r+1}+1_{j=\LL}1_{\s}\mbox{ for }j\in \llbracket r+1,s\rrbracket\},\\
&\mathcal N_{\s}^2:=\{\sigma\in \mathcal N^2:\sigma(x_j)=i_j \mbox{ for }j\in \llbracket 0,r\rrbracket, ~\sigma(x)=i_{r+1}, \mbox{ and } \sigma(x_j)=i_j -(i_s-i_{r+1})\mbox{ for }j\in \llbracket s+1,k+1\rrbracket\};
\end{align*}
note that the definition of $\mathcal N_{\s}^2$ is independent of $\s$.\\

\emph{Case (2).} For $\s\in\{-,=,+\}$ set
\begin{align*}
\mathcal N^1_{\s}:=\{\sigma\in \mathcal N^1:\sigma(x_j)=i_j-i_{r+1}\mbox{ for }j\in \llbracket r+1,s\rrbracket\},
\end{align*}
and
\begin{align*}
\mathcal N_{\s}^2:=\{\sigma\in \mathcal N^2&: \sigma(x_j)=i_j+1_{\{j=\LL\}}1_{\s}\mbox{ for }j\in \llbracket 0,r\rrbracket, ~\sigma(x)=i_{r+1},\\
& \mbox{ and } \sigma(x_j)=i_j -(i_s-i_{r+1})+1_{j=\LL}1_{\s}\mbox{ for }j\in \llbracket s+1,k+1\rrbracket\};
\end{align*}
note that the definition of $\mathcal N_{\s}^1$ is independent of $\s$.

Before exploiting the splitting mechanism we start with a quick observation. 
\begin{lemma}
\label{lem:subcritcplus}
For every splitting pair $(r,s)$,
\begin{equation}
\label{eq:subcritcplus}
 |\abar_{>x_{r+1},<x_s}|\le   i_s-i_{r+1}-1-1_{r+1=\LL}-1_{s=\LL}.
 \end{equation}
\end{lemma}
\begin{proof}
The converse of Theorem \ref{thm:trivsubcrit} yields
\[
 |\abar_{>x_{r+1},<x_s}|\le   i_s-i_{r+1}-1\quad\mbox{for every splitting pair } (r,s).
 \]
Hence, it suffices to consider the case where either $r+1=\LL$ or $s=\LL$. Suppose $r+1=\LL$; the case $s=\LL$ is proven analogously. Then, every $\sigma\in\mathcal N_+$ (which must exist since $|\mathcal N_=|>0\Rightarrow |\mathcal N_+|>0$ as $|\mathcal N_=|^2=|\mathcal N_-||\mathcal N_+|$) satisfies $\sigma(x_{r+1})=i_{r+1}+1$ and $\sigma(x_s)=i_s$. Hence, given $z\in \abar_{>x_{r+1},<x_s}$, the number of available spots for $\sigma(z)$ is  $|\llbracket i_{r+1}+2,i_s-1\rrbracket|=(i_s-1)-( i_{r+1}+2)+1=i_s-i_{r+1}-2$.
\end{proof}
\begin{proposition}
\label{prop:split}
Fix a splitting pair $(r,s)$ satisfying $\LL\notin\{r+1,s\}$, and let $\abar_1,\abar_2$ be the split based on $(r,s)$.  One of the following must occur:
\begin{enumerate}[(i)]
\item $|\mathcal N_=^{\iota}|^2=|\mathcal N_-^{\iota}||\mathcal N_+^{\iota}|$ for every $\iota\in\{1,2\}$. 
\item  $|\abar_{>x_{r+1},<x_s}|\le i_s-i_{r+1}-2$.
\end{enumerate}
\end{proposition}

\begin{proof}
We will prove the proposition under the assumption that case (1) occurs; the proof for case (2) is analogous. Note that under case (1) we trivially have $|\mathcal N_=^2|^2=|\mathcal N_-^2||\mathcal N_+^2|$ since $\mathcal N_{\s}^2$ is independent of $\s$. 

It suffices to show that if (ii) is false then (i) is true. This will be proven by showing that if (ii) is false, then, for any $\s\in\{-,=,+\}$,
\begin{equation}
\label{eq:bijections}
|\mathcal N_{\s}|=|\mathcal N_{\s}^1||\mathcal N_{\s}^2|,
\end{equation}
where we recall that $\mathcal N_{\s}^2$ is independent of $\s$. Plugging \eqref{eq:bijections} into $|\mathcal N_=|^2=|\mathcal N_-|\mathcal N_+|$ gives $|\mathcal N_=^1|^2|\mathcal N_{\s}^2|^2=|\mathcal N_-^1||\mathcal N_+^1||\mathcal N_{\s}^2|^2$. Canceling $|\mathcal N_{\s}^2|$ on both sides ($|\mathcal N_{\s}^2|>0$ since $|\mathcal N_=|>0$) gives (i). 

We now turn to prove \eqref{eq:bijections} under the assumption that (ii) is false. By \eqref{eq:subcritcplus}, (ii) being false is equivalent to $|\abar_{>x_{r+1},<x_s}|= i_s-i_{r+1}-1$, i.e., $|\abar_1|=i_s-i_{r+1}+1$. We will  prove \eqref{eq:bijections}  by constructing a bijection $b:\mathcal N_{\s}\to \mathcal N_{\s}^1\times \mathcal N_{\s}^2$ for $\s\in \{-,=,+\}$. Fix $\s\in \{-,=,+\}$ and define a map $b$ via $b=(b_1,b_2)$, with $b_1:\mathcal N_{\s}\to \mathcal N_{\s}^1, ~b_2:\mathcal N_{\s}\to \mathcal N_{\s}^2$, where we set, for each $\sigma\in \mathcal N_{\s}$,
\begin{align*}
&\mbox{For }z\in \abar_1:\quad b_1(\sigma)(z)=\sigma(z)-i_{r+1},\\
&\mbox{For }z\in \abar_2:\quad b_2(\sigma)(z)=
\begin{cases}
\sigma(z)\mbox{ if }\sigma(z)\in \llbracket 0,i_{r+1}-1\rrbracket,\\
i_{r+1} \mbox{ if } z=x,\\
\sigma(z)-(i_s-i_{r+1})\mbox{ if }\sigma(z)\in \llbracket i_s+1, n+1\rrbracket.
\end{cases}
\end{align*}
We will first check that, given $\sigma\in \mathcal N_{\s}$, $b_1(\sigma)\in \mathcal N_{\s}^1$ and $b_2(\sigma)\in \mathcal N_{\s}^2$. We will then construct a map $b':\mathcal N_{\s}^1\times \mathcal N_{\s}^2\to \mathcal N_{\s}$ and show that $b\circ b'=b'\circ b=\Id$, completing the proof. That $b_1(\sigma)\in \mathcal N_{\s}^1$ and $b_2(\sigma)\in \mathcal N_{\s}^2$ follows from the definitions of $\mathcal N_{\s}^1,\mathcal N_{\s}^2$ and the fact that $\sigma\in \mathcal N_{\s}$. The map $b':\mathcal N_{\s}^1\times \mathcal N_{\s}^2\to \mathcal N_{\s}$ is defined by taking $\sigma_{\iota}\in \mathcal N_{\s}^{\iota}$, for $\iota=1,2$, and setting, for $z\in \abar$,
\begin{align*}
b'(\sigma_1,\sigma_2)(z)=
\begin{cases}
\sigma_2(z)&\mbox{ if }z\in \abar_2\mbox{ and }\sigma_2(z)\in \llbracket 0,i_{r+1}-1\rrbracket\\
\sigma_1(z)+i_{r+1}&\mbox{ if }z\in \abar_1\\
\sigma_2(z)+(i_s-i_{r+1})&\mbox{ if }z\in \abar_2\mbox{ and }\sigma_2(z)\in \llbracket i_{r+1}+1,|\abar_2|\rrbracket
\end{cases}.
\end{align*}
To see that $b'(\sigma_1,\sigma_2)\in  \mathcal N_{\s}$ we first need to check that given $z<w$ we have $b'(\sigma_1,\sigma_2)(w)< b'(\sigma_1,\sigma_2)(z)$. If $w,z\in\abar_1$ or $w,z\in\abar_2$, this follows from $\sigma_i\in\mathcal N^{\iota}_{\s}$, for $\iota=1,2$, so it remains to check $w\in\abar_1,~z\in\abar_2$ and $w\in\abar_2, ~z\in\abar_1$; we check the first case and the second case is analogous. Suppose that $w\in\abar_1$ and $z\in\abar_2$. Then, we must have $\sigma_2(z)\in \llbracket 0,i_{r+1}-1\rrbracket$ since, by the definition of $x$, $w>z\Rightarrow x>z$ and $\sigma_2(x)=i_{r+1}$. Hence, $b'(\sigma_1,\sigma_2)(w)=\sigma_1(w)+i_{r+1}> \sigma_2(z)=b'(\sigma_1,\sigma_2)(z)$. 
Now that we know that $b'(\sigma_1,\sigma_2)$ respects the relations of $\abar$, in order to show that $b'(\sigma_1,\sigma_2)\in  \mathcal N_{\s}$, it remains to check that $b'(\sigma_1,\sigma_2)(x_j)=i_j+1_{j=\LL}1_{\s}$ for all $1\le j\le k$. This follows immediately from the definitions of $\mathcal N_{\s}^{\iota}$ for $\iota\in\{1,2\}$ and $\s\in \{-,=,+\}$.
Finally, that $b\circ b'=b'\circ b=\Id$ follows from the construction of $b$ and $b'$.
\end{proof}
The next result provides a \emph{geometric} characterization under which the case in Proposition \ref{prop:split}(i) occurs.
\begin{lemma}
\label{lem:split}
Let $\mathcal\poly '\subseteq \mathcal\poly$ be a sharp subcritical collection. Then, there exists a splitting pair $(r,s)$ satisfying $\LL\notin\{r+1,s\}$, with a corresponding split $\abar_1,\abar_2$, such that $\poly_r,\poly_s\in\mathcal\poly'$ and $|\mathcal N_=^{\iota}|^2=|\mathcal N_-^{\iota}||\mathcal N_+^{\iota}|$ for every $\iota\in\{1,2\}$. 
\end{lemma}
\begin{proof}
By Lemma \ref{lem:sharsubcrit}, 
\[
\mathcal \poly'=(\mathcal\poly_{j_1},\ldots,\mathcal\poly_{j_p}),
\]where $j_0:=-1$, $0\le j_1<\cdots<j_p\le k$, $j_{p+1}:=k+1$, $\kappa_q\le i_{j_q+1}-i_{j_q}-1-1_{j_q\in\{\LL-1,\LL\}}$, and $p\in [n-k-2]$,
must satisfy
\[
\forall ~ q\in [p]:\quad j_q\notin\{\LL-1,\LL\}\quad\text{and}\quad1_{\{j_q+1<j_{(q+1)}\}}\,|\abar_{>x_{j_q+1},<x_{j_{(q+1)}}}|=1_{\{j_q+1<j_{(q+1)}\}}\,[i_{j_{(q+1)}}-i_{j_q+1}-1].
\]
Note that, for any $0\le q\le p$, $(j_q+1,j_{(q+1)})\neq (0,k+1)$. Indeed, for the latter to occur we need to have $p=1$ and $q=0$, but then $(j_0+1,j_{(0+1)})= (0,j_p)\neq (0,k+1)$ as $j_p<k+1$. We now show that there exists $0\le q'\le p$ such that $(j_{q'}+1,j_{(q'+1)})$ is a splitting pair. Indeed, if not, then $j_q+1=j_{(q+1)}$ for every $0\le q\le p$ so we get $j_1=0,j_2=1,\ldots ,j_{p+1}=k+1$ which contradicts $j_q\notin\{\LL-1,\LL\}$. Setting $r:=j_{q'},\, s:=j_{(q'+1)}$, we get a splitting pair $(r,s)$ such that $\poly_r,\poly_s\in \mathcal\poly'$  and $|\abar_{>x_{r+1},<x_s}|= i_s-i_{r+1}-1$. By Proposition \ref{prop:split}, we must have $|\mathcal N_=^{\iota}|^2=|\mathcal N_-^{\iota}||\mathcal N_+^{\iota}|$ for every $\iota\in\{1,2\}$. 
\end{proof}

Using Lemma \ref{lem:split}, the characterization of the subcritical extremals of Stanley's inequalities now follows.
\begin{theorem}{\textnormal{(\textbf{Subcritical extremals})}}
\label{thm:subcrit}
$~$

Suppose that $\mathcal\poly$ has a sharp-subcritical collection. Then there exists a splitting pair $(r,s)$ such that the associated posets split $\abar_1,\abar_2$ satisfies  $|\mathcal N_{=}^{\iota}|^2=|\mathcal N_{-}^{\iota}||\mathcal N_{+}^{\iota}|$ for every $\iota\in\{1,2\}$. 
\end{theorem}
Our induction hypothesis Assumption \ref{ass:induct} is that Theorem \ref{thm:supcrit} and Theorem \ref{thm:crit} hold for $k-1$. Hence, without loss of generality we may assume from now on that
\begin{align}
\label{eq:ass}
\text{For all splits }\abar_1,\abar_2: \quad  |\mathcal N_{=}^{\iota}|^2\neq|\mathcal N_{-}^{\iota}||\mathcal N_{+}^{\iota}|\quad \forall~\iota\in\{1,2\}.
\end{align}
By Theorem \ref{thm:subcrit}, the assumption \eqref{eq:ass} implies that $\mathcal \poly$ is critical. Further, by Proposition \ref{prop:split},
\[
|\abar_{>x_{r+1},<x_s}|\le i_s-i_{r+1}-2\quad\text{for every splitting pair $(r,s)$ satisfying $\LL\notin\{r+1,s\}$},
\]
so using in addition Lemma \ref{lem:subcritcplus}, we get
\[
|\abar_{>x_{r+1},<x_s}|\le i_s-i_{r+1}-2\quad\text{for every splitting pair $(r,s)$}.
\]
Putting everything together we assume from now on:
\begin{assumption}
\label{ass:crit}
The collection $\mathcal \poly$ is critical and
\[
|\abar_{\ge x_{r+1},\le x_s}|\le i_s-i_{r+1}\quad\mbox{for every splitting pair $(r,s)$}.
\]
\end{assumption}

\section{Mixing}
\label{sec:beyond}
Under the current assumptions, we know that $\abar$ cannot be totally ordered (Remark \ref{rem:totorder}). In this section, we develop the notion of \emph{mixing} which takes advantage of the fact that $\abar$ must have some incomparable elements. The level of mixing will depend on the criticality notions developed in Section \ref{sec:notionsCrit}, which will be further developed in the current section. We begin with Section \ref{subsec:towardlin} which characterizes the locations where elements of the poset can be placed. We then introduce in Section \ref{subsec:guide} the notions of criticality and maximality for splitting pairs. Finally, Section \ref{subsec:place} provides information on the mixing properties of  splitting pairs.

\subsection{Range}
\label{subsec:towardlin}
A fixed element $y\in\aint$ can only be placed in a limited number of locations under any linear extension. For example, if $\aint$ is totally ordered, there would be only one such location. We start by defining a few quantities associated to $y$ which will provide information on the possible placements of $y$ under linear extensions.

\begin{definition}
\label{def:lmu}
Given $y\in\aint$ let $i_{\max}(y)$ be the maximum index such that $y>x_{i_{\max}(y)}$ and let $i_{\min}(y)$ be the minimum index such that $y<x_{i_{\min}(y)}$. Set 
\begin{align*}
l_{\s}(y):=\max_{r\le i_{\max}(y)}(i_r^\circ+|\abar_{>x_r,\le y}|)\quad\text{and}\quad u_{\s}(y):=\min_{s\ge i_{\min}(y)}(i_s^\circ-|\abar_{\ge y,<x_s}|),
\end{align*}
where 
\[
i_j^\s := i_j+1_{j=\LL}1_{\s},
\]
and let
\begin{align*}
m^{\s}_{\min}(y):=\min_{\sigma\in\mathcal N_{\s}}\sigma(y)\quad\text{and}\quad m^{\s}_{\max}(y):=\max_{\sigma\in\mathcal N_{\s}}\sigma(y).
\end{align*}

\end{definition}
Note that $i_j^\s$ is the location where $x_j$ is placed under every linear extension in $\mathcal N_{\s}$.  Hence, for any choice of $r\le i_{\max}(y)$ (res. $s\ge i_{\min}(y)$), $y$ must be placed at a location at least as large (res. small) as $i_r^\circ+|\abar_{>x_r,\le y}|$ (res. $i_s^\circ-|\abar_{\ge y,<x_s}|$).

Definition \ref{def:lmu} immediately implies the following  relations between $l_{\s}$ (res. $u_{\s}$) for $\s\in\{-,=,+\}$:
\begin{lemma}
\label{lem:UL}
Fix $y\in\aint$. Then,\\

\begin{enumerate}[(i)]
\item $l_{=}(y)-1\le l_{-}(y)\le l_{=}(y)\le l_{+}(y)\le l_{=}(y)+1$.\\

\item If $i_{\max}(y)<\LL$, then $l_{-}(y)=l_{=}(y)=l_{+}(y)$.\\

\item $u_{=}(y)-1\le u_{-}(y)\le u_{=}(y)\le u_{+}(y)\le u_{=}(y)+1$.\\

\item If $i_{\min}(y)>\LL$, then $u_{-}(y)=u_{=}(y)=u_{+}(y)$.
\end{enumerate}
\end{lemma}
The next result provides necessary and sufficient conditions for an element of the poset to be placed at a specific location under linear extensions.
\begin{lemma}
\label{lem:range}
Fix $y\in \aint$, $\s\in\{-,=,+\}$, and $i\in [n]$. There exists $\sigma\in\mathcal N_{\s}$ with $\sigma(y)=i$ if, and only if, $i\in\llbracket l_{\s}(y),u_{\s}(y)\rrbracket$ and $i\neq i_m^\s$ for any $m\in [k]$.
\end{lemma}
\begin{proof}
$~$

$\Longrightarrow$: Fix $\sigma\in\mathcal N_{\s}$ such that $\sigma(y)=i$. Since $y\neq x_m$ for all $m\in [k]$ it follows that $i\neq i_m^\s$. We now show $i\le u_{\s}(y)$; the argument for $i\ge l_{\s}(y)$ is analogous. Given any $s\ge i_{\min}(y)$, every element $z\in \abar_{>y,<x_s}$ must satisfy $i=\sigma(y)<\sigma(z)<\sigma(x_s)$. Hence, $\sigma(z)$ can take on only $\sigma(x_s)-i-1$ possible values, which means that $|\abar_{>y,<x_s}|\le \sigma(x_s)-i-1$. In other words, $i\le \sigma(x_s)-|\abar_{\ge y,<x_s}|=i_s^\s-|\abar_{\ge y,<x_s}|$. The latter holds for any $s\ge i_{\min}(y)$ which shows $i\le u_{\s}(y)$. \\

$\Longleftarrow$: The assumption $i\neq i_m^\s$ for any $m\in [k]$ implies that we can choose $m\in [k]$ such that $i_m^\s<i<i_{m+1}^\s$. Consider the poset $\abar':=\abar$ with the relabeling
\begin{align*}
&x_j'=x_j~\text{for } j\in \llbracket 1,m\rrbracket, \quad x_{m+1}'=y,\quad x_j'=x_{j-1}~\text{for } j\in \llbracket m+2,k+1\rrbracket,\\
&i_j'=i_j^\s~\text{for } j\in \llbracket 1,m\rrbracket, \quad i_{m+1}'=i,\quad i_j'= i_{j-1}^\s~\text{for } j\in \llbracket m+2,k+1\rrbracket.
\end{align*}
To complete the proof it suffices to show that there exists a linear extension $\sigma'$ of $\abar'$ satisfying $\sigma'(x_j')=i_j'$ for all $j\in\llbracket 1,k+1\rrbracket$. By Theorem \ref{thm:trivsubcrit}, it suffices to show that
\begin{equation}
\label{eq:temp}
 |\abar'_{>x_{r+1}',<x_s'}|\le i_s'-i_{r+1}'-1\quad\text{for all }0\le r+1< s\le k+1. 
\end{equation}
When $r+1\neq m+1,s\neq m+1$, \eqref{eq:temp} holds by the assumption $|\mathcal N_{\s}|>0$ for all $\s\in\{-,=,+\}$ and Theorem \ref{thm:trivsubcrit}. The case $r+1=m+1=s$ is impossible since $r+1<s$. It remains to check the cases $r+1=m+1,s\neq m+1$ and $r+1\ne m+1,s= m+1$. We verify \eqref{eq:temp} in the case $s=m+1$; the proof for the case $r+1=m+1$ is analogous. When $s=m+1$, \eqref{eq:temp} is equivalent to 
\begin{equation}
\label{eq:temp1}
|\abar_{>x_{r+1},<y}|=|\abar'_{>x_{r+1}',<x_s'}|\le i_s'-i_{r+1}'-1=i-i_{r+1}^\s-1.
\end{equation}
When $r+1\le i_{\max}(y)$, \eqref{eq:temp1} holds since, by assumption, $i-i_{r+1}^{\s}-1\ge l_{\s}(y)-i_{r+1}^{\s}-1$, so \eqref{eq:temp1} holds by the definition of $l_{\s}(y)$. When $i_{\max}(y)<r+1<s=m+1$, $\abar_{>x_{r+1},<y}=\varnothing$ because if there exists $x_{r+1}<z<y$, that would imply $x_{r+1}<y$, which contradicts the maximality of $i_{\max}(y)$. Hence, \eqref{eq:temp1} is equivalent to $0\le i-i_{r+1}^{\s}-1$, which holds since  $i_{r+1}^\s\le i_m^\s< i$, where the last inequality holds by the definition of $m$.
\end{proof}
Lemma \ref{lem:range} immediately implies:
\begin{corollary}
\label{cor:range}
Fix $y\in \aint$ and $\s\in\{-,=,+\}$. Then,
\[
l_{\s}(y)\le m^{\s}_{\min}(y)\quad\text{and}\quad m^{\s}_{\max}(y)\le u_{\s}(y).
\]
\end{corollary}

A second corollary of Lemma \ref{lem:range} is the proof of Remark \ref{rem:posetchar}. Note that Assumption \ref{ass:cl} is not needed for the following result.

\begin{proposition}
\label{prop:IIIimpliesIV}
The condition in Theorem \ref{thm:supcrit}(iii) is equivalent to 
\[
\forall\,y<x_{\LL}~\exists\, s(y)\in\llbracket 0,k+1\rrbracket \text{ s.t. } y<x_{s(y)}\text{ and }|\abar_{>y, <x_{s(y)}}|>i_{s(y)}-i_{\LL},
\]
and
\[
\forall\,y>x_{\LL}~\exists\, r(y)\in\llbracket 0,k+1\rrbracket \text{ s.t. } y>x_{r(y)}\text{ and }|\abar_{>x_{r(y)},<y}|>i_{\LL}-i_{r(y)}.
\]
\end{proposition}
\begin{proof}
By Lemma \ref{lem:suff}(a), the conditions in Theorem \ref{thm:supcrit}(iii) are equivalent to: $ \sigma^{-1}(i_{\LL}-1)\incomp x_{\LL}\text{ and } \sigma^{-1}(i_{\LL}+1)\incomp x_{\LL} ~\forall \,\sigma\in\mathcal N_=$. We start by showing that
\begin{align*}
&\forall y<x_{\LL}~\exists\, s(y)\in\llbracket 0,k+1\rrbracket \text{ s.t. } y<x_{s(y)}\text{ and }|\abar_{>y,<x_{s(y)}}|>i_{s(y)}-i_{\LL}\\
&\Longleftrightarrow \\
&\sigma^{-1}(i_{\LL}-1)\incomp x_{\LL}~\forall \,\sigma\in\mathcal N_=; 
\end{align*}
The equivalence $\forall\,y>x_{\LL}~\exists\, r(y)\in\llbracket 0,k+1\rrbracket \text{ s.t. } y>x_{r(y)}\text{ and }|\abar_{>x_{r(y)},<y}|>i_{\LL}-i_{r(y)}\Longleftrightarrow \sigma^{-1}(i_{\LL}+1)\incomp x_{\LL}~\forall \,\sigma\in\mathcal N_=$ is analogous.

Indeed, the statement $ \sigma^{-1}(i_{\LL}-1)\incomp x_{\LL} ~\forall\,\sigma\in\mathcal N_=$ is equivalent to the statement that for all $y<x_{\LL}$, there exists no $\sigma\in\mathcal N_=$ such that $\sigma(y)=i_{\LL}-1$. We will show that the latter is equivalent to $u_=(y)<i_{\LL}-1$, which completes the proof. To see this equivalence, note that if $u_=(y)<i_{\LL}-1$, then Lemma \ref{lem:range} implies that  exists no $\sigma\in\mathcal N_=$ such that $\sigma(y)=i_{\LL}-1$. Conversely, suppose there exists no $\sigma\in\mathcal N_=$ such that $\sigma(y)=i_{\LL}-1$, so, by Lemma \ref{lem:range}, $i_{\LL}-1\neq \llbracket l_=(y),u_=(y)\rrbracket$. Note that, by Lemma \ref{lem:range}, $u_=(y)\le i_{\LL}-1$ as $y<x_{\LL}$. Hence, the possibility of $i_{\LL}-1< l_=(y)\le u_=(y)$ cannot occur, which means that $i_{\LL}-1\neq \llbracket l_=(y),u_=(y)\rrbracket\Rightarrow u_=(y)<i_{\LL}-1$, as claimed.
\end{proof}

\subsection{Introduction to mixing}
\label{subsec:guide}
When $\abar$ is totally ordered we have, for any splitting pair $(r,s)$,
\[
\abar_{\ge x_{r+1},\le x_s}\overset{\text{bijection}}{\mapsto} \llbracket i_{r+1}, i_s \rrbracket
\]
under any linear extension $\sigma\in\bigcup_{\s\in\{-,=,+\}}\mathcal N_{\s}$. But under the current assumptions, $\abar$ is not totally ordered (Remark \ref{rem:totorder}), which means that a certain amount of \emph{mixing} must occurs; see Definition \ref{def:mixed_element} for a precise statement. In Section \ref{subsec:place} we will show that there is at least one mixed element (Lemma \ref{lem:splitequiv}) for any splitting pair $(r,s)$. When the splitting pair is in addition an $\LL$-splitting pair we characterize the exact number of mixed element, which depends on the criticality level of the pair:
\begin{definition}
\label{def:splitpaircrit}
An $\LL$-splitting pair $(r,s)$ is \emph{supercritical} if $\mathcal\poly':=(\mathcal \poly_0,\ldots,\mathcal \poly_r,\mathcal \poly_s,\ldots,\mathcal \poly_k)$ satisfies $\dim\left(\sum_{\poly\in \mathcal \poly'}\poly\right)\ge |\mathcal \poly'|+2$, and is \emph{sharp-critical} if $\dim\left(\sum_{\poly\in \mathcal \poly'}\poly\right)= |\mathcal \poly'|+1$.
\end{definition}
We show in Section \ref{subsec:place} how the above notion of criticality is related to the number of mixed elements (Lemma  \ref{lem:generalmixed}). The sharp-critical $\LL$-splitting pairs give rise to the following unique pair which will play an important role in the characterization of the extremals of the critical posets.
\begin{definition}
\label{def:rmaxsmin}
 Let $(r_{\iota},s_{\iota})_{\iota}$ be the sharp-critical $\ell$-splitting pairs, where we assume that at least one such pair exists. The \emph{maximal splitting pair} $(r_{\max},s_{\min})$ is given by $r_{\max}:=\max_{\iota}r_{\iota}$ and $s_{\min}:=\min_{\iota}s_{\iota}$. 
Associated to the maximal splitting pair are
\begin{align}
\label{eq:Kmax}
\begin{split}
&\mathcal\poly_{\max}:=(\mathcal\poly_0,\ldots,\mathcal\poly_{r_{\max}},\mathcal\poly_{s_{\min}},\ldots, \mathcal\poly_k),\\
&\bsg_{\max}:=\bsg_{\llbracket 0, r_{\max}\rrbracket\cup \llbracket s_{\min}, k\rrbracket},\quad\text{and}\quad \aint\backslash\bsg_{\max}=\aint_{>x_{r_{\max}+1},<x_{s_{\min}}},
\end{split}
\end{align}
where the last identity follows from Lemma \ref{lem:betaintv}. 
\end{definition}

The notion of the maximal splitting pair in Definition \ref{def:rmaxsmin} is tied to the notion of maximal sharp-critical collections introduced \cite[section 9.1]{SvH20}, as part of  the characterization of the extremals of the Alexandrov-Fenchel inequality for critical polytopes. In particular, a sharp-critical collection $\mathcal\poly '\subseteq\mathcal \poly$ is \emph{maximal}  if, for any $\mathcal\poly '\subseteq\mathcal\poly ''\subseteq\mathcal \poly$, we have $\dim\left(\sum_{\poly\in \mathcal \poly''}\poly\right)\ge |\mathcal \poly''|+2$. In other words, any addition of polytopes to $\mathcal \poly'$ destroys its sharp-critical nature. The next result explains the connection between these two notions of maximality. 
\begin{proposition}
\label{prop:max_notions}
Suppose there exists a sharp-critical collection. Then, $\mathcal\poly_{\max}$ is the only maximal sharp-critical collection.
\end{proposition}
\begin{proof}
We start by recalling  that all sharp-critical maximal collections of $\mathcal \poly$ must be disjoint \cite[Lemma 9.2]{SvH20}. By assumption there exists  a sharp-critical collection $\mathcal\poly'$ so let  $\mathcal \poly_*$ be the (necessarily unique) maximal sharp-critical collection containing  $\mathcal\poly'$.  On the other hand, Lemma \ref{lem:sharcrit} shows that any two sharp-critical collections of $\mathcal \poly$ have a non-trivial intersection. It follows that $\mathcal \poly_*$ is the only maximal sharp-critical collection in $\mathcal\poly$.
Next we show that
\[
\mathcal \poly_*= \{\textnormal{union of all sharp-critical collections}\}=\mathcal\poly_{\max},
\]
where the second identity follows from Lemma \ref{lem:sharcrit}, which completes the proof. Indeed, clearly, $\mathcal \poly_*\subseteq\bigcup \{\textnormal{sharp-critical collection}\}$ since $\mathcal \poly_*$ is a sharp-critical collection. If $\bigcup \{\textnormal{sharp-critical collection}\}$ a strictly greater than $\mathcal \poly_*$, i.e., it contains a polytope $K$ not in $\mathcal \poly_*$, then there exists a sharp-critical collection $\mathcal\poly''$ such that $K\in \mathcal\poly''$. Let $\mathcal \poly_{**}$ be the (necessarily unique) maximal sharp-critical collection containing  $\mathcal\poly''$. Then $\mathcal \poly_{**}\neq \mathcal \poly_{*}$ (as $K\in \mathcal \poly_{**}$ but $K\notin \mathcal \poly_{*}$), which contradicts the fact $\mathcal \poly_{*}$ is the only maximal sharp-critical collection. 
\end{proof}

We conclude the section by introducing notation that will be used throughout the paper. Let 
\begin{align}
\label{eq:[]}
\llbracket i_j,i_{j+1}\rrbracket^{\s}:=\llbracket i^\s_j,i^\s_{j+1}\rrbracket = \llbracket i_j+1_{j=\LL}1_{\s},i_{j+1}+1_{j+1=\LL}1_{\s}\rrbracket.
\end{align}
We use this notation when constants are added as well, for example, $\llbracket i_j+1,i_{j+1}-1\rrbracket^{\s}:=\llbracket i^\s_j+1,i^\s_{j+1}-1\rrbracket$.

\subsection{Mixing properties of splitting pairs}
\label{subsec:place}
In this section we analyze the mixing properties of splitting pairs---see Figure \ref{fig:orgsec} for a summary. 
\begin{figure}
\centering
\begin{tikzpicture}
\node[rectangle, draw=black!60, fill=blue!10, very thick, minimum width=50mm, rounded corners, align=center] (a) at (0,0) {\footnotesize \textbf{splitting pair}\\ \footnotesize $\ge 1$ mixed element(s)\\
\footnotesize{(Lemma \ref{lem:splitequiv})}};
\node[rectangle, draw=black!60, fill=blue!10, very thick, minimum width=50mm, rounded corners, align=center] (c) at (-3,-3) {\footnotesize \textbf{supercritical $\LL$-splitting pair}\\\footnotesize  $\ge 2$ mixed elements\\
\footnotesize{(Corollary \ref{cor:Ibetastrong})}};
\node[rectangle, draw=black!60, fill=blue!10, very thick, minimum width=50mm, rounded corners, align=center] (d) at (3,-3) {\footnotesize \textbf{sharp-critical $\LL$-splitting pair}\\ \footnotesize exactly $1$ mixed element\\
\footnotesize (Lemma \ref{lem:generalmixed})};
\node[rectangle, draw=black!60, fill=blue!10, very thick, minimum width=50mm, rounded corners, align=center] (f) at (3,-6) {\footnotesize \textbf{sharp-critical maximal splitting pair}\\ \footnotesize exactly $1$ mixed element, $\mixed$\\
\footnotesize (Corollary \ref{cor:max})
};
\draw [very thick] (a) -- (c);
\draw [very thick] (a) -- (d) -- (f);
\end{tikzpicture}
\caption{A summary of the mixing results from Section \ref{subsec:place}.}
\label{fig:orgsec}
\end{figure}
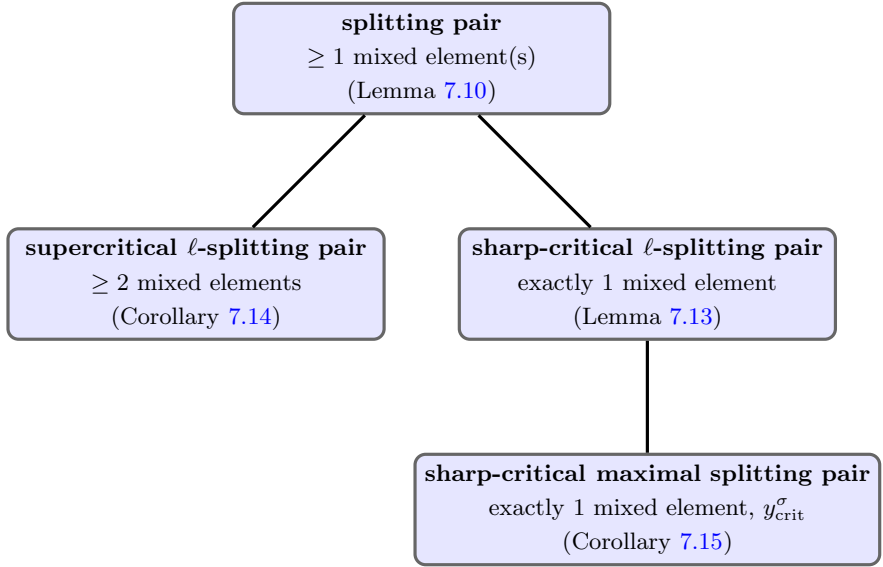
We start by making the definition of a mixed element precise (recall \eqref{eq:betai}):
\begin{definition}
\label{def:mixed_element}
Fix a splitting pair $(r,s)$ and $\sigma\in\mathcal N_{\s}$ for $\s\in \{-,=,+\}$. An element $y^{\sigma}\in \bsg_r\cup\bsg_s$ is a \emph{mixed element} if $\sigma(y^\sigma) \in \llbracket i_{r+1}, i_s \rrbracket\backslash\{i_{r+1},\ldots,i_s\}$.
\end{definition}
Our first result in this section is on the existence of mixed elements. 
\begin{lemma}
\label{lem:splitequiv}
Fix a splitting pair $(r,s)$ and $\sigma\in\mathcal N_{=}$. There exists a mixed element $y^{\sigma}\in \bsg_r\cup\bsg_s$ such that $\sigma(y^\sigma) \in \llbracket i_{r+1}, i_s \rrbracket\backslash\{i_{r+1},\ldots,i_s\}$.
\end{lemma}
\begin{proof}
Recall that $|\abar_{\ge x_{r+1},\le x_s}|\le i_s-i_{r+1}$ by Assumption \ref{ass:crit}, which is equivalent to $|\aint_{> x_{r+1},< x_s}|\le i_s-i_{r+1}-(s-(r+1))-1$. Fix $\sigma\in\mathcal N_{=}$. If there exists no $y^{\sigma} \in \bsg_r\cup\bsg_s$ with  $\sigma(y^{\sigma})\in\llbracket i_{r+1}, i_s \rrbracket\backslash\{i_{r+1},\ldots,i_s\}$, then, by Lemma \ref{lem:betaintv},
\begin{align*}
|\aint_{> x_{r+1},< x_s}|&=|\aint\backslash(\bsg_r\cup\bsg_s\cup \aint_{< x_{r+1}}\cup \aint_{>x_s})|\ge |\llbracket i_{r+1}, i_s \rrbracket\backslash\{i_{r+1},\ldots,i_s\}|\\
&=i_s-i_{r+1}+1-(s-(r+1)+1)=i_s-i_{r+1}-(s-(r+1)),
\end{align*}
which is a contradiction. 
\end{proof}
\begin{corollary}
\label{cor:splitequiv}
For every $0\le j\le k$, $i_j+1<i_{j+1}$.
\end{corollary}
\begin{proof}
If $k=1$ then the corollary holds by the assumption $i_{\LL}<i_{\LL+1}-1$. Otherwise, note that $(r,s) = (j-1,j+1)$ is a splitting pair. Fix $\sigma\in\mathcal N_=$ and note that Lemma \ref{lem:splitequiv} implies that there exists $y^{\sigma}\not \in \abar_{\ge x_{j}, \le x_{j+1}}$ with $\sigma(y^{\sigma})\in \llbracket i_{j}, i_{j+1}\rrbracket$. The first condition gives $y^\sigma \not \in \{ x_j,x_{j+1}\}$, so $\sigma(y^\sigma)\not \in \{ i_j, i_{j+1}\}$. We conclude that $\llbracket i_{j}+1, i_{j+1}-1\rrbracket = \llbracket i_{j}, i_{j+1}\rrbracket \setminus \{i_j, i_{j+1}\}$ is nonempty. 
\end{proof}
Next we move to the mixing properties of $\LL$-splitting pairs. This requires the following simple result.
\begin{lemma}
\label{lem:Ibeta}
$~$
\begin{itemize}
\item Fix $j\in\llbracket 0,k\rrbracket$. For every $\sigma\in\mathcal N_=$, $ \llbracket i_j+1,i_{j+1}-1\rrbracket\subseteq \sigma(\bsg_j)$ and, for every $S\subseteq \llbracket 0,k\rrbracket$, $\bigcup_{j\in S}\llbracket i_j+1,i_{j+1}-1\rrbracket\subseteq \sigma(\bsg_S)$.\\

\item  Fix $j\in \llbracket 0,k\rrbracket\backslash\{\LL-1,\LL\}$ and $\s\in\{-,+\}$. For every $\sigma\in\mathcal N_{\s}$, $\llbracket i_j+1,i_{j+1}-1\rrbracket^{\s}\subseteq \sigma(\bsg_j)$ and, for every $S\subseteq \llbracket 0,k\rrbracket\backslash\{\LL-1,\LL\}$, $\bigcup_{j\in S}\llbracket i_j+1,i_{j+1}-1\rrbracket^{\s}\subseteq \sigma(\bsg_S)$.
\end{itemize}
\end{lemma}
\begin{proof}
$~$
\begin{itemize}
\item Fix $\sigma\in\mathcal N_=$. We will show that $\sigma(y)\in \llbracket i_j+1,i_{j+1}-1\rrbracket\Rightarrow y\in\bsg_j$ which implies $\llbracket i_j+1,i_{j+1}-1\rrbracket\subseteq \sigma(\bsg_j)$; the statement about $S$ follows by taking unions. If $\sigma(y)\in \llbracket i_j+1,i_{j+1}-1\rrbracket$, then clearly $y\in\aint$ and $\sigma(x_j)=i_j<\sigma(y)<i_{j+1}=\sigma(x_{j+1})$. Hence, neither $y<x_j$ nor $y>x_{j+1}$ can occur. It follows that $y\in \bsg_j$.
\item The proof is the same as for the first part where we use that $j\notin\{\LL-1,\LL\}\Rightarrow \sigma(x_j)=i_j\text{ and }\sigma(x_{j+1})=i_{j+1}$.
\end{itemize}
\end{proof}

We now show how the mixing properties of $\LL$-splitting pairs are related to their criticality properties.

\begin{lemma}
\label{lem:generalmixed}
Fix  an $\LL$-splitting pair $(r,s)$, let
\[
\mathcal\poly':=(\mathcal\poly_0,\ldots, \mathcal\poly_{r},\mathcal\poly_{s},\ldots ,\mathcal\poly_k),
\]
and set
\[
c := \dim\left(\sum_{\poly\in \mathcal \poly'}\poly\right) - |\mathcal \poly'|.
\]
Then, for any fixed $\sigma\in\mathcal N_{\s}$, for $\s\in\{-,=,+\}$, there are exactly $c$ distinct mixed elements $y^\sigma_1,\ldots, y^\sigma_c\in \bsg_r\cup\bsg_s$ satisfying $\sigma(y^\sigma_1), \ldots, \sigma(y^\sigma_c)\in \llbracket i_{r+1}, i_s \rrbracket\backslash\{i_{r+1},\ldots,i_s\}$. 
\end{lemma}
\begin{proof}
By Lemma \ref{lem:spancollec},
\[
|\bsg_{\llbracket 0,r\rrbracket\cup \llbracket s,k\rrbracket}|=\dim\left(\sum_{\poly\in \mathcal \poly'}\poly\right) \quad \text{and} \quad |\mathcal \poly'|=| \cup_{j\in \llbracket 0,r\rrbracket\cup \llbracket s,k\rrbracket}\llbracket i_j+1,i_{j+1}-1\rrbracket|.
\]
On the other hand, applying Lemma \ref{lem:Ibeta} to $S:=\llbracket 0,r\rrbracket\cup \llbracket s,k\rrbracket$ yields $\cup_{j\in \llbracket 0,r\rrbracket\cup \llbracket s,k\rrbracket}\llbracket i_j+1,i_{j+1}-1\rrbracket\subseteq \sigma(\bsg_{\llbracket 0,r\rrbracket\cup \llbracket s,k\rrbracket})$. Hence, there are exactly $c$ distinct elements $\{y_i^{\sigma}\}_{i\in [c]}$ satisfying $y_i^{\sigma}\in\bsg_{\llbracket 0,r\rrbracket\cup \llbracket s,k\rrbracket}$ and $\sigma(y_i^{\sigma})\notin \cup_{j\in \llbracket 0,r\rrbracket\cup \llbracket s,k\rrbracket}\llbracket i_j+1,i_{j+1}-1\rrbracket$. Now recall that $\bsg_{\llbracket 0,r\rrbracket\cup \llbracket s,k\rrbracket}=\bsg_r\cup\bsg_s\cup\aint_{<x_{r+1}}\cup \aint_{>{x_s}}$ (Lemma \ref{lem:betaintv}), and note that $\sigma(y_i^\sigma) \notin \cup_{j\in \llbracket 0,r\rrbracket\cup \llbracket s,k\rrbracket}\llbracket i_j+1,i_{j+1}-1\rrbracket$ implies that $y_i^{\sigma}\in \bsg_r\cup\bsg_s$.
\end{proof}

\begin{corollary}
\label{cor:Ibetastrong}
 Let $(r,s)$ be a supercritical $\LL$-splitting pair. Then, for any $\sigma\in\mathcal N_{\s}$, for $\s\in\{-,=,+\}$, there are $c\ge 2$ distinct mixed elements $y^\sigma_1,\ldots, y^\sigma_c\in \bsg_r\cup\bsg_s$ satisfying $\sigma(y^\sigma_1), \ldots, \sigma(y^\sigma_c)\in \llbracket i_{r+1}, i_s \rrbracket\backslash\{i_{r+1},\ldots,i_s\}$. 
\end{corollary}
Note that Corollary \ref{cor:Ibetastrong} is an improvement on Lemma \ref{lem:splitequiv} in the setting of supercritical $\LL$-splitting pairs, as it guarantees the existence of two distinct mixed elements rather than one. In addition,  because Corollary \ref{cor:Ibetastrong} specializes to $\LL$-splitting pairs it can handle  $\mathcal N_{\s}$, for any $\s\in\{-,=,+\}$, while Lemma \ref{lem:splitequiv} applies only to $\mathcal N_{=}$. 

We conclude this section by specializing to the setting where the $\LL$-splitting pair is maximal. Since the maximal splitting pair is sharp-critical, Lemma \ref{lem:generalmixed} immediately gives that we have exactly one mixed element. 

\begin{corollary}
\label{cor:max}
Fix $\s\in \{-,=,+\}$ and $\sigma\in\mathcal N_{\s}$. There exists a unique mixed element $\mixed$  satisfying $\mixed \in \bsg_{r_{\max}}\cup\bsg_{s_{\min}}$ and $\sigma(\mixed)\in \llbracket i_{r_{\max}+1}, i_{s_{\min}} \rrbracket\backslash\{i_{r_{\max}+1},\ldots,i_{s_{\min}}\}$. 
\end{corollary}

\section{The extreme normal directions}
\label{sec:ext}
Once Assumption \ref{ass:crit} is set in place, we are ready, in principle, to apply Theorem \ref{thm:SvH}. However, Theorem \ref{thm:SvH} characterizes the extremals \emph{geometrically} in terms of the $(B,\mathcal \poly)$-extreme normal directions so a \emph{combinatorial} interpretation of these vectors is needed. The goal of this section is to characterize, combinatorially, a sufficient number of the $(B,\mathcal \poly)$-extreme normal directions so that Theorem \ref{thm:SvH} can be applied. 

We recall that $\{e_j\}_{j\in[n-k]}$ is the standard basis of $\R^{n-k}$ and, for $u,v\in [n-k]$ distinct, we let $e_{uv}:=\frac{e_u-e_v}{\sqrt{2}}$ and $o_{uv}:=\frac{e_u+e_v}{\sqrt{2}}$. We also recall the definition \eqref{eq:betai}:
\[
\bsg_i:=\aint\backslash(\aint_{<x_i}\cup \aint_{>x_{i+1}}).
\]
The next result characterizes certain faces of the polytopes $\{\poly_i\}$. 
\begin{lemma}
\label{lem:dirspan}
Fix $ i\in \llbracket 0,k\rrbracket$. We have,
\begin{enumerate}[(i)]
\item For $y_j\notin\bsg_i$, $\lin(F(\poly_i,\pm e_j))=\R^{\bsg_i}$, and for $y_u,y_v\notin\bsg_i$, $\lin(F(\poly_i,\pm e_{uv}))=\R^{\bsg_i}$.
\item For $y_j\in\bsg_i$, $\lin(F(\poly_i,-e_j))=\R^{\bsg_i\backslash\aint_{\le y_j}}$ and $\lin(F(\poly_i,e_j))=\R^{\bsg_i\backslash\aint_{\ge y_j}}$.
\item For $y_u,y_v\in\bsg_i$ such that $y_v$ covers $y_u$ in $\aint$, $\lin(F(\poly_i,e_{uv}))=\R^{\bsg_i\backslash\{y_u,y_v\}}\oplus\sspan(o_{uv})$.
\end{enumerate}
\end{lemma}
\begin{proof}
We start by recalling \eqref{eq:Kiordpoly}:
\[
\poly_i=O_{\bsg_i}+1_{\aint_{>x_{i+1}}}\mbox{ for } i\in \llbracket 0,k\rrbracket
\]
so that
\[
\lin(F(\poly_i,\mathrm u))=\lin(F(O_{\bsg_i},\mathrm u))\quad\forall\, \mathrm u\in S^{n-k-1}.
\]
\begin{enumerate}[(i)]
\item Let $\mathrm u\in \{\pm e_j\}$ so, since $h_{O_{\bsg_i}}(\mathrm u)=0$ as $y_j\notin\bsg_i$, we get that $\lin(F(\poly_i,\mathrm u))=O_{\bsg_i}\cap \{t_j=0\}=O_{\bsg_i}$, where the last equality holds as $y_j\notin\bsg_i$. Similarly, let $\mathrm u\in\{\pm e_{uv}\}$ so, since $h_{O_{\bsg_i}}(\mathrm u)=0$ as $y_u,y_v\notin\bsg_i$, we get that $\lin(F(\poly_i,\mathrm u))=O_{\bsg_i}\cap \{t_u=t_v\}=O_{\bsg_i}$, where the last equality holds as $y_u,y_v\notin\bsg_i$. The proof is complete as $\dim O_{\bsg_i}=|\bsg_i|$ (Lemma \ref{lem:dimOb}).
\item Since $h_{O_{\bsg_i}}(-e_j)=0$, we get $\lin(F(\poly_i,-e_j))=O_{\bsg_i}\cap \{t_j=0\}=O_{\bsg_i\backslash\aint_{\le y_j}}$ where the last equality holds as $y_j\in \bsg_i$. Analogously, since $h_{O_{\bsg_i}}(e_j)=1$ (because $y_j\in\bsg_i$), we get $\lin(F(\poly_i,e_j))=O_{\bsg_i}\cap \{t_j=1\}=O_{\bsg_i\backslash\aint_{\ge y_j}}$.
\item  Since $y_u\le y_v$ we have $h_{O_{\bsg_i}}(e_{uv})=0$, so $\lin(F(\poly_i,e_{uv}))=O_{\bsg_i}\cap \{t_u=t_v\}$. Since $y_v$ covers $y_u$, it follows from Lemma \ref{lem:dimOb}(iii) that $\dim(\lin(F(\poly_i,e_{uv})))=|\bsg_i|-1$. On the other hand, since $\lin(F(\poly_i,e_{uv}))\perp e_{uv}$, we have $\lin(F(\poly_i,e_{uv}))\subseteq \R^{\bsg_i}\cap e_{uv}^{\perp}=\R^{\bsg_i\backslash\{y_u,y_v\}}\oplus\sspan(o_{uv})$. The proof is complete since $\dim(\R^{\bsg_i\backslash\{y_u,y_v\}}\oplus\sspan(o_{uv}))=|\bsg_i|-1$. 
\end{enumerate}
\end{proof}

The following proposition, which is the main result of this section,  characterizes combinatorially some of the $(B,\mathcal \poly)$-extreme normal directions. We remark that the $(B,\mathcal \poly)$-extreme normal directions given in Proposition \ref{prop:dir}($e$--$h$) will  be used only for the characterization of the extremals of sharp-critical posets. 

\begin{proposition}
\label{prop:dir}
The following vectors are $(B,\mathcal \poly)$-extreme normal directions:\\

\begin{enumerate}[(a)] 
\item For each fixed $ 0\le m\le \LL$: $-e_j$ for any $j$ such that $y_j\in\aint_{>x_m}$ and there exists $\sigma\in\mathcal N_=$ satisfying $\sigma(y_j)=i_m+1$. \\
\item For each fixed $ \LL\le m\le k+1$: $e_j$ for any $j$ such that $y_j\in\aint_{<x_m}$ and there exists $\sigma\in\mathcal N_=$ satisfying $\sigma(y_j)=i_m-1$.\\
\item $e_{uv}$ for any $u,v$ such that $y_u<y_v$ and there exists $\sigma\in\mathcal N_=$ satisfying $\sigma(y_u)+1=\sigma(y_v)$.\\
\item $e_{uv}$ for any $u,v$ such that $y_u<y_v$ and  there exists $\sigma\in\mathcal N_=$ satisfying $\sigma(y_u)=i_{\LL}-1$ and $\sigma(y_v)=i_{\LL}+1$.\\
\item For each fixed $r_{\max}+1\le m\le\LL-1$: $-e_j$  for any $j$ such that $y_j\in\aint_{>x_m}$ and there exists $\sigma\in\mathcal N_=$ satisfying $\sigma(y_j)=i_m+2$.\\
\item For each fixed $\LL+1\le m\le s_{\min}$: $e_j$  for any $j$ such that $y_j\in\aint_{<x_m}$ and  there exists $\sigma\in\mathcal N_=$ satisfying $\sigma(y_j)=i_m-2$.\\
\item $-e_j$  for any $j$ such that $y_j\in\aint_{>x_{\LL-1}}$ and there exists $\sigma\in\mathcal N_+$ satisfying $\sigma(y_j)=i_{\LL-1}+2$.\\
\item $e_j$  for any $j$ such that $y_j\in\aint_{<x_{\LL+1}}$ and there exists $\sigma\in\mathcal N_-$ satisfying $\sigma(y_j)=i_{\LL+1}-2$.
\end{enumerate}
\end{proposition}
Note that parts (a--b), which suffice for the supercritical posets, provide information about nearest neighbors of $x_m$, while parts (e--f), which are needed for the critical posets, provide information about second-nearest neighbors of $x_m$.
\begin{proof}{(of Proposition \ref{prop:dir})}
By Definition \ref{def:extreme}, we need to show that, whenever $\mathrm u$ is one of the vectors in the proposition, we have, for any collection $\mathcal \poly'\subseteq\mathcal \poly$,
\[
\dim\left(\sum_{\poly\in\mathcal \poly'}F(\poly,\mathrm u)\right)\ge |\mathcal \poly'|.
\]
Let $j_0:=-1<0\le j_1<\cdots<j_p\le k<k+1=:j_{p+1}$ and $\kappa_1,\ldots, \kappa_p$, with $0\le\kappa_q\le i_{j_q+1}-i_{j_q}-1-1_{j_q\in\{\LL-1,\LL\}}$, for $j_q\in\llbracket 0,k\rrbracket$, and set
\[
\mathcal\poly':=(\underbrace{\poly_{j_1},\ldots,\poly_{j_1}}_{\kappa_1},\ldots,\underbrace{\poly_{j_p},\ldots,\poly_{j_p}}_{\kappa_p}),
\]
\[
J:=\{j_1,\ldots,j_p\}.
\]
For notational simplicity we set
\begin{align}
\label{eq:Ij}
\I_j:=\llbracket i_j+1,i_{j+1}-1\rrbracket\quad \text{for }j\in\llbracket 0,k\rrbracket,\quad \I_S:=\cup_{j_q\in S}\I_{j_q}\quad \text{for }S\subset \llbracket 0,k\rrbracket;
\end{align}
for example, 
\[
\I_{\llbracket r+1,s\rrbracket}=\llbracket i_{r+1},i_s\rrbracket\backslash\{i_{r+1},\ldots, i_s\}.
\]
Note that
\[
|\I_J|-1_{\LL-1\in J}-1_{\LL\in J}\ge  |\mathcal \poly'|,
\]
because  $0\le\kappa_q\le i_{j_q+1}-i_{j_q}-1-1_{j_q\in\{\LL-1,\LL\}}$ and since $\I_{j_q}=i_{j_q+1}-i_{j_q}-1$.
\begin{enumerate}[(a)]
\item Fix $0\le m\le \LL$ and consider $\sigma \in\mathcal N_=$ such that $\sigma(y_j)=i_m+1$ where $j$ is such that $y_j\in\aint_{>x_m}$.  Let
\[
\gamma_{j_q}:=
\begin{cases}
\bsg_{j_q}&\mbox{if }y_j\notin\bsg_{j_q}\\
\bsg_{j_q}\backslash\aint_{\le y_j} &\mbox{if }y_j\in\bsg_{j_q}
\end{cases},
\]
and $\gamma_J:=\cup_{j_q\in J}\gamma_q$. By Lemma \ref{lem:dirspan}(i--ii),
\[
\lin(F(\poly_{j_q},-e_j))=\R^{\gamma_{j_q}}\quad \text{for all}\quad  j_q\in J,
\]
so, by \eqref{eq:linspan},
\[
\lin\left(F\left(\sum_{\poly\in\mathcal \poly'}\poly,-e_j\right)\right)=\R^{\gamma_J}.
\]
It follows that
\[
\dim\left(\sum_{\poly\in\mathcal \poly'}F(\poly,-e_j)\right)=|\gamma_J|,
\]
so it remains to show that $|\gamma_J|\ge |\mathcal \poly'|$. Since $|\I_J|-1_{\LL-1\in J}-1_{\LL\in J}\ge  |\mathcal \poly'|$, it will suffice to show that
\[
|\gamma_J|\ge |\I_J|-1_{\LL-1\in J}-1_{\LL\in J},
\]
which requires the following claim.
\begin{claim}
\label{cl:a}
$~$

\begin{enumerate}[(i)]
\item For $j_q\neq m$, $\I_{j_q}	\subseteq\sigma(\gamma_{j_q})$.
\item For $j_q= m$, $\I_{m}\backslash\{i_m+1\}	\subseteq\sigma(\gamma_m)$.
\end{enumerate}
\end{claim}
\begin{proof}
$~$

\begin{enumerate}[(i)]
\item We need to consider the cases  $y_j\notin\bsg_{j_q}$ and $y_j\in\bsg_{j_q}$. If $y_j\notin\bsg_{j_q}$ then the result holds by Lemma \ref{lem:Ibeta}. Suppose $y_j\in\bsg_{j_q}$. Then, we must have $m<j_q$; otherwise, $j_q<m$ (by assumption $j_q\neq m$) so  $x_{j_q+1}\le x_m<y_j$, but this implies $y_j\notin\bsg_{j_q}$, which is a contradiction. Now let $y$ be any element such that $\sigma(y)\in \I_{j_q}$, which by Lemma \ref{lem:Ibeta}, implies that $y\in\bsg_{j_q}$. Since $\sigma(y)\ge i_{j_q}+1>i_m+1=\sigma(y_j)$, we can conclude that, in fact, $y\in\bsg_{j_q}\backslash\aint_{\le y_j}=\gamma_{j_q}$. To summarize, $\sigma(y)\in \I_{j_q}\Rightarrow y\in\gamma_{j_q}$, which shows $\I_{j_q}\subseteq\sigma(\gamma_{j_q})$.
\item We need to consider the cases  $y_j\notin\bsg_m$ and $y_j\in\bsg_m$. Suppose $y_j\notin\bsg_m$. By Corollary \ref{cor:splitequiv}, $i_m+1<i_{m+1}$ so $\sigma(y_j)=i_m+1\in \I_m\subseteq\sigma(\bsg_m)$, where we used Lemma \ref{lem:Ibeta}. This contradicts $y_j\notin\bsg_m$ so we are left to consider $y_j\in\bsg_m$. Let $y$ be any element such that $\sigma(y)\in \I_{m}\backslash\{i_m+1\}=\llbracket i_m+2,i_{m+1}-1\rrbracket$. Then, $y\in \bsg_{m}\backslash\aint_{\le y_j}$ since, by Lemma \ref{lem:Ibeta}, $y\in \bsg_{m}$, but we also have $\sigma(y)\ge i_m+2>i_m+1=\sigma(y_j)$. To summarize,  $\sigma(y)\in \I_{m}\backslash\{i_m+1\}\Rightarrow y\in \bsg_{m}\backslash\aint_{\le y_j}$, which shows $\I_{m}\backslash\{i_m+1\}\subseteq\sigma(\gamma_m)$. 
\end{enumerate}
\end{proof}
In order to use Claim \ref{cl:a} in the proof of $|\gamma_J|\ge |\I_J|-1_{\LL-1\in J}-1_{\LL\in J}$, we  distinguish between two cases: $m\notin J$ and $m\in J$. If $m\notin J$, then taking a union over $j_q\in J$ in Claim \ref{cl:a} gives $\I_J\subseteq\sigma(\gamma_J)$, so $|\gamma_J|\ge |\I_J|\ge  |\I_J|-1_{\LL-1\in J}-1_{\LL\in J}$, as desired. 

Suppose then that $m\in J$. Taking a union over $j_q\in J$ in Claim \ref{cl:a} gives $\I_J\backslash\{i_m+1\}\subseteq \sigma(\gamma_J)$. Hence, if $\LL\in J$, we have $|\gamma_J|\ge  |\I_J|-1\ge   |\I_J|-1_{\LL-1\in J}-1_{\LL\in J}$, which completes the proof. It remains to consider the case $m\in J$ and $\LL\notin J$:

Choose the largest $0\le b\le p$ such that $j_b<\LL$, so $j_b<\LL<j_{b+1}$, and, in particular, $(j_b,j_{b+1})$ is an $\ell$-splitting pair. By Lemma \ref{lem:splitequiv}, there exists $y^{\sigma} \in \bsg_{j_b}\cup\bsg_{j_b+1}$ such that $\sigma(y^\sigma) \in\I_{\llbracket j_b+1,j_{b+1}-1\rrbracket}$. 
Since $m=j_q<\LL$ for some $ 0\le q\le p$, and since $b$ is the largest element in $\llbracket 0,p\rrbracket$ such that $j_b<\LL$, we have $q\le b$, and hence $m\le j_b$. It follows that $\sigma(y_j)=i_m+1<i_{j_b+1}+1\le \sigma(y^{\sigma})$, and, in particular, $y^{\sigma}\notin \aint_{\le y_j}$. Hence, $y^{\sigma}\in (\bsg_{j_b}\backslash \aint_{\le y_j})\cup(\bsg_{j_{b+1}}\backslash \aint_{\le y_j})\subseteq \gamma_{j_b}\cup  \gamma_{j_{b+1}}\subseteq\gamma_J$, so $(\I_J\backslash\{i_m+1\})\cup\{\sigma(y^{\sigma})\}\subseteq \sigma(\gamma_J)$. Finally, $\sigma(y^{\sigma})\notin\I_J$ because $J$ and $\llbracket j_b+1,j_{b+1}-1\rrbracket$ do not intersect, which completes the proof since it implies that $|\gamma_J|\ge |(\I_J\backslash\{i_m+1\})\cup\{\sigma(y^{\sigma})\}|\ge |\I_J|-1+1=|\I_J|\ge    |\I_J|-1_{\LL-1\in J}-1_{\LL\in J}$.\\

\item The proof is analogous to part (a). \\

\item Fix $u,v$ such that there exist $y_u<y_v$ with $\sigma\in\mathcal N_=$ satisfying $\sigma(y_u)+1=\sigma(y_v)$. For $j_q\in J$, let
\[
\gamma_{j_q}:=
\begin{cases}
\bsg_{j_q}&\mbox{if }y_u,y_v\notin \bsg_{j_q}\\
\bsg_{j_q}\backslash\{y_u,y_v\}&\mbox{if }y_u,y_v\in\bsg_{j_q},\\
\bsg_{j_q}\backslash\aint_{\ge y_u} &\mbox{if }y_u\in\bsg_{j_q},y_v\notin\bsg_{j_q},\\
\bsg_{j_q}\backslash\aint_{\le y_v} &\mbox{if }y_u\notin\bsg_{j_q},y_v\in\bsg_{j_q}.
\end{cases}
\]
We start by describing the faces of $\{\poly_{j_q}\}_{j_q\in J}$ in the directions $\{e_{uv}\}$.
\begin{claim}
\label{cl:c1}
For every $j_q\in J$,
\[
\lin(F(K_{j_q},e_{uv}))=
\begin{cases}
\R^{\gamma_{j_q}}\oplus\sspan(o_{uv})&\mbox{if }y_u,y_v\in \bsg_{j_q},\\
\R^{\gamma_{j_q}}&\mbox{ otherwise}.
\end{cases}
\]
\end{claim}
\begin{proof}
There are four cases to consider:
\begin{itemize}
\item $y_u,y_v\in \bsg_{j_q}$: The claim follows from Lemma \ref{lem:dirspan}(iii).
\item $y_u,y_v\notin \bsg_{j_q}$: The claim follows from Lemma \ref{lem:dirspan}(i).
\item $y_u\in \bsg_{j_q}, y_v\notin \bsg_{j_q}$: We will show that $\lin(F(K_{j_q},e_{uv}))=\lin(F(K_{j_q},e_u))$, and the claim will then follow from Lemma \ref{lem:dirspan}(ii). Indeed, the assumption $y_v\notin \bsg_{j_q}$ implies that $y_v\in\aint_{<x_{j_q}}\cup\aint_{>x_{j_q+1}}$. But $y_v\notin\aint_{<x_{j_q}}$ because, otherwise, $y_u<y_v<x_{j_q}$, which contradicts the assumption $y_u\in \bsg_{j_q}$. Hence, $y_v>x_{j_q+1}$ so, by the definition \eqref{eq:Ki} of $K_{j_q}$, $t_v=1$ for any $t\in O_{\aint}$. Since $\sup_{t_u\in [0,1]}t_u=1$ (as $y_u\in \bsg_{j_q}\Rightarrow y\not<x_{j_q}$), it follows that $h_{K_{j_q}}(e_{uv})=\sup_{t_u\in [0,1]}\frac{t_u-t_v}{\sqrt{2}}=0$, and hence 
\[
\lin(F(K_{j_q},e_{uv}))=K_{j_q}\cap \{t_u=t_v\}=K_{j_q}\cap \{t_u=1\}=\lin(F(K_{j_q},e_{u})),
\]
as needed.
\item $y_u\notin \bsg_{j_q},y_v\in \bsg_{j_q}$: The argument is analogous to the previous case: $y_u\in \bsg_{j_q}$ and $y_v\notin \bsg_{j_q}$.
\end{itemize}
\end{proof}

Next we prove the analogue of Claim \ref{cl:a}.
\begin{claim}
\label{cl:c2}
Choose $m\in \llbracket 0,k+1\rrbracket$ such that $i_m<\sigma(y_u)<\sigma(y_v)<i_{m+1}$.
\begin{enumerate}[(i)]
\item For $j_q\neq m$, $\I_{j_q}	\subseteq\sigma(\gamma_{j_q})$.
\item For $j_q= m$, $\I_{m}\backslash\{\sigma(y_u),\sigma(y_v)\}	\subseteq\sigma(\gamma_m)$.
\end{enumerate}

\begin{proof}
$~$
We need to consider the four cases (1) $y_u,y_v\in \bsg_{j_q}$, (2) $y_u,y_v\notin \bsg_{j_q}$, (3) $y_u\in \bsg_{j_q}, y_v\notin \bsg_{j_q}$, and (4) $y_u\notin \bsg_{j_q},y_v\in \bsg_{j_q}$. 
\begin{enumerate}[(i)]
\item Case (1): For any $y$ such that $\sigma(y)\in \I_{j_q}$, we have $y\in \bsg_{j_q}$, by Lemma \ref{lem:Ibeta}, and $y\notin\{y_u,y_v\}$, since $\sigma(y_u),\sigma(y_v)\in \I_m$, and $ \I_m\cap \I_{j_q}=\varnothing$ as $m\neq j_q$. Hence, $y\in  \bsg_{j_q}\backslash\{y_u,y_v\}=\gamma_{j_q}$, so we conclude $\I_{j_q}\subseteq \sigma(\gamma_{j_q})$. 

Case (2): Since $\gamma_{j_q}=\bsg_{j_q}$, Lemma \ref{lem:Ibeta} implies $\I_{j_q}\subseteq \sigma(\gamma_{j_q})$. 

Case (3): For any $y$ such that $\sigma(y)\in \I_{j_q}$, we have $y\in \bsg_{j_q}$, by Lemma \ref{lem:Ibeta}.  On the other hand, the proof of Claim \ref{cl:c1} showed that $y_v>x_{j_q+1}$, so the assumption on $m$ implies that $j_q<m$, which means that $\sigma(y)<i_{j_q+1}\le i_m<\sigma (y_u)$. In particular, $y\notin \aint_{\ge y_u}$ so we conclude that $y\in \bsg_{j_q}\backslash  \aint_{\ge y_u}=\gamma_{j_q}$. It follows that $\I_{j_q}\subseteq \sigma(\gamma_{j_q})$. 

Case (4) is analogous to case (3).

 \item Case (1): For any $y\in\I_{m}\backslash\{\sigma(y_u),\sigma(y_v)\}$, Lemma \ref{lem:Ibeta} implies that $y\in \bsg_m\backslash \{y_u,y_v\}=\gamma_m$, which implies that $\I_{m}\backslash\{\sigma(y_u),\sigma(y_v)\}	\subseteq \sigma(\gamma_m)$.

Case (2): Since $\gamma_m=\bsg_m$, Lemma \ref{lem:Ibeta} implies $\I_{m}\backslash\{\sigma(y_u),\sigma(y_v)\}	\subseteq \sigma(\gamma_m)$. 

Case (3): As shown in part (i) case (3), we must have $j_q<m$ so this case cannot occur. 

Case (4) is analogous to case (3).
\end{enumerate}
\end{proof}
\end{claim}
Choose $m\in \llbracket 0,k+1\rrbracket$ such that $i_m<\sigma(y_u)<\sigma(y_v)<i_{m+1}$. To complete the proof we  distinguish between two cases: $m\notin J$ and $m\in J$. Suppose $m\notin J$. By \eqref{eq:linspan} and Claim \ref{cl:c1}, $\R^{\gamma_J}\subseteq\lin\left(\sum_{\poly\in\mathcal \poly'}F(\poly,e_{uv})\right)$,
so $\dim\left(\sum_{\poly\in\mathcal \poly'}F(\poly,e_{uv})\right)\ge |\gamma_J|$. On the other hand, by Claim \ref{cl:c2} and as $m\notin J$, $|\gamma_J|\ge |\I_J|$. We conclude 
\[
\dim\left(\sum_{\poly\in\mathcal \poly'}F(\poly,e_{uv})\right)\ge |\I_J|\ge  |\I_J|-1_{\LL-1\in J}-1_{\LL\in J}\ge |\mathcal \poly'|,
\]
which completes the proof.

Suppose that $m\in J$. By the definition of $m$, $\sigma(y_u),\sigma(y_v)\in \I_m$, so Lemma \ref{lem:Ibeta} implies that $y_u,y_v\in \bsg_m$. By Claim \ref{cl:c1}, it follows that $F(\poly_m,e_{uv})=\R^{\gamma_m}\oplus \sspan(o_{uv})$. On the other hand, for any $j_q\in J$, by the definition of $\gamma_{j_q}$, we have $y_u,y_v\notin\gamma_{j_q}$. Hence, $\R^{\gamma_{j_q}}\cap \sspan(o_{uv})= \{0\}$ for all $j_q\in J$, and in particular, $\R^{\gamma_{J}}\cap \sspan(o_{uv})= \{0\}$. It follows from  \eqref{eq:linspan} that
\[
\lin\left(\sum_{\poly\in\mathcal \poly'}F(\poly,e_{uv})\right)=\R^{\gamma_J}\oplus \sspan(o_{uv}),
\]
and
\[
\dim\left(\sum_{\poly\in\mathcal \poly'}F(\poly,e_{uv})\right)= |\gamma_J|+1.
\]
We now consider separately the cases $\LL\in J$ and $\LL\notin J$. Suppose $\LL\in J$. By Claim \ref{cl:c2}, $|\gamma_J|\ge |\I_J|-2$ so 
\[
|\gamma_J|+1\ge  |\I_J|-1\ge  |\I_J|-1-1_{\LL-1\in J}\ge |\mathcal \poly'|,
\]
which completes the proof. It remains to consider the case $m\in J$ and $\LL\notin J$:

Choose the largest $b\in\llbracket 0,p\rrbracket$ such that $j_b<\LL$, so $j_b<\LL<j_{b+1}$, and, in particular, $(j_b,j_{b+1})$ is an $\ell$-splitting pair. By Lemma \ref{lem:splitequiv}, there exists $y^{\sigma}\in \bsg_{j_b}\cup\bsg_{j_{b+1}}$ with $\sigma(y^{\sigma})\in \I_{\llbracket j_b+1,j_{b+1}-1\rrbracket}$. We will show that 
\begin{align}
\label{eq:ysigmgamma}
y^{\sigma}\in  \gamma_{j_b}\cup\gamma_{j_{b+1}}.
\end{align}
Assume for now that \eqref{eq:ysigmgamma} holds. Then, $(\I_J\backslash\{\sigma(y_u),\sigma(y_v)\})\cup\{\sigma(y^{\sigma})\}\subseteq \sigma(\gamma_J)$. On the other hand, arguing as in part (a) for the case $m\in J,\LL\notin J$, we have $\sigma(y^{\sigma})\notin \I_J$. Hence, $|\gamma_J|+1\ge |\I_J|$, so $\dim\left(\sum_{\poly\in\mathcal \poly'}F(\poly,e_{uv})\right)\ge  |\I_J|\ge  |\I_J|-1_{\LL-1\in J}-1_{\LL\in J}\ge |\mathcal \poly'|$, which completes the proof.

It remains to prove \eqref{eq:ysigmgamma}. We will show $y^{\sigma}\in  \bsg_{j_b}\Rightarrow y^{\sigma}\in  \gamma_{j_b}$, and the argument for $y^{\sigma}\in  \bsg_{j_{b+1}}\Rightarrow y^{\sigma}\in  \gamma_{j_{b+1}}$ is analogous. Since $y^{\sigma}\in \bsg_{j_b}\cup\bsg_{j_{b+1}}$, \eqref{eq:ysigmgamma} will follow. Suppose then that $y^{\sigma}\in  \bsg_{j_b}$ so our task is to show that $y^{\sigma}\in  \gamma_{j_b}$. There are two cases to consider: $j_b\ge m$ and $j_{b+1}\le m$; we will consider the case $j_b\ge m$ and the argument for the case $j_{b+1}\le m$ is analogous.

Let us start by showing that $\gamma_{j_q}$ cannot be equal to $\bsg_{j_q}\backslash\aint_{\ge y_u}$. Indeed, the latter occurs only if $y_u\in\bsg_{j_q},y_v\notin \bsg_{j_q}$, in which case, either $y_v<x_{j_b}$ or $y_v>x_{j_b+1}$. If $y_v<x_{j_b}$, then $y_u<y_v<x_{j_b}$ which contradicts $y_u\in\bsg_{j_q}$. If $y_v>x_{j_b+1}$, then $\sigma(x_{j_b+1})<\sigma(y_v)<i_{m+1}=\sigma(x_{m+1})$, which contradicts $m\le j_b$. We conclude that $\gamma_{j_q}\in\{\bsg_{j_q},\bsg_{j_q}\backslash\{y_u,y_v\},\bsg_{j_q}\backslash\aint_{\le y_u}\}$, and since $y^{\sigma}\in  \bsg_{j_b}$, it suffices to show that $y^{\sigma}\notin \{y_u,y_v\}$ and $y^{\sigma}\notin\aint_{\le y_u}$. To see that $y^{\sigma}\notin \{y_u,y_v\}$, note that $\sigma(y^{\sigma})\in \I_{\llbracket j_b+1,j_{b+1}-1\rrbracket}$ while $\sigma(y_u),\sigma(y_v)\in \I_m$. Since $m\le j_b$, $\I_{\llbracket j_b+1,j_{b+1}-1\rrbracket}\cap \I_m=\varnothing$ so $y^{\sigma}\notin \{y_u,y_v\}$. To see that $y^{\sigma}\notin\aint_{\le y_u}$, note that, since  $m\le j_b$, $\sigma(y_u)<i_{m+1}\le i_{j_b+1}<\sigma(y^{\sigma})$, where the last inequality holds as $\sigma(y^{\sigma})\in \I_{\llbracket j_b+1,j_{b+1}-1\rrbracket}$.\\

\item Fix $u,v$ such that there exist $y_u<y_v$ with $\sigma\in\mathcal N_=$ satisfying $\sigma(y_u)=i_{\LL}-1$ and $\sigma(y_v)=i_{\LL}+1$. For $j_q\in J$ we let $\gamma_{j_q}$ be as in part (c). We start by showing that Claim \ref{cl:c1} holds here as well.

\begin{claim}
\label{cl:d1}
For every $j_q\in J$,
\[
\lin(F(K_{j_q},e_{uv}))=
\begin{cases}
\R^{\gamma_{j_q}}\oplus\sspan(o_{uv})&\mbox{if }y_u,y_v\in \bsg_{j_q},\\
\R^{\gamma_{j_q}}&\mbox{ otherwise}.
\end{cases}
\]
\end{claim}
\begin{proof}
The proof is the same as the proof of Claim \ref{cl:c1}, but we need to check that, when $y_u,y_v\in\bsg_{j_q}$, $y_v$ covers $y_u$ in $\aint$. The latter must be true since, otherwise, there exists $z\in\aint$ such that $y_u<z<y_v$, so $i_{\LL}-1=\sigma(y_u)<\sigma(z)<\sigma(y_v)=i_{\LL}+1$. This implies $z=x_{\LL}$, which contradicts $z\in\aint$.
\end{proof}
Next we prove the analogue of Claim \ref{cl:c2}.

\begin{claim}
\label{cl:d2}
$~$
\begin{enumerate}[(i)]
\item For $j_q\notin\{\LL-1,\LL\}$, $\I_{j_q}	\subseteq\sigma(\gamma_{j_q})$.
\item  For $j_q=\LL-1$, $\I_{\LL-1}\backslash\{i_{\LL}-1\}	\subseteq\sigma(\gamma_{\LL-1})$.
\item  For $j_q=\LL$, $\I_{\LL}\backslash\{i_{\LL}+1\}	\subseteq\sigma(\gamma_{\LL})$.
\end{enumerate}
\end{claim}
\begin{proof}
We need to consider the four cases (1) $y_u,y_v\in \bsg_{j_q}$, (2) $y_u,y_v\notin \bsg_{j_q}$, (3) $y_u\in \bsg_{j_q}, y_v\notin \bsg_{j_q}$, and (4) $y_u\notin \bsg_{j_q},y_v\in \bsg_{j_q}$.
\begin{enumerate}[(i)]
\item  Case (1): For any $y$ such that $\sigma(y)\in \I_{j_q}$, we have $y\notin\{y_u,y_v\}$ since $\sigma(y_u),\sigma(y_v)\notin \I_{j_q}$ (because $j_q\notin\{\LL-1,\LL\}$). Hence, by Lemma \ref{lem:Ibeta}, $y\in\bsg_{j_q}\backslash \{y_u,y_v\}=\gamma_{j_q}$, so we conclude $\I_{j_q}	\subseteq\sigma(\gamma_{j_q})$. 

Case (2):  By Lemma \ref{lem:Ibeta}, $\I_{j_q}\subseteq \bsg_{j_q}=\gamma_{j_q}$ so  $\I_{j_q}	\subseteq\sigma(\gamma_{j_q})$. 

Case (3): We start by showing that $j_q<\LL$. Indeed, suppose for contradiction that $ j_q\ge \LL$. Since $y_v\notin\bsg_{j_q}$, we have that either $y_v<x_{j_q}$ or $y_v>x_{j_q+1}\ge x_{\LL+1}$. We cannot have $y_v>x_{j_q+1}\ge x_{\LL+1}$, since $\sigma(y_v)=i_{\LL}+1<i_{\LL+1}=\sigma(x_{\LL+1})$. Hence, we must have $y_u<y_v<x_{j_q}$, which contradicts $y_u\in\bsg_{j_q}$. We conclude that $j_q<\LL$. The assumption $j_q\notin\{\LL-1,\LL\}$ implies that in fact $j_q<\LL-1$. Hence, for any $y$ such that $\sigma(y)\in \I_{j_q}$, we have $y\in\bsg_{j_q}\backslash\aint_{\ge y_u}=\gamma_{j_q}$, because $\sigma(y)<i_{j_q+1}\le i_{\LL-1}<i_{\LL}-1=\sigma(y_u)$. It follows that $\I_{j_q}\subseteq \sigma(\gamma_{j_q})$.

 Case (4) is analogous to case (3). 
 
\item  Case (1): For any $y$ such that $\sigma(y)\in \I_{\LL-1}\backslash\{i_{\LL}-1\}$, we have  $y\notin\{y_u,y_v\}$ so, by Lemma \ref{lem:Ibeta}, $\I_{\LL-1}\backslash\{i_{\LL}-1\}\subseteq \sigma(\gamma_{\LL-1})$. 

Case (2): By Lemma \ref{lem:Ibeta}, $\I_{\LL-1}\subseteq \sigma(\bsg_{\LL-1})=\sigma(\gamma_{\LL-1})$ so  $\I_{\LL-1}	\subseteq\sigma(\gamma_{\LL-1})$. 

Case (3):  For any $y$ such that $\sigma(y)\in \I_{\LL-1}\backslash\{i_{\LL}-1\}$, we have $y\in\bsg_{\LL-1}\backslash\aint_{\ge y_u}=\gamma_{\LL-1}$, because, by the definition of  $\I_{\LL-1}\backslash\{i_{\LL}-1\}$, $\sigma(y)<i_{\LL}-1=\sigma(y_u)$. It follows that $\I_{\LL-1}\backslash\{i_{\LL}-1\}	\subseteq\sigma(\gamma_{\LL-1})$. 

Case (4) is analogous to case (3).

\item The argument is analogous to (ii).

\end{enumerate}
\end{proof}
 By \eqref{eq:linspan} and Claim \ref{cl:d1}, $\R^{\gamma_J}\subseteq\lin\left(\sum_{\poly\in\mathcal \poly'}F(\poly,e_{uv})\right)$, so $\dim\left(\sum_{\poly\in\mathcal \poly'}F(\poly,e_{uv})\right)\ge |\gamma_J|$. By Claim \ref{cl:d2}, using the fact that $\{\I_{j_q}\}_{j_q\in J\backslash\{\LL-1,\LL\}}, \I_{\LL-1},\I_{\LL}$ are disjoint, we have 
 \[
 |\gamma_J|\ge \sum_{j_q\in J}[ |\I_{j_q}|-1_{j_q=\LL-1}-1_{j_q=\LL}]=|\I_J|-1_{\LL-1\in J}-1_{\LL\in J}\ge |\mathcal\poly'|,
 \]
which completes the proof.\\
 
\item Fix $ r_{\max}+1\le m\le\LL-1$ and consider $\sigma\in\mathcal N_=$ such that $\sigma(y_j)=i_m+2$ where $j$ is such that $y_j\in\aint_{>x_m}$. By Corollary \ref{cor:splitequiv}, $ \sigma(y_j)=i_m+2\le i_{m+1}=\sigma(x_{m+1})$, and since $\sigma(y_j)\neq \sigma(x_{m+1})$ (as $y_j\neq x_{m+1}$), we get that  $ \sigma(y_j)=i_m+2<i_m+3\le  \sigma(x_{m+1})=i_{m+1}$. It follows that $i_m+1<i_{m+1}-1$, so $\sigma(y_j)\in \I_m$. 

For $j_q\in J$, let $\gamma_{j_q}$ be as in part (a), and note that an analogous argument yield
\[
\dim\left(\sum_{\poly\in\mathcal \poly'}F(\poly,-e_j)\right)=|\gamma_J|,
\]
and
\begin{claim}
\label{cl:e}
$~$

\begin{enumerate}[(i)]
\item For $j_q\neq m$, $\I_{j_q}	\subseteq\sigma(\gamma_{j_q})$.
\item For $j_q= m$, $\I_{m}\backslash\{i_m+1,i_m+2\}	\subseteq\sigma(\gamma_m)$.
\end{enumerate}
\end{claim}
In order to complete the proof we  distinguish between two cases: $m\notin J$ and $m\in J$. The proof of the case $m\notin J$ is the same as in part (a). Suppose that $m\in J$ and consider the following cases:
\begin{itemize}
\item $\LL-1,\LL\in J$: The proof is complete since $|\mathcal \poly'|\le |\I_J|-1_{\LL-1\in J}-1_{\LL\in J}=|\I_J|-2$, and since Claim \ref{cl:e} yields $|\gamma_J|\ge |\I_J|-2$.\\

\item $\LL-1\in J$ and $\LL\notin J$: Since $\LL\notin J$, there is an index $j_b$ such that $j_b=\LL-1$ and $j_{b+1}>\LL$, and note that $(j_b,j_{b+1})$ is a splitting pair. Note that since $m\le \LL-1$, and $m\in J$, we must have $m\le j_b$. By Lemma \ref{lem:splitequiv}, there exists $y^{\sigma}\in \bsg_{j_b}\cup \bsg_{j_b+1}$ such that $\sigma(y^{\sigma})\in \I_{\llbracket j_b+1,j_{b+1}-1\rrbracket}$. Suppose  $y^{\sigma}\in \bsg_{j_b}$; the proof for the case $y^{\sigma}\in \bsg_{j_b+1}$ is analogous. Since $\sigma(y^{\sigma})>i_{j_b+1}\ge i_{m+1}> \sigma(y_j)$, we get $y^{\sigma}\in \bsg_{j_b}\backslash\aint_{\le y_j}\subseteq \gamma_{j_b}\subseteq \gamma_J$. Hence, 
\[
(\I_J\backslash\{ i_m+1,i_m+2\})\cup\{\sigma(y^{\sigma})\}\subseteq\sigma(\gamma_J).
\]
Since $\sigma(y^{\sigma})\in \I_{\llbracket j_b+1,j_{b+1}-1\rrbracket}$, we have $\sigma(y^{\sigma})\notin \I_J$ (because $j_b=\LL-1$ and $\LL\notin J$ so the indices $\{j_b+1,\ldots,j_{b+1}-1\}=\llbracket j_b+1,j_{b+1}-1\rrbracket$ are not in $J$), so we get that $|\gamma_J|\ge |\I_J|-1=|\I_J|-1_{\LL-1\in J}-1_{\LL\in J}\ge |\mathcal \poly'|$. \\

\item $\LL-1\notin J$ and $\LL\in J$: The proof is analogous to the case $\LL-1\in J$ and $\LL\notin J$.\\

\item $\LL-1,\LL\notin J$: Since $\LL-1,\LL\notin J$, we can choose $b$ to be an index such that $j_b<\LL-1<\LL<j_{b+1}$, or the largest index such  $j_b<\LL-1<\LL$, and note that $(j_b,j_{b+1})$ is an $\LL$-splitting pair. Note that since $m\le \LL-1$, and $m\in J$, we must have $m\le j_b$. Consider the collection 
\[
\mathcal \poly'':=(\mathcal\poly_0,\ldots, \mathcal\poly_{j_b},\mathcal\poly_{j_b+1},\ldots,\mathcal\poly_k)
\]
and note that, by Assumption \ref{ass:crit}, $\mathcal \poly''$ is critical. We claim that $\mathcal \poly''$ is  in fact  supercritical. Indeed, if $\mathcal \poly''$ is sharp-critical, then $j_b\le r_{\max}$. But $j_b\ge m>r_{\max}$, so we get a contradiction. Since  $\mathcal \poly''$ is supercritical, and since $(j_b,j_{b+1})$ is an $\LL$-splitting pair, Corollary \ref{cor:Ibetastrong} provides  two distinct $y^{\sigma},z^{\sigma}\in\bsg_{j_b}\cup\bsg_{j_{b+1}}$, with $\sigma(y^{\sigma}), \sigma(z^{\sigma})\in \I_{\llbracket j_b+1,j_{b+1}-1\rrbracket}$, from which it follows that 
\[
\I_{\llbracket 0,j_b\rrbracket\cup \llbracket j_{b+1},k\rrbracket}\cup\{\sigma(y^{\sigma}),\sigma(z^{\sigma})\}\subseteq \sigma(\bsg_{\llbracket 0,j_b\rrbracket\cup \llbracket j_{b+1},k\rrbracket}).
\]
Suppose that $y^{\sigma}\in \bsg_{j_b}$; the case $y^{\sigma}\in \bsg_{j_{b+1}}$ is analogous. Since $m\le j_b$, $\sigma(y^{\sigma})>i_{j_b+1}\ge i_{m+1}> \sigma(y_j)$, so we can conclude that $y^{\sigma}\in\bsg_{j_b}\backslash\aint_{\le y_j}\subseteq \gamma_{j_b}\subseteq\gamma_J$. Analogous argument shows that $z^{\sigma}\in \gamma_J$. By Claim \ref{cl:e}, it follows that
\[
(\I_J\backslash\{ i_m+1,i_m+2\})\cup\{\sigma(y^{\sigma}),\sigma(z^{\sigma})\}\subseteq\sigma(\gamma_J).
\]
Since $\sigma(y^{\sigma}),\sigma(z^{\sigma})\in \I_{\llbracket j_b+1,j_{b+1}-1\rrbracket}$, we have $\sigma(y^{\sigma}),\sigma(z^{\sigma})\notin \I_J$ (because $b$ satisfies $j_b<\LL-1<\LL<j_{b+1}$, or the maximal $j_b<\LL-1$, so the indices $\{j_b+1,\ldots,j_{b+1}-1\}=\llbracket j_b+1,j_{b+1}-1\rrbracket$ are not in $J$). On the other hand, because $m\in J$ and $i_m+2<i_{m+1}$, we have $i_m+1,i_m+2\in \I_J$. It follows that $(\I_J\backslash\{ i_m+1,i_m+2\})\cup\{\sigma(y^{\sigma}),\sigma(z^{\sigma})\}|=|\I_J|$, and hence, $|\gamma_J|\ge |\I_J|\ge |\mathcal \poly'| $.
\end{itemize}
$~$\\

\item  The proof is analogous to part (e).\\

\item Consider $\sigma\in\mathcal N_+$ such that $\sigma(y_j)=i_{\LL-1}+2$ where $j$ is such that $y_j\in\aint_{>x_{\LL-1}}$. By Corollary \ref{cor:splitequiv}, $\sigma(x_{\LL-1})<i_{\LL-1}+2=\sigma(y_j)<i_{\LL}+1=\sigma(x_{\LL})$, so we  conclude that $y_j\in\bsg_{\LL-1}$. For $j_q\in J$ let $\gamma_{j_q}$ be as in part (a), and note that an analogous argument yields $\dim\left(\sum_{\poly\in\mathcal \poly'}F(\poly,-e_j)\right)=|\gamma_J|$. We start with the analogue of Claim \ref{cl:d2}. 

\begin{claim}
\label{cl:g}
$~$
\begin{enumerate}[(i)]
\item For $j_q\notin\{\LL-1,\LL\}$, $\I_{j_q}	\subseteq\sigma(\gamma_{j_q})$.
\item  For $j_q=\LL-1$, $(\I_{\LL-1}\cup\{i_{\LL}\})\backslash\{i_{\LL-1}+1,i_{\LL-1}+2\}	\subseteq\sigma(\gamma_{\LL-1})$.
\item  For $j_q=\LL$, $\I_{\LL}\backslash\{i_{\LL}+1\}	\subseteq\sigma(\gamma_{\LL})$.
\end{enumerate}
\end{claim}

\begin{proof}
There two cases to consider: (1) $y_j\notin\beta_{j_q}$ and (2) $y_j\in\beta_{j_q}$.
\begin{enumerate}[(i)]
\item Case (1): By Lemma \ref{lem:Ibeta}, $\I_{j_q}\subseteq \sigma(\beta_{j_q})=\sigma(\gamma_{j_q})$. 

Case (2): First we note that $j_q\ge \LL-1$ since, otherwise, $y_j>x_{\LL-1}\ge x_{j_q+1}$ which contradicts $y_j\in\beta_{j_q}$. Since  $j_q\notin\{\LL-1,\LL\}$, it follows that in fact $\LL<j_q$. Hence, for any $y$ such that $\sigma(y)\in \I_{j_q}$, we have $\sigma(y)>\sigma(x_{j_q})\ge \sigma(x_{\LL+1})=i_{\LL+1}>i_{\LL-1}+2=\sigma(y_j)$, so that $y\notin\aint_{\le y_j}$. It follows that $y\in \gamma_{j_q}$, so we conclude $\I_{j_q}	\subseteq\sigma(\gamma_{j_q})$.

\item Case (1) cannot occur since we have shown that $y_j\in\bsg_{\LL-1}$. 

Case (2): Every $y$ such that $\sigma(y)\in(\I_{\LL-1}\cup\{i_{\LL}\})\backslash\{i_{\LL-1}+1,i_{\LL-1}+2\}$ satisfies $\sigma(x_{\LL-1})<\sigma(y)<\sigma(x_{\LL})$, so $y\in\bsg_{\LL-1}$. Further, $\sigma(y)>i_{\LL-1}+2=\sigma(y_j)$, so  $y\notin\aint_{\le y_j}$. It follows that $y\in \gamma_{\LL-1}$, so we conclude $(\I_{\LL-1}\cup\{i_{\LL}\})\backslash\{i_{\LL-1}+1,i_{\LL-1}+2\}	\subseteq\sigma(\gamma_{\LL-1})$.

\item Case (1): Every $y$ such that $\sigma(y)\in \I_{\LL}\backslash\{i_{\LL}+1\}$ satisfies $\sigma(x_{\LL})=i_{\LL}+1<\sigma(y)<i_{\LL+1}=\sigma(x_{\LL+1})$, so $y\in\bsg_{\LL}=\gamma_{\LL}$. We conclude that $\I_{\LL}\backslash\{i_{\LL}+1\}	\subseteq\sigma(\gamma_{\LL})$.

Case (2): Every $y$ such that $\sigma(y)\in \I_{\LL}\backslash\{i_{\LL}+1\}$ satisfies $\sigma(x_{\LL})=i_{\LL}+1<\sigma(y)<i_{\LL+1}=\sigma(x_{\LL+1}$, so $y\in\bsg_{\LL}$. Further, $\sigma(y)>i_{\LL}+1>i_{\LL-1}+2=\sigma(y_j)$, so  $y\notin\aint_{\le y_j}$. It follows that $y\in \gamma_{\LL}$, so we conclude $ \I_{\LL}\backslash\{i_{\LL}+1\}	\subseteq\sigma(\gamma_{\LL})$.
\end{enumerate}
\end{proof}
By  \eqref{eq:linspan} $\lin\left(\sum_{\poly\in\mathcal \poly'}F(\poly,-e_j)\right)=\R^{\gamma_J}$, so $\dim\left(\sum_{\poly\in\mathcal \poly'}F(\poly,-e_j)\right)= |\gamma_J|$. By Claim \ref{cl:g}, using the fact that $\{\I_{j_q}\}_{j_q\in J\backslash\{\LL-1,\LL\}}, \I_{\LL-1},\I_{\LL}$ are disjoint, it suffices to show that $|(\I_{\LL-1}\cup\{i_{\LL}\})\backslash\{i_{\LL-1}+1,i_{\LL-1}+2\}|=|\I_{\LL-1}|-1$, and that $|\I_{\LL}\backslash\{i_{\LL}+1\}|=|\I_{\LL}|-1$, since then
\[
 |\gamma_J|\ge \sum_{j_q\in J}[ |\I_{j_q}|-1_{j_q=\LL-1}-1_{j_q=\LL}]=|\I_J|-1_{\LL-1\in J}-1_{\LL\in J}\ge |\mathcal\poly'|,
 \]
which completes the proof. To see that $|(\I_{\LL-1}\cup\{i_{\LL}\})\backslash\{i_{\LL-1}+1,i_{\LL-1}+2\}|=|\I_{\LL-1}|-1$, we note that $|\I_{\LL-1}\cup\{i_{\LL}\}|=|\I_{\LL-1}|+1$, and that $i_{\LL-1}+1,i_{\LL-1}+2\in \I_{\LL-1}\cup\{i_{\LL}\}$, because $i_{\LL-1}+1<i_{\LL-1}+2\le i_{\LL}$, by Corollary \ref{cor:splitequiv}. Hence, $|(\I_{\LL-1}\cup\{i_{\LL}\})\backslash\{i_{\LL-1}+1,i_{\LL-1}+2\}|=(|\I_{\LL-1}|+1)-2=|\I_{\LL-1}|-1$. Finally, it is clear that $|\I_{\LL}\backslash\{i_{\LL}+1\}|=|\I_{\LL}|-1$, since $i_{\LL}+1\in \I_{\LL}$. \\

\item  The proof is analogous to part (g).
\end{enumerate}
\end{proof}

\section{Supercritical posets}
\label{sec:supercrit}
In this section we complete the characterization of the extremals of Stanley's inequalities for supercritical posets. The following result, together with Proposition \ref{prop:suff}, Lemma \ref{lem:suff}, Proposition \ref{prop:clposet}, and Proposition \ref{prop:critequiv}, complete the proof of Theorem \ref{thm:supcrit}.

\begin{theorem}
\label{thm:supcritsec}
Suppose that $\mathcal\poly$ is supercritical and that $|\mathcal N_{=}|^2=|\mathcal N_{-}||\mathcal N_{+}|$. Then,
\[|\mathcal N_=(\incomp,\sim)|=|\mathcal N_=(\sim,\incomp)|=|\mathcal N_=(\sim,\sim)|=0.
\]
\end{theorem}

In order to prove Theorem \ref{thm:supcritsec}, we will invoke Theorem \ref{thm:SvHComb} and use the extreme normal directions found in Proposition \ref{prop:dir}(a--d). Theorem \ref{thm:SvHComb} tells us that there exist $a\ge 0$ and $\mathrm v\in \R^{n-k}$ such that
\begin{align}
\label{eq:exth}
h_{\poly_{\LL-1}}(\mathrm u)=h_{a\poly_{\LL}+\mathrm v}(\mathrm u) \quad \mbox{for all }(B,\mathcal \poly)\textnormal{-extreme normal directions } \mathrm u. 
\end{align}
The following results derive constraints from \eqref{eq:exth} on the allowed $a$ and $\mathrm v$. We start with $\mathrm v$.
\begin{proposition}
$~$

\label{prop:av}
\begin{enumerate}[(a)]
\item For each fixed $0\le m\le\LL-1$: $\mathrm v_j=0$ for any $j$ such that $y_j\in\aint_{>x_m}$ and there exists $\sigma\in\mathcal N_=$ satisfying $\sigma(y_j)=i_m+1$.\\

\item For each fixed $\LL+1\le m\le k+1$: $\mathrm v_j=1-a$ for any $j$ such that $y_j\in\aint_{<x_m}$ and  there exists $\sigma\in\mathcal N_=$ satisfying $\sigma(y_j)=i_m-1$.\\

\item $\mathrm v_u=\mathrm v_v$ for any $u,v$ such that  $y_u<y_v$ and there exists $\sigma\in\mathcal N_=$ satisfying $\sigma(y_u)+1=\sigma(y_v)$.\\

\item $\mathrm v_u=\mathrm v_v$ for any $u,v$ such that $y_u<y_v$ and  there exists $\sigma\in\mathcal N_=$ satisfying $\sigma(y_u)=i_{\LL}-1$ and $\sigma(y_v)=i_{\LL}+1$.
\end{enumerate}
\end{proposition}

\begin{proof}
$~$

\begin{enumerate}[(a)]
\item  By Proposition \ref{prop:dir}(a), $-e_j$ is a $(B,\mathcal \poly)$-extreme normal direction, so by \eqref{eq:exth}, $h_{\poly_{\LL-1}}(-e_j)=ah_{\poly_{\LL}}(-e_j)-\mathrm v_j$. Since $\sigma(y_j)=i_m+1$, and $m\le \LL-1$, we have $\sigma(y_j)=i_m+1\le i_{\LL-1}+1<i_{\LL},i_{\LL+1}$, so $y_j\notin\aint_{>x_{\LL}}\cup \aint_{>x_{\LL+1}}$. Hence, it follows from \eqref{eq:Kiordpoly} that $h_{\poly_{\LL-1}}(-e_j)=h_{\poly_{\LL}}(-e_j)=0$. We conclude that $\mathrm v_j=0$. \\

\item By Proposition \ref{prop:dir}(b), $e_j$ is a $(B,\mathcal \poly)$-extreme normal direction, so by \eqref{eq:exth}, $h_{\poly_{\LL-1}}(-e_j)=ah_{\poly_{\LL}}(-e_j)+\mathrm v_j$. Since $\sigma(y_j)=i_m-1$, and $m\ge \LL+1$, we have $\sigma(y_j)=i_m-1\ge i_{\LL+1}-1>i_{\LL},i_{\LL-1}$, so $y_j\notin\aint_{<x_{\LL-1}}\cup \aint_{<x_{\LL}}$. Hence, it follows from \eqref{eq:Kiordpoly} that $h_{\poly_{\LL-1}}(e_j)=1$ and $ah_{\poly_{\LL}}(e_j)+\mathrm v_j=a+\mathrm v_j$. We conclude that $\mathrm v_j=1-a$. \\

\item By Proposition \ref{prop:dir}(c), $e_{uv}$ is a $(B,\mathcal \poly)$-extreme normal direction, so by \eqref{eq:exth}, $h_{\poly_{\LL-1}}(e_{uv})=ah_{\poly_{\LL}}(e_{uv})+\frac{1}{\sqrt{2}} (\mathrm v_u-\mathrm v_v)$. We will show that $h_{\poly_{\LL-1}}(e_{uv})=h_{\poly_{\LL}}(e_{uv})=0$, from which we can conclude $\mathrm v_u=\mathrm v_v$. We will show that $h_{\poly_{\LL-1}}(e_{uv})=0$; the proof of $h_{\poly_{\LL}}(e_{uv})=0$ is analogous. We distinguish between the following cases:

Case (1): $y_u,y_v\in\bsg_{\LL-1}$.  By \eqref{eq:Kiordpoly},  $h_{\poly_{\LL-1}}(e_{uv})=0$ since $t_u\le t_v$ for $t\in O_{\bsg_{\LL-1}}$, and equality is attained with $t=0$. 

Case (2): $y_u\in\bsg_{\LL-1},y_v\notin\bsg_{\LL-1}$, or $y_u\notin\bsg_{\LL-1},y_v\in\bsg_{\LL-1}$.  See the proof of Claim \ref{cl:c1}.

Case (3): $y_u,y_v\notin\bsg_{\LL-1}$. Since there exists $\sigma\in\mathcal N_=$ with $\sigma(y_u)+1=\sigma(y_v)$, the assumption $y_u,y_v\notin\bsg_{\LL-1}$ implies that either $y_u,y_v<x_{\LL-1}$, or $y_u,y_v>x_{\LL}$. Hence, either $t_u=t_v=1$, or $t_u=t_v=0$ for any $t\in\poly_{\LL-1}$, so, in particular, $h_{\poly_{\LL-1}}(e_{uv})=0$.\\

\item The proof is analogous to part (c), where we note that $y_u\notin\bsg_{\LL-1}$ cannot occur. 
\end{enumerate}
\end{proof}
While Proposition \ref{prop:av}(a--b) took care of elements neighboring $x_m$'s, the next result takes care of elements that are at the bottom (res. the top) of the poset.

\begin{lemma}
\label{lem:minmax}
For any $y_j\in\aint$: If $m^{=}_{\min}(y_j)<i_{\LL}$ then $\mathrm v_j=0$, and if $m^{=}_{\max}(y_j)>i_{\LL}$ then $\mathrm v_j=1-a$.
\end{lemma}

\begin{proof}
We prove that $m^{=}_{\max}(y_j)>i_{\LL}\Rightarrow \mathrm v_j=1-a$; the proof of $m^{=}_{\min}(y_j)<i_{\LL}\Rightarrow \mathrm v_j=0$ is analogous.

Set $y_{j_0}:=y_j$ and construct the sequence $y_{j_0}<y_{j_1}<\cdots <y_{j_p}$, for some $p<\infty$, iteratively, according to the algorithm below. The sequence will be constructed so that $y_{j_i}\in\aint$ for every $i\in\llbracket 0,p\rrbracket$, $\mathrm v_{j_i}=\mathrm v_{j_{i+1}}$ for all $i\in \llbracket 0,p-1\rrbracket$, and $\mathrm v_{j_p}=1-a$. Clearly, it will then follow that $\mathrm v_j=\mathrm v_{j_0}=1-a$, completing the proof.

Assume that the sequence $y_{j_0}<y_{j_1}<\cdots <y_{j_i}$ has been constructed. Set $M:=m^{=}_{\max}(y_{j_i})$, and note that $i_{\LL}<m^{=}_{\max}(y_{j_0})\le M$. Consider the following two cases:

\begin{itemize}
\item $M\neq i_m-1$ for every $\LL<m$: Choose $\sigma\in\mathcal N_=$ such that $\sigma(y_{j_i})=M$ (such a $\sigma$ must exist by the definition of $M$) and set $y_{j_{i+1}}:=\sigma^{-1}(M+1)$. We first show that $M+1\neq i_m$ for any $m\in\llbracket 0,k\rrbracket$. Indeed, by assumption $M+1\neq i_m$ for every $\LL<m$, and if $m\le \LL$, then $i_m\le i_{\LL}<M+1$. It follows that $y_{j_{i+1}}\in\aint$. Next we show that $y_{j_i}<y_{j_{i+1}}$. Indeed, otherwise, by the definition of $M$, $y_{j_i}$ and $y_{j_{i+1}}$ must be incomparable, so we can swap the positions of $y_{j_i}$ and $y_{j_{i+1}}$ in $\sigma$ to get $\sigma'\in\mathcal N_=$ such that $\sigma'(y_{j_i})=M+1$, which contradicts the maximality of $M$. We conclude that $y_{j_i}<y_{j_{i+1}}$. Finally,  by Proposition \ref{prop:av}(c), $\mathrm v_{j_i}=\mathrm v_{j_{i+1}}$.\\

\item $M= i_m-1$ for some $\LL<m$: In this case, the sequence will be terminated with $p:=i$. Note that Corollary \ref{cor:splitequiv} implies that $y_{j_i}\in\aint$, since $M=i_m-1$. We will show that $\sigma(y_{j_i})<\sigma(x_m)$ for all $\sigma\in\cup_{\s\in\{-,=,+\}}\mathcal N_{\s}$. Then, by Assumption \ref{ass:cl}, it follows that $y_{j_i}<x_m$ so, by Proposition \ref{prop:av}(b), $\mathrm v_{j_i}=1-a$. To show that that $\sigma(y_{j_i})<\sigma(x_m)$ for all $\sigma\in\cup_{\s\in\{-,=,+\}}\mathcal N_{\s}$, suppose for contradiction otherwise, which means that there exists $\sigma\in\mathcal N_{\s}$, for some $\s\in\{-,=,+\}$, such that $\sigma(y_{j_i})>\sigma(x_m)=i_m$. Set $q:=\sigma(y_{j_i})$. We will show that Lemma \ref{lem:range} can be applied with $y_{j_i}$, $=$, and $q$, to yield $\sigma'\in\mathcal N_=$ such that $\sigma'(y_{j_i})=q$, contradicting the maximality of $M$ (since $q>i_m>i_m-1=M$).

To apply Lemma \ref{lem:range} to $y_{j_i}$, $=$, and $q$, we need to check that all of the conditions of the lemma are satisfied. Applying the lemma to $y_{j_i}$, $\s$, and $q$, we get $q\le u_{\s}(y_{j_i})$, and by Lemma \ref{lem:UL} (as $i_{\min}(y_{j_i})>m>\LL$), we get $q\le u_{\s}(y_{j_i})=u_{=}(y_{j_i})$. On the other hand, by Corollary \ref{cor:range}, $l_{=}(y_{j_i})\le m^{=}_{\max}(y_{j_i})=M=i_m-1<q$. We conclude that the condition $q\in\llbracket l_{=}(y_{j_i}),u_{=}(y_{j_i})\rrbracket$ holds. Finally, we show that $q\neq i_r+1_{=}=i_r$ for any $r\in\llbracket 1,k\rrbracket$. Indeed, if $q=i_r$ for some $r\in\llbracket 1,k\rrbracket$, then $i_r=q>i_m$, which implies $\LL<m<r$, and hence $\sigma(x_r)=i_r$ as $r\neq \LL$. It follows that $\sigma(y_{j_i})=q=\sigma(x_r)$, contradicting $y\in\aint$. 
\end{itemize}
\end{proof}

Next we move to $a$.

\begin{lemma}
\label{lem:a=1}
$a=1$.
\end{lemma}
\begin{proof}
Fix $\sigma\in \mathcal N_{=}$ and set $y_u^{\sigma}:=\sigma^{-1}(i_{\LL}-1)$, $y_v^{\sigma}:=\sigma^{-1}(i_{\LL}+1)$. There are a few cases to check:\\

\begin{itemize}
\item $y_u^{\sigma}\incomp x_{\LL}$: If $m_{\max}^=(y_u^{\sigma})>i_{\LL}$, then, since $m_{\min}^=(y_u^{\sigma})\le \sigma (y_u^{\sigma})<i_{\LL}$, Lemma \ref{lem:minmax} implies that $\mathrm v_u=0$ and $\mathrm v_u=1-a$ so $a=1$. 
Suppose then that $m_{\max}^=(y_u^{\sigma})<i_{\LL}$. We claim that $u_=(y_u^{\sigma})\le i_{\LL}$. Indeed, otherwise,  $u_=(y_u^{\sigma})\ge i_{\LL}+1\ge \sigma(y_u^{\sigma})\ge l_=(y_u^{\sigma})$. Hence, since $i_{\LL}+1\neq \sigma(x_m)$ for any $m\in\llbracket 1,k\rrbracket$, Lemma \ref{lem:range} implies that there exists $\sigma'\in\mathcal N_=$ satisfying $\sigma'(y_u^{\sigma})=i_{\LL}+1$, which contradicts the maximality of $m_{\max}^=(y_u^{\sigma})<i_{\LL}$. Now, since $u_=(y_u^{\sigma})\le i_{\LL}$, there must exist $b$, with $y_u^{\sigma}<x_b$, such that $i_b^{\s}-|\abar_{\ge y_{u}^{\sigma},<x_b}|\le i_{\LL}$. It follows that $|\abar_{\ge y_{u}^{\sigma},<x_b}|\ge i_b-i_{\LL}$, where  we used $i_b^{=}=i_b$. Fix $z\in \abar_{>y_{u}^{\sigma},<x_b}$, and note that $z\neq x_{\LL}$, since otherwise $x_{\LL}>y_{u}^{\sigma}$, which contradicts the assumption $y_u^{\sigma}\incomp x_{\LL}$. In particular, since $\sigma(x_{\LL})=i_{\LL}$, we have $\sigma(z)\neq i_{\LL}$. Since $i_{\LL}-1=\sigma(y_{u}^{\sigma})<\sigma(z)<\sigma(x_b)=i_b$, we conclude that $\sigma(z)\in \llbracket i_{\LL}+1,i_b-1\rrbracket$. The size of $ \llbracket i_{\LL}+1,i_b-1\rrbracket$ is $i_b-i_{\LL}-1$, so combining $|\abar_{> y_{u}^{\sigma},<x_b}|\ge i_b-i_{\LL}-1$, with $\sigma(z)\in \llbracket i_{\LL}+1,i_b-1\rrbracket$ for every  $z\in \abar_{>y_{u}^{\sigma},<x_b}$, shows that $\sigma(\abar_{>y_{u}^{\sigma},<x_b})=\llbracket i_{\LL}+1,i_b-1\rrbracket$. In particular, since $\sigma(y_v^{\sigma})=i_{\LL}+1$, we get that $y_v^{\sigma}\in \abar_{>y_{u}^{\sigma},<x_b}$, so $y_{v}^{\sigma}<y_{u}^{\sigma}$. It follows from Proposition \ref{prop:av}(c) that $\mathrm v_u=\mathrm v_v$. Since $m_{\min}^=(y_u^{\sigma})<i_{\LL}$, and $m_{\max}^=(y_v^{\sigma})>i_{\LL}$, Lemma \ref{lem:minmax} yields $0=\mathrm v_u=\mathrm v_v=1-a$. We conclude that $a=1$.\\

\item  $y_v^{\sigma}\incomp x_{\LL}$: Analogous to the case $y_u^{\sigma}\incomp x_{\LL}$.\\

\item $y_u^{\sigma}<x_{\LL}$ and $x_{\LL}<y_v^{\sigma}$: By Proposition \ref{prop:av}(d), $\mathrm v_u=\mathrm v_v$. Since $y_u^{\sigma}<x_{\LL}$, we have $m_{\min}^=(y_u^{\sigma})<i_{\LL}$ so, by Lemma \ref{lem:minmax}, $\mathrm v_u=0$. Similarly, since $x_{\LL}<y_v^{\sigma}$, we have $m_{\max}^=(y_v^{\sigma})>i_{\LL}$ so, by Lemma \ref{lem:minmax}, $\mathrm v_v=1-a$.  We conclude that $0=\mathrm v_u=\mathrm v_v=1-a$, so $a=1$. 
\end{itemize}
\end{proof}
We are now ready to prove Theorem \ref{thm:supcritsec}.
\begin{proof}{\textnormal{(of Theorem \ref{thm:supcritsec})}}
We will show that
\begin{align}
\label{eq:supcritequiv}
\forall~ \sigma\in\mathcal N_=:\quad \sigma^{-1}(i_{\LL}-1)\incomp x_{\LL}\quad\text{and}\quad \sigma^{-1}(i_{\LL}+1)\incomp x_{\LL},
\end{align}
which is equivalent to $|\mathcal N_=(\incomp,\sim)|=|\mathcal N_=(\sim,\incomp)|=|\mathcal N_=(\sim,\sim)|=0$.

Let $y_j\in\aint$ be any element such that there exists $\sigma\in\mathcal N_=$ with $\sigma(y_j)=i_{\LL}+1$; the proof for elements $y_j\in\aint$ with $\sigma\in\mathcal N_=$ satisfying $\sigma(y_j)=i_{\LL}-1$ is analogous. Since $m^{=}_{\max}(y_j)>i_{\LL}$, Lemma \ref{lem:minmax} yields $\mathrm v_j=1-a=0$, where the last equality follows from Lemma \ref{lem:a=1}. Assume for contradiction that $x_{\LL}$ is comparable to $y_j$, which, by the assumption $\sigma(y_j)=i_{\LL}+1$, means that $x_{\LL}<y_j$. By Proposition \ref{prop:dir}(a), $-e_j$ is a $(B,\mathcal \poly)$-extreme normal direction so, by \eqref{eq:exth},
$h_{\poly_{\LL-1}}(-e_j)=h_{\poly_{\LL}}(-e_j)$. Since $x_{\LL}<y_j$, we have $h_{\poly_{\LL-1}}(-e_j)=-1$. On the other hand, $i_{\LL}<\sigma(y_j)=i_{\LL}+1<i_{\LL+1}$, so $y_j\in\bsg_{\LL}$. By \eqref{eq:Kiordpoly}, $h_{\poly_{\LL}}(-e_j)=0\neq -1=h_{\poly_{\LL-1}}(-e_j)$, so we have arrived at the desired contradiction.
\end{proof}

\section{Critical posets}
\label{sec:crit}
In this section we complete the characterization of the extremals of Stanley's inequalities for critical posets (as well as Theorem \ref{thm:N_=>0intro}). We will assume that $\mathcal \poly$ is sharp-critical since, otherwise, we reduce back to the supercritical setting. We note that the assumption that $\mathcal \poly$ is sharp-critical implies, by Proposition \ref{prop:max_notions}, that the maximal sharp-critical collection $\mathcal\poly_{\max}$, with its associated splitting pair $(r_{\max},s_{\min})$, exist.  The following result (Theorem \ref{thm:critsec}), together with Proposition \ref{prop:suff},  Lemma \ref{lem:suff}, Proposition \ref{prop:clposet}, and Proposition \ref{prop:critequiv}, complete the proof of Theorem \ref{thm:crit}. 

The proof of Theorem \ref{thm:N_=>0intro} follows by Corollary \ref{cor:referee}, and by applying Theorem \ref{thm:subcrit} repeatedly until arriving at a critical subposet. Once a critical subposet is reached, Theorem \ref{thm:crit} can be applied to the critical subposet, together with the bijection construction in the proof of Proposition \ref{prop:split}, to conclude that the results of Theorem \ref{thm:crit} hold for the original poset as well. 

\begin{theorem}
\label{thm:critsec}
Suppose that $\mathcal\poly$ is sharp-critical and that $|\mathcal N_{=}|^2=|\mathcal N_{-}||\mathcal N_{+}|$. Then,
\[
|\mathcal N_-(\sim,\sim)|=|\mathcal N_+(\sim,\sim)|=0.
\]
\end{theorem}

\subsection{The critical subspace}
We now enter the critical territory so the equation
\[
h_{\poly_{\LL-1}}(\mathrm u)=h_{a\poly_{\LL}+\mathrm v}(\mathrm u) \quad \mbox{for all }(B,\mathcal \poly)\textnormal{-extreme normal directions } \mathrm u,
\]
which held for supercritical posets, is no longer valid. Instead, we only have
\[
h_{\poly_{\LL-1}+\sum_{j=1}^dQ_j}(\mathrm u)=h_{a\poly_{\LL}+\mathrm v+\sum_{j=1}^dP_j}(\mathrm u) \quad \mbox{for all }(B,\mathcal \poly)\textnormal{-extreme normal directions } \mathrm u,
\]
where $(P_1,Q_1),\ldots,(P_d,Q_d)$ are $\mathcal \poly$-degenerate pairs. Our approach to this problem is to find a  subspace $E^{\perp}$, on which we do in fact have $h_{\poly_{\LL-1}}(\mathrm u)=h_{a\poly_{\LL}+\mathrm v}(\mathrm u)$ for all $(B,\mathcal \poly)$-extreme normal directions $\mathrm u\in E^{\perp}$. Since we now require that the $(B,\mathcal \poly)$-extreme normal directions are contained in $E^{\perp}$, we will need more of them in order to derive enough constraints to characterize the extremals of critical posets. These extreme normal directions are the ones given in Proposition \ref{prop:dir}(e--h). We define the subspace $E^{\perp}$ by
\begin{align}
\label{eq:Eperp}
E^{\perp}:=\R^{\aint\backslash \bsg_{\max}},
\end{align}
where we recall \eqref{eq:Kmax}. We call the subspace $E$ the \emph{critical subspace} and note that, by Lemma \ref{lem:spancollec}, $\lin(\mathcal\poly_{\max})=\R^{\bsg_{\max}}=E$. The following result explains the connection between  $\mathcal\poly$-degenerate pairs and $E$. 

\begin{lemma}
\label{lem:deg}
Let $(P,Q)$ be a $\mathcal\poly$-degenerate pair. Then, $\lin(P),\lin(Q)\subseteq E$.
\end{lemma}
\begin{proof}
The result follows by \cite[Lemma 9.6]{SvH20} and  Proposition \ref{prop:max_notions}.
\end{proof}

When we restrict to the subspace $E^{\perp}$, we are in the supercritical case in the following sense:
\begin{lemma}
\label{lem:supcritE}
There exist $a\ge 0$ and $\mathrm v\in S^{n-k-1}$ such that
\[
h_{\poly_{\LL-1}}(\mathrm u)=h_{a\poly_{\LL}+\mathrm v}(\mathrm u) \quad \mbox{for all }(B,\mathcal \poly)\textnormal{-extreme normal directions } \mathrm u\text{ which are contained in $E^{\perp}$}.
\]
\end{lemma}
\begin{proof}
Let $\mathrm u\in E^{\perp}$ be a $(B,\mathcal \poly)$-extreme normal direction. By Theorem \ref{thm:SvHComb}
\[
h_{\poly_{\LL-1}+\sum_{j=1}^dQ_j'+\sum_{j=1}^dq_j}(\mathrm u)=h_{a\poly_{\LL}+\mathrm v+\sum_{j=1}^dP_j'+\sum_{j=1}^dp_j}(\mathrm u),
\]
where $(P_j,Q_j)_{j\in\llbracket 1,d\rrbracket}$ are $\mathcal \poly$-degenerate pairs and $P_j'=P_j-p_j$, $Q_j'=Q_j-q_j$ where  $p_j\in P_j$, $q_j\in Q_j$ are fixed. Hence, with $\mathrm v':=v+\sum_{j=1}^dp_j-\sum_{j=1}^dq_j,~\tilde P:=\sum_{j=1}^dP_j'$, and $\tilde Q:=\sum_{j=1}^dQ_j'$, we have
\[
h_{\poly_{\LL-1}+\tilde Q}(\mathrm u)=h_{a\poly_{\LL}+\mathrm v'+\tilde P}(\mathrm u).
\]
Since $\tilde P,\tilde Q\subseteq E$ and $\mathrm u\in E^{\perp}$, we have $h_{\tilde Q}(\mathrm u)=h_{\tilde P}(\mathrm u)=0$. Relabeling $\mathrm v'\to\mathrm v$ completes the proof.
\end{proof}

\subsection{The critical extremals}
In order to prove Theorem \ref{thm:critsec} we need to prove the analogues of Proposition \ref{prop:av}, Lemma \ref{lem:minmax}, and Lemma \ref{lem:a=1}, as well as some additional results. Roughly speaking, on 
\[
\aint\backslash\bsg_{\max}=\aint_{>x_{r_{\max}+1},<x_{s_{\min}}},
\]
we have a supercritical behavior. Indeed, the proof of the following result is analogous to the proof of Proposition \ref{prop:av} once we use the full power of Proposition \ref{prop:dir}, Lemma \ref{lem:supcritE}, and restrict to  $y_j,y_u,y_v\in \aint_{>x_{r_{\max}+1},<x_{s_{\min}}}$, rather than allowing for all $y_j,y_u,y_v\in \aint$.
\begin{proposition}
\label{prop:avcrit}
For any $y_j,y_u,y_v\in \aint_{>x_{r_{\max}+1},<x_{s_{\min}}}$: 
\begin{enumerate}[(a)]
\item For each fixed $0\le m\le \LL-1$: $\mathrm v_j=0$ for any $j$ such that $y_j\in\aint_{>x_m}$ and there exists $\sigma\in\mathcal N_=$ satisfying $\sigma(y_j)=i_m+1$.\\

\item For each fixed $\LL+1\le m\le k+1$: $\mathrm v_j=1-a$ for any $j$ such that $y_j\in\aint_{<x_m}$ and there exists $\sigma\in\mathcal N_=$ satisfying $\sigma(y_j)=i_m-1$.\\

\item $\mathrm v_u=\mathrm v_v$ for any $u,v$ such that $y_u<y_v$ and there exists $\sigma\in\mathcal N_=$ satisfying $\sigma(y_u)+1=\sigma(y_v)$.\\

\item $\mathrm v_u=\mathrm v_v$ for any $u,v$ such that $y_u<y_v$ and there exists $\sigma\in\mathcal N_=$ satisfying $\sigma(y_u)=i_{\LL}-1$ and $\sigma(y_v)=i_{\LL}+1$.\\

\item For each fixed $ r_{\max}\le m\le\LL-1$: $\mathrm v_j=0$ for any $j$ such that  $y_j\in\aint_{>x_m}$ and there exists $\sigma\in\mathcal N_=$ satisfying either $\sigma(y_j)=i_m+1$ or $\sigma(y_j)=i_m+2$.\\

\item For each fixed $\LL+1\le m\le s_{\min}$: $\mathrm v_j=1-a$ for any $j$ such that  $y_j\in\aint_{<x_m}$ and there exists $\sigma\in\mathcal N_=$ satisfying either $\sigma(y_j)=i_m-1$ or $\sigma(y_j)=i_m-2$.\\

\item $\mathrm v_j=0$ for any $j$ such that $y_j\in\aint_{>x_{\LL-1}}$ and there exists $\sigma\in\mathcal N_+$ satisfying $\sigma(y_j)=i_{\LL-1}+2$.\\

\item $\mathrm v_j=1-a$ for any $j$ such that $y_j\in\aint_{<x_{\LL+1}}$ and there exists $\sigma\in\mathcal N_-$ satisfying $\sigma(y_j)=i_{\LL+1}-2$.\\
\end{enumerate}
\end{proposition}

Towards the proofs of the analogues of Lemma \ref{lem:minmax} and Lemma \ref{lem:a=1} we recall Corollary \ref{cor:max}, together with some of its immediate consequences.

\begin{corollary}
\label{cor:max_plus}
Fix $\s\in \{-,=,+\}$ and $\sigma\in\mathcal N_{\s}$. There exists a unique mixed element $\mixed$  satisfying $\mixed \in \bsg_{r_{\max}}\cup\bsg_{s_{\min}}$ and $\sigma(\mixed)\in \llbracket i_{r_{\max}+1}, i_{s_{\min}} \rrbracket\backslash\{i_{r_{\max}+1},\ldots,i_{s_{\min}}\}$. In particular, any other element $y\neq \mixed$ satisfying $\sigma(y) \in \llbracket i_{r_{\max}+1}, i_{s_{\min}} \rrbracket$ must satisfy $y\in\abar_{\ge x_{r_{\max}+1}, \le x_{s_{\min}}}$.  Furthermore, $\mixed$ satisfies either $\mixed \not \ge x_{r_{\max}+1}$ or $\mixed \not \le x_{s_{\min}}$. If $\mixed \not \ge x_{r_{\max}+1}$, then $\mixed \not \ge y$ for any $y\in \abar_{\ge x_{r_{\max}+1}, \le x_{s_{\min}}}$. Analogously, if $\mixed \not \le x_{s_{\min}}$, then $\mixed \not \le y$ for any $y \in \abar_{\ge x_{r_{\max}+1}, \le x_{s_{\min}}}$.
\end{corollary}

The following result is the analogue of Lemma \ref{lem:minmax} where again we restrict to $y_j\in\aint_{>x_{r_{\max}+1},<x_{s_{\min}}}$ rather than allowing for all $y_j\in\aint$.
\begin{lemma}
\label{lem:minmaxcrit}
For any $y_j\in \aint_{>x_{r_{\max}+1},<x_{s_{\min}}}$: If $m^{=}_{\min}(y_j)<i_{\LL}$ then $\mathrm v_j=0$, and if $m^{=}_{\max}(y_j)>i_{\LL}$ then $\mathrm v_j=1-a$.
\end{lemma}
\begin{proof}
We prove that $m^{=}_{\max}(y_j)>i_{\LL}\Rightarrow\mathrm v_j=1-a$; the proof of $m^{=}_{\min}(y_j)<i_{\LL}\Rightarrow\mathrm v_j=0$ is analogous.

Set $y_{j_0}:=y_j\in  \aint_{>x_{r_{\max}+1},<x_{s_{\min}}}$ and construct the sequence $y_{j_0}<y_{j_1}<\cdots <y_{j_p}$, for some $p<\infty$, iteratively, according to the algorithm below. The sequence will be constructed so that $y_{j_i}\in \aint_{>x_{r_{\max}+1},<x_{s_{\min}}}$ for every $i\in\llbracket 0,p\rrbracket$, $\mathrm v_{j_i}=\mathrm v_{j_{i+1}}$ for all $i\in \llbracket 0,p-1\rrbracket$, and $\mathrm v_{j_p}=1-a$. Clearly, it will then follow that $\mathrm v_j=\mathrm v_{j_0}=1-a$, completing the proof.

Assume that the sequence $y_{j_0}<y_{j_1}<\cdots <y_{j_i}$ has been constructed. Set $M:=m^{=}_{\max}(y_{j_i})$ and note that $i_{\LL}<m^{=}_{\max}(y_{j_0})\le M<i_{s_{\min}}$. Let $b$ be the index satisfying $i_b<M<i_{b+1}$ so that $\LL\le b\le s_{\min}-1$.  Consider the following two cases:\\

\begin{itemize}
\item $M\notin\{i_m-2,i_m-1\}$ for every $\LL<m$. Choose $\sigma\in\mathcal N_=$ such that $\sigma(y_{j_i})=M$ (such a $\sigma$ must exist by the definition of $M$) and set $y_r=\sigma^{-1}(M+1), y_s=\sigma^{-1}(M+2)$. Note that $i_{b+1}\notin \{M+1,M+2\}$ since $M\notin\{i_m-2,i_m-1\}$ for every $\LL<m$, so in particular, we can take $b+1=m$ (using $b+1>\LL$). Hence, we have $i_b<M,M+1,M+2<i_{b+1}$, so $M,M+1,M+2\in \llbracket i_b+1,i_{b+1}-1\rrbracket$, and hence $y_r,y_s\in \aint$. Note that $y_{j_i}<y_r$ since otherwise their positions in $\sigma$ can be swapped to contradict the maximality of $M$. Further, $M,M+1,M+2\in \llbracket i_b+1,i_{b+1}-1\rrbracket\Rightarrow \sigma(y_r),\sigma(y_s)\in \llbracket i_b+1,i_{b+1}-1\rrbracket\subseteq \I_{\llbracket r_{\max}+1,s_{\min}-1\rrbracket}$, where the last containment holds since $b\le s_{\min}-1$ (as shown above), and since $r_{\max}+1\le b$ (because $r_{\max}+1<\LL\le b$ as $(r_{\max},s_{\min})$ is an $\LL$-splitting pair). Corollary \ref{cor:max_plus} now yields $y_r,y_s\in \aint_{>x_{r_{\max}+1},<x_{s_{\min}}}\cup\{\mixed\}$. We now choose $y_{j_{i+1}}$ as follows:
\begin{enumerate}[(1)]
\item If $y_r\in \aint_{>x_{r_{\max}+1},<x_{s_{\min}}}$ set $y_{j_{i+1}}:=y_r$. Then we see that $y_{j_i}<y_{j_{i+1}}$ and that $y_{j_{i+1}}\in \aint_{>x_{r_{\max}+1},<x_{s_{\min}}}$ so Proposition \ref{prop:avcrit}(c) yields $\mathrm v_{j_{i+1}}=\mathrm v_{j_{i}}$.
\item If $y_r=\mixed$, then $y_s \in \aint_{>x_{r_{\max}+1},<x_{s_{\min}}}$. If $\mixed \not\ge x_{r_{\max}+1}$, then $\mixed \not \ge y_{j_i}$, a contradiction. Otherwise, $\mixed \not\le x_{s_{\min}}$, so $\mixed \not \le y_s$. Hence, we can swap the positions of $y_r=\mixed$ and $y_s$, which reduces to (1). 
\end{enumerate}

\item $M\in\{i_m-2,i_m-1\}$ for some $\LL<m$. In this case the sequence will be terminated with $p:=i$. Arguing as in the analogous case in Lemma \ref{lem:minmax}, we get that $y_{j_i}<x_m$. Note that $m=b+1\le s_{\min}$ (the last inequality was shown above), so since $\LL+1\le m\le s_{\min}$, Proposition \ref{prop:avcrit}(f) yields $\mathrm v_{j_i}=1-a$.
\end{itemize}
\end{proof}
The following result can be viewed as a continuation of Lemma \ref{lem:minmaxcrit}. To ease the notation we will use
\begin{align}
\label{eq:I}
\I_j := \llbracket i_j+1,i_{j+1}-1\rrbracket\quad \text{and} \quad \I_S:= \cup_{j\in S} \I_j\quad \text{for}\quad S\subseteq \llbracket 0, k\rrbracket.
\end{align}

\begin{lemma}
\label{lem:minellmaxcrit}
For any $y_j\in \aint_{>x_{r_{\max}+1},<x_{s_{\min}}}$:  If $\min_{\s\in\{-,=,+\}}m_{\min}^{\s}(y_j)<i_{\LL}+1_{\s}$ then $\mathrm v_j=0$, and if $\max_{\s\in\{-,=,+\}}m_{\max}^{\s}(y_j)>i_{\LL}+1_{\s}$  then $\mathrm v_j=1-a$. 
\end{lemma}
\begin{proof}
We will prove $\max_{\s\in\{-,=,+\}}m_{\max}^{\s}(y_j)>i_{\LL}+1_{\s}\Rightarrow\mathrm v_j=1-a$; the proof of $\min_{\s\in\{-,=,+\}}m_{\min}^{\s}(y_j)<i_{\LL}+1_{\s}\Rightarrow\mathrm v_j=0$ is analogous. Fix $\s\in\{-,=,+\}$ and $\sigma\in \mathcal N_{\s}$ such that $\sigma(y_j)>\sigma(x_{\LL})=i_{\LL}+1_{\s}$. There are three cases to consider:\\

\begin{enumerate}[(1)]

\item $\s$ is $=$. We have $m^{=}_{\max}(y_j)\ge \sigma(y_j)$ and by assumption $\sigma(y_j)>\sigma(x_{\LL})=i_{\LL}$. Hence, $m^{=}_{\max}(y_j)>i_{\LL}$  and the proof is complete by Lemma \ref{lem:minmaxcrit}.\\

\item $\s$ is $+$. Let $q:=\sigma(y_j)>i_{\LL}+1$. We are going to apply Lemma \ref{lem:range} with $y_j,=$, and $q$ so we will check its conditions. Since $i_{\min}(y_j)>\LL$, Lemma \ref{lem:UL} and Corollary \ref{cor:range} yield
 $u_{=}(y_j)=u_{+}(y_j)\ge m_{\max}^+(y_j)\ge q$ and  $l_{=}(y_j)\le l_{+}(y_j)\le m_{\min}^+(y_j)\le\sigma(y_j)=q$, so we conclude that  $q\in\llbracket l_{=}(y_j),u_{=}(y_j)\rrbracket$. Next we show that $q\neq i_m$ for any $m\in [k]$. Indeed, if $m\le \LL$ then $i_m\le i_{\LL}<q$, and if $m>\LL$, then $i_m=q$ implies $\sigma(y_j)=\sigma(x_m)$, which is impossible since $y_j\in \aint_{>x_{r_{\max}+1},<x_{s_{\min}}}\subseteq\aint$. It follows from Lemma \ref{lem:range} that there exists $\sigma'\in\mathcal N_=$ such that $\sigma'(y_j)=q$. It follows that $m_{\max}^=(y_j)\ge \sigma'(y_j)=q>i_{\LL}$, and the proof is complete by Lemma \ref{lem:minmaxcrit}. \\

\item $\s$ is $-$. If $m_{\max}^=(y_j)>i_{\LL}$ we are done by Lemma \ref{lem:minmaxcrit}. Suppose then that $m^{=}_{\max}(y_j)<i_{\LL}$ (note that $m^{=}_{\max}(y_j)=i_{\LL}$ is impossible). 
\begin{claim}
\label{cl:1}
$m_{\max}^-(y_j)=i_{\LL}=\sigma(y_j)$.
\end{claim}
\begin{proof}
Suppose for contradiction that $m_{\max}^-(y_j)\ge i_{\LL}+1$, so there must exist $\sigma_1\in\mathcal N_-$ with $\sigma_1(y_j)\ge i_{\LL}+1$.  Since $i_{\min}(y_j)>\LL$, Lemma \ref{lem:UL} and Corollary \ref{cor:range} yield
 $u_{=}(y_j)=u_{-}(y_j)\ge m_{\max}^-(y_j)\ge \sigma_1(y_j)=i_{\LL}+1$. On the other hand, by  Corollary \ref{cor:range} and the assumption  $m_{\max}^=(y_j)<i_{\LL}$, we have $l_{=}(y_j)\le m_{\min}^=(y_j)\le  m_{\max}^=(y_j)<i_{\LL}$, so we conclude that  $i_{\LL}+1\in\llbracket l_{=}(y_j),u_{=}(y_j)\rrbracket$. By Corollary \ref{cor:splitequiv}, $i_m\neq i_{\LL}+1$ for any $m\in [k]$ so Lemma \ref{lem:range} implies that there exists $\sigma_2\in\mathcal N_=$ satisfying $\sigma_2(y_j)=i_{\LL}+1$, which contradicts $m^{=}_{\max}(y_j)<i_{\LL}$. We conclude that $m_{\max}^-(y_j)\le i_{\LL}$. Since, by assumption, $\sigma(y_j)>\sigma(x_{\LL})=i_{\LL}-1$ we get $m_{\max}^-(y_j)=\sigma(y_j)= i_{\LL}$.
\end{proof}
Let $y_v$ be such that $\sigma(y_v)=i_{\LL}+1$ and note that $y_v\in\aint$ by Corollary \ref{cor:splitequiv}. We must have $y_j<y_v$ since if $y_j\incomp y_v$ (by Claim \ref{cl:1} it is impossible to have $y_v<y_j$), then we can swap the positions of $y_j$ and $y_v$ in $\sigma$ to get $\sigma_3\in\mathcal N_-$ satisfying $\sigma_3(y_j)=i_{\LL}+1$, which contradicts  Claim \ref{cl:1}. Next we show that there exists $\sigma'\in\mathcal N_=$ satisfying $\sigma'(y_j)=i_{\LL}-1$ and $\sigma'(y_v)=i_{\LL}+1$. Indeed, since we assume $m_{\max}^=(y_j)<i_{\LL}$, we have that, for any $\sigma_4\in\mathcal N_{=}$, $\sigma_4(y_j)<\sigma_4(x_{\LL})$. Hence, since by the assumption  $\sigma(y_j)>\sigma(x_{\LL})$,  we must have $y_j\incomp x_{\LL}$. Swapping the positions of $y_j$ and $x_{\LL}$ in $\sigma$ yields $\sigma'$, where we used Claim \ref{cl:1}.\\

We will now analyze the element $y_v$. Since $\sigma'(y_v)=i_{\LL}+1$ we see that $\sigma'(y_v)\in \I_{\llbracket r_{\max}+1,s_{\min}-1\rrbracket}$ because $(r_{\max},s_{\min})$ is an $\LL$-splitting pair. Hence, Corollary \ref{cor:max_plus} yields that either $y_v=\mixed$ or $y_v \in \aint_{>x_{r_{\max}+1},<x_{s_{\min}}}$. Consider both cases:\\

\begin{enumerate}[(a)]
\item $y_v\in \aint_{>x_{r_{\max}+1},<x_{s_{\min}}}$. Since $y_j<y_v$, and since there exists $\sigma'\in\mathcal N_=$ satisfying  $\sigma'(y_j)=i_{\LL}-1$ and $\sigma'(y_v)=i_{\LL}+1$, Proposition \ref{prop:avcrit}(d) yields $\mathrm v_j=\mathrm v_v$. On the other hand, $\mathrm v_v=1-a$ by Lemma \ref{lem:minmaxcrit} since $m_{\max}^=(y_v)>i_{\LL}$ as $\sigma'(y_v)=i_{\LL}+1$. We conclude that $\mathrm v_j=1-a$, which proves the lemma.\\

\item $y_v=\mixed$. Since $y_v>y_j$ and $y_j \in \aint_{>x_{r_{\max}+1},<x_{s_{\min}}}$ (as we cannot have $y_j=\mixed$), we have $\mixed > x_{r_{\max}+1}$. Hence, we must have $\mixed \not \le x_{s_{\min}}$. Let $z$ be such that $\sigma(z)=i_{\LL}+2$ and note that $\sigma'(z)=i_{\LL}+2$ as well (since $\sigma'$ was obtained from $\sigma$ by swapping the positions of $y_j$ and $x_{\LL}$ in $\sigma$). If $\sigma'(z)\in \I_{\llbracket r_{\max}+1,s_{\min}-1\rrbracket}$, then, by Corollary \ref{cor:max_plus}, since $z\neq \mixed$, we must have $z\in \aint_{>x_{r_{\max}+1},<x_{s_{\min}}}$. Recall that $\mixed \not \le x_{s_{\min}}$ so $\mixed \not \le z$. Hence, we can swap $y_v=\mixed$ and $z$ to reduce to the case (a).

Suppose then that $\sigma'(z)\notin \I_{\llbracket r_{\max}+1,s_{\min}-1\rrbracket}$:
\begin{claim}
\label{cl:2}
If $\sigma'(z)\notin \I_{\llbracket r_{\max}+1,s_{\min}-1\rrbracket}$ then $z=x_{\LL+1}$ and $i_{\LL+1}=i_{\LL}+2$.
\end{claim} 
\begin{proof}
Since $\sigma'(z)=i_{\LL}+2\notin \I_{\llbracket r_{\max}+1,s_{\min}-1\rrbracket}$ we get that $i_{\LL}+2\notin \I_{\LL}$ (because $r_{\max}+1<\LL<s_{\min}$ as $(r_{\max},s_{\min})$ is an $\LL$-splitting pair). Hence, $i_{\LL}+2\ge i_{\LL+1}$ (since $i_{\LL}+2\le i_{\LL}$ is impossible). On the other hand, Corollary \ref{cor:splitequiv} yields $i_{\LL}+1<i_{\LL+1}\le i_{\LL}+2$ so we conclude $i_{\LL}+2=i_{\LL+1}$. Since $\sigma'\in\mathcal N_=$ we also conclude that $z=x_{\LL+1}$.
\end{proof}
Since $(r_{\max},s_{\min})$ is an $\LL$-splitting pair, we have that either $s_{\min}=\LL+1$ or $s_{\min}>\LL+1$. If $s_{\min}=\LL+1$ then, since by assumption $y_j\in \aint_{>x_{r_{\max}+1},<x_{s_{\min}}}$, we have $y_j<x_{s_{\min}}=x_{\LL+1}$. Since, by Claim \ref{cl:1} and Claim \ref{cl:2}, $\sigma(y
_j)=i_{\LL}=i_{\LL+1}-2$, Proposition \ref{prop:avcrit}(h) shows that $\mathrm v_j=1-a$. 

Suppose then that $s_{\min}>\LL+1$. Consider the set 
\[
\gamma := \{y\in \alpha : \sigma'(y) \in \I_{\llbracket \LL+1,s_{\min}-1\rrbracket}, y\not > x_{l+1} \}.
\]
We claim that $\gamma$ is nonempty. Indeed, since $(\LL,s_{\min})$ is a splitting pair, Lemma \ref{lem:splitequiv} yields $y^{\sigma'}\in\bsg_{\LL}\cup\bsg_{s_{\min}}$ such that  $\sigma'(y^{\sigma'})\in \I_{\llbracket \LL+1,s_{\min}-1\rrbracket}$. We must have that either $y^{\sigma'} \not \le x_{s_{\min}}$ or $y^{\sigma'}  \not \ge x_{l+1}$. We cannot have $y^{\sigma'}  \not \le x_{s_{\min}}$ since $r_{\max}+1<\LL+1$ and $y^{\sigma'}  \neq \mixed$ (as $\sigma'(\mixed)=i_{\LL}+1\notin  \I_{\llbracket \LL+1,s_{\min}-1\rrbracket}\ni \sigma'(y^{\sigma'} )$) imply $y^{\sigma'} \in \aint_{>x_{r_{\max}+1},<x_{s_{\min}}}$. Hence, $y^{\sigma'} \in \gamma$. Now pick $y\in \gamma^\downarrow$, which exists as $\gamma$ is nonempty. Note that $y\in \gamma^\downarrow$ implies that $y\in \aint_{>x_{r_{\max}+1},<x_{s_{\min}}}$ because $ \sigma'(y) \in \I_{\llbracket \LL+1,s_{\min}-1\rrbracket}\subseteq  \I_{\llbracket r_{\max}+1,s_{\min}-1\rrbracket}$ yields, by Corollary \ref{cor:max_plus}, $y\in  \aint_{>x_{r_{\max}+1},<x_{s_{\min}}}\cup \{\mixed\}$, and $y\neq \mixed$ since   $\sigma'(\mixed)=i_{\LL}+1\notin  \I_{\llbracket \LL+1,s_{\min}-1\rrbracket}\ni \sigma'(y )$.

We will show next that the positions of $y$ and $y_v$ can be swapped in both $\sigma$ and $\sigma'$ to yield valid linear extensions in $\mathcal N_-,\mathcal N_=$, respectively.  This completes the proof since we reduce back to 3(a).

Let us now verify that the swaps yield valid linear extensions. We will show the validity of the swap of $\sigma$; the argument for $\sigma'$ is analogous since by construction $\sigma$ and $\sigma'$ are the same up to the swap of $y_j$ and $x_{\LL}$. Suppose this swap violated some relation so that there exists $w$ such that $\sigma(y_v)=i_{\LL}+1<\sigma(w)<\sigma(y)$, satisfying either $y_v<w$ or $w<y$. We cannot have $\mixed =y_v<w$ because $\sigma(w)\in \llbracket i_{r_{\max}+1}, i_{s_{\min}}\rrbracket$ implies, by Corollary \ref{cor:max_plus}, that $w\le x_{s_{\min}}$ (as $w\neq \mixed$). But then $\mixed = y_v < w \le x_{s_{\min}}$, which contradicts $\mixed \not \le x_{s_{\min}}$, as was shown at  the beginning of (3). We also cannot have $w<y$ since, otherwise, $w\not \ge x_{l+1}$ by the definition of $\gamma$. But if $w \not \in \aint$, then $w = x_r$ for some $r\ge l+1$ (as $i_{\LL}+1<\sigma(w)$) which implies $w\ge x_{l+1}$, a contradiction. On the other hand, if $w\in \aint$, then combined with $\sigma(w)\in \llbracket i_{l+1}, i_{s_{\min}}\rrbracket$ we have that $\sigma(w)\in \I_{\llbracket l+1,s_{\min}-1\rrbracket}$. Hence, $w\in \gamma$, which contradicts $y\in \gamma^\downarrow$. 
\end{enumerate}
\end{enumerate}
\end{proof}

Next we move to proving the analogue of Lemma \ref{lem:a=1}. We will again use the notation \eqref{eq:I}.
\begin{lemma}
\label{lem:a=1crit}
$a=1$. 
\end{lemma}
\begin{proof}
We will show that there exists $y_j\in \aint_{>x_{r_{\max}+1},<x_{s_{\min}}}$ such that $y_j\incomp x_{\LL}$. This will complete the proof since, by Assumption \ref{ass:cl}, $y_j\incomp x_{\LL}$ implies that there exist $\sigma,\sigma'\in\cup_{\s\in\{-,=,+\}}\mathcal N_{\s}$ satisfying  $\sigma(y_j)>\sigma(x_{\LL})$ and $\sigma'(y_j)<\sigma'(x_{\LL})$. Applying Lemma \ref{lem:minellmaxcrit} yields $0=\mathrm v_j=1-a$ so $a=1$.

We now show that there exists $y_j\in \aint_{>x_{r_{\max}+1},<x_{s_{\min}}}$ such that $y_j\incomp x_{\LL}$. Suppose for contradiction that such $y_j$ does not exist. Then, for any $y\in \aint_{>x_{r_{\max}+1},<x_{s_{\min}}}$, we must have either $y<x_{\LL}$ or $y>x_{\LL}$. In particular, we have the disjoint union
\begin{align}
\label{eq:union}
\aint_{>x_{r_{\max}+1},<x_{s_{\min}}}= [\aint_{>x_{r_{\max}+1},<x_{s_{\min}}}\cap \aint_{<x_{\LL}}]\cup  [\aint_{>x_{r_{\max}+1},<x_{s_{\min}}}\cap \aint_{>x_{\LL}}].
\end{align}
Let us show that 
\begin{align}
\label{eq:eps<>}
|\aint_{>x_{r_{\max}+1},<x_{s_{\min}}}\cap \aint_{>x_{\LL}}|\le |\I_{\llbracket \LL,s_{\min}-1\rrbracket}|-1 \quad \text{and}\quad |\aint_{>x_{r_{\max}+1},<x_{s_{\min}}}\cap \aint_{<x_{\LL}}|\le |\I_{\llbracket r_{\max}+1,\LL-1\rrbracket}|-1;
\end{align}
we prove the first inequality and the proof of the second inequality is analogous. Given any $\sigma\in\mathcal N_+$ and $y\in \aint_{>x_{r_{\max}+1},<x_{s_{\min}}}\cap \aint_{>x_{\LL}}$ we have $i_{\LL}+1<\sigma(y)<i_{s_{\min}}$ so $\sigma(y)\in \I_{\llbracket \LL,s_{\min}-1\rrbracket}\backslash\{i_l+1\}$. It follows that $|\aint_{>x_{r_{\max}+1},<x_{s_{\min}}}\cap \aint_{>x_{\LL}}|\le |\I_{\llbracket \LL,s_{\min}-1\rrbracket}|-1$ as desired. By \eqref{eq:union} and \eqref{eq:eps<>} we now get
\begin{align}
\label{eq:eps<>union}
|\aint_{>x_{r_{\max}+1},<x_{s_{\min}}}|\le|\I_{\llbracket r_{\max}+1,\LL-1\rrbracket}|+ |\I_{\llbracket \LL,s_{\min}-1\rrbracket}|-2= |\I_{\llbracket r_{\max}+1,s_{\min}-1\rrbracket}|-2.
\end{align}
However, by Lemma \ref{lem:generalmixed}, $|\aint_{>x_{r_{\max}+1},<x_{s_{\min}}}|=|\I_{\llbracket r_{\max}+1,s_{\min}-1\rrbracket}|-|\{\text{mixed elements}\}|$. Hence, the number of mixed elements is at least 2 which means that the maximal splitting pair is supercritical, which contradicts Proposition \ref{prop:max_notions}.
\end{proof}

We are now ready to prove Theorem \ref{thm:critsec}. 

\begin{proof}[Proof of Theorem \ref{thm:critsec}]
We start by proving the analogue of \eqref{eq:supcritequiv}.

\begin{lemma}
\label{lem:incompcrit}
Let $y\in \aint_{>x_{r_{\max}+1},<x_{s_{\min}}}$.
\begin{enumerate}[(a)]
\item If there exists $\sigma\in\mathcal N_=$ such that either $\sigma(y)=i_{\LL}-1$ or $\sigma(y)=i_{\LL}+1$, then $y\incomp x_{\LL}$. 
\item If there exists $\sigma\in\mathcal N_-\cup \mathcal N_+$ such that $\sigma(y)=i_{\LL}$, then $y\incomp x_{\LL}$. 
\end{enumerate}
\end{lemma}
\begin{proof}
$~$
\begin{enumerate}[(a)]
\item  We proceed as in the proof of Theorem \ref{thm:supcritsec} where we use Lemma \ref{lem:supcritE} rather than \eqref{eq:exth}.

\item Let $y\in \aint_{>x_{r_{\max}+1},<x_{s_{\min}}}$ be such that there exists  $\sigma \in\mathcal N_-$ with $\sigma(y)=i_{\LL}$; the proof for the case $\sigma \in\mathcal N_+$ is analogous. Since we cannot have $y<x_{\LL}$ it suffices to show that $y\not>x_{\LL}$. Suppose for contradiction that $y>x_{\LL}$. By Lemma \ref{lem:range}, $l_{-}(y)\le i_{\LL}$ so by Lemma \ref{lem:UL} $l_{=}(y)\le i_{\LL}+1$. On the other hand, for any $\sigma'\in \mathcal N_=$, Corollary \ref{cor:range} yields $i_{\LL}=\sigma'(x_{\LL})<\sigma'(y)\le m_{\max}^=(y)\le u_=(y)$ so $u_=(y)\ge i_{\LL}+1$. Since $i_{\LL}+1\neq i_m$ for any $m\in [k]$ (by Corollary \ref{cor:splitequiv}), Lemma \ref{lem:range} yields $\sigma''\in\mathcal N_=$ such that $\sigma''(y)=i_{\LL}+1$. By part (a), $y\incomp x_{\LL}$, which contradicts $y>x_{\LL}$. 
\end{enumerate}
\end{proof}
We now prove $|\mathcal N_+(\sim,\sim)|=0$; the proof of $|\mathcal N_-(\sim,\sim)|=0$ is analogous. Suppose for contradiction that $|\mathcal N_+(\sim,\sim)|>0$ so there exists $\sigma\in \mathcal N_+$ such that $y_u:=\sigma^{-1}(i_{\LL}-1)$ and $y_v:=\sigma^{-1}(i_{\LL})$ satisfy $y_u,y_v<x_{\LL}$. Since $ i_{\LL}-1,i_{\LL}\in \I_{\llbracket r_{\max}+1,s_{\min}-1\rrbracket}$ (because $(r_{\max},s_{\min})$ is an $\LL$-splitting pair so $i_{r_{\max}+1}\le i_{\LL-1}<i_{\LL}-1$ by Corollary \ref{cor:splitequiv}), Corollary \ref{cor:max_plus} yields $y_u,y_v\in \aint_{>x_{r_{\max}+1},<x_{s_{\min}}}\cup\{\mixed\}$. Consider the following two cases:

If $y_v\in \aint_{>x_{r_{\max}+1},<x_{s_{\min}}}$, then, by Lemma \ref{lem:incompcrit}(b), $y_v\incomp x_{\LL}$ which contradicts $y_v<x_{\LL}$. 

If $y_v=\mixed$, we have $y_u \in  \aint_{>x_{r_{\max}+1},<x_{s_{\min}}}$. Then, because $\mixed < x_{\LL} < x_{s_{\min}}$, we must have $\mixed \not \ge x_{r_{\max}+1}$, which implies $y_v = \mixed \not \ge y_u$. Hence, we can swap the positions of $y_u$ and $y_v$ in $\sigma$ to reduce to the previous case.
\end{proof}

\section*{Notation index}

\begin{itemize}
\item $[p]:=\{1,\ldots, p\}$ for positive integers $p$.\\

\item $\llbracket p,q \rrbracket :=\{p,p+1,\ldots,q-1,q\}$ for integers $p\le q$; \eqref{eq:bracket}.\\
\item $\abar=\{y_1,\ldots,y_{n-k},x_0,x_1,\ldots, x_k,x_{k+1}\}$ and $\aint=\{y_1,\ldots,y_{n-k}\}$ where $x_0$ (res. $x_{k+1}$) is smaller (res. bigger) than every element in $\abar$.\\
\item $i_0=0$ and $i_{k+1}=n+1$. $j_0=-1$ and $j_{p+1}=k+1$. \\
\item  $1_{\s}=1_{\{\s\text{ is }+\}}-1_{\{\s\text{ is }-\}}$ for $\s\in\{-,=,+\}$.\\
\item $\bsg_i=\aint\backslash(\aint_{<x_i}\cup \aint_{>x_{i+1}})$ and $\bsg_S=\cup_{i\in S}\bsg_i$; \eqref{eq:betai}.\\

\item $i_{\max}(y)$ (res. $i_{\min}(y)$) is the maximum (res. minimum) number such that $y>x_{i_{\max}(y)}$ (res. $y<x_{i_{\min}(y)}$); Definition \ref{def:lmu}.\\

\item $l_{\s}(y):=\max_{r\le i_{\max}(y)}(i_r+1_{\s}+|\abar_{>x_r,<y}|+1)$ and $u_{\s}(y):=\min_{s\ge i_{\min}(y)}(i_s+1_{\s}-|\abar_{>y,<x_s}|-1)$; Definition \ref{def:lmu}.\\

\item  $m^{\s}_{\min}(y)=\min_{\sigma\in\mathcal N_{\s}}\sigma(y)\quad\text{and}\quad m^{\s}_{\max}(y)=\max_{\sigma\in\mathcal N_{\s}}\sigma(y)$ for $\s\in\{-,=,+\}$ and $y\in\aint$; Definition \ref{def:lmu}.\\

\item $i_j^\s := i_j+1_{j=\LL}1_{\s}$; Definition \ref{def:lmu}.\\

\item $r_{\max}=\max_{\iota}r_{\iota}$ and $s_{\min}=\min_{\iota}s_{\iota}$  where $(r_{\iota},s_{\iota})$ are the sharp-critical $\LL$-splitting pairs; Definition \ref{def:rmaxsmin}.\\

\item $\mixed$; Corollary \ref{cor:max}.\\

\item $\mathcal\poly_{\max}$, $\bsg_{\max}:=\bsg_{\llbracket 0, r_{\max}\rrbracket\cup \llbracket s_{\min}, k\rrbracket}$, $\aint\backslash\bsg_{\max}=\aint_{>x_{r_{\max}+1},<x_{s_{\min}}}$, and $E^{\perp}:=\R^{\aint\backslash \bsg_{\max}}$; \eqref{eq:Kmax}, \eqref{eq:Eperp}.\\

\item $\llbracket i_j,i_{j+1}\rrbracket^{\s}:=\llbracket i^\s_j,i^\s_{j+1}\rrbracket = \llbracket i_j+1_{j=\LL}1_{\s},i_{j+1}+1_{j+1=\LL}1_{\s}\rrbracket$ and $\llbracket i_j+1,i_{j+1}-1\rrbracket^{\s}:=\llbracket i^\s_j+1,i^\s_{j+1}-1\rrbracket$; \eqref{eq:[]}.\\

\item $\I_{j_q}:=\llbracket i_j+1,i_{j+1}-1\rrbracket\quad \text{for }j_q\in\llbracket 0,k\rrbracket,\quad \I_J:=\cup_{j_q\in J}\I_{j_q}$; \eqref{eq:Ij}, \eqref{eq:I}.\\

\end{itemize}

\subsection*{Acknowledgments} We are grateful to David Jerison, Greta Panova, and Yufei Zhao for helpful comments on this work. We are especially grateful to Swee Hong Chan, Igor Pak, and Ramon van Handel for their valuable comments. We also thank the anonymous referee for many useful comments that improved this paper; Corollary \ref{cor:referee} is due to them. 

Zhao Yu Ma was partly supported by UROP at MIT. This material is based upon work supported by the National Science Foundation under Award Number 2002022.

\bibliographystyle{amsplain0}
\bibliography{ref_Stanley}

\providecommand{\bysame}{\leavevmode\hbox to3em{\hrulefill}\thinspace}
\providecommand{\MR}{\relax\ifhmode\unskip\space\fi MR }
\providecommand{\MRhref}[2]{%
  \href{http://www.ams.org/mathscinet-getitem?mr=#1}{#2}
}
\providecommand{\href}[2]{#2}
\begin{thebibliography}{10}

\bibitem{Bra15}
Petter Br\"{a}nd\'{e}n, \emph{Unimodality, log-concavity, real-rootedness and
  beyond}, Handbook of enumerative combinatorics, Discrete Math. Appl. (Boca
  Raton), CRC Press, Boca Raton, FL, 2015, pp.~437--483.

\bibitem{Ber89}
Francesco Brenti, \emph{Unimodal, log-concave and {P}\'{o}lya frequency
  sequences in combinatorics}, Mem. Amer. Math. Soc. \textbf{81} (1989),
  viii+106.

\bibitem{Ber94}
Francesco Brenti, \emph{Log-concave and unimodal sequences in algebra,
  combinatorics, and geometry: an update}, Jerusalem combinatorics '93,
  Contemp. Math., vol. 178, Amer. Math. Soc., Providence, RI, 1994, pp.~71--89.

\bibitem{CP21}
Swee~Hong Chan and Igor Pak, \emph{Log-concave poset inequalities}, Preprint
  arXiv:2110.10740 (2021).

\bibitem{CP22a}
Swee~Hong Chan and Igor Pak, \emph{Introduction to the combinatorial atlas},
  Expositiones Mathematicae, to appear (2022).

\bibitem{chan2023equality}
Swee~Hong Chan and Igor Pak, \emph{Equality cases of the
  {A}lexandrov--{F}enchel inequality are not in the polynomial hierarchy},
  Preprint arXiv:2309.05764 (2023).

\bibitem{chanpakequiv2023}
Swee~Hong Chan and Igor Pak, \emph{Linear extensions of finite posets},
  Preprint arXiv:2311.02743 (2023).

\bibitem{chan2021extensions}
Swee~Hong Chan, Igor Pak, and Greta Panova, \emph{Extensions of the
  {K}ahn--{S}aks inequality for posets of width two}, Preprint arXiv:2106.07133
  (2021).

\bibitem{CP22}
Swee~Hong Chan, Igor Pak, and Greta Panova, \emph{Effective poset
  inequalities}, Preprint arXiv:2205.02798 (2022).

\bibitem{CFG80}
F.~R.~K. Chung, P.~C. Fishburn, and R.~L. Graham, \emph{On unimodality for
  linear extensions of partial orders}, SIAM J. Algebraic Discrete Methods
  \textbf{1} (1980), 405--410.

\bibitem{Huh18}
June Huh, \emph{Combinatorial applications of the {H}odge-{R}iemann relations},
  Proceedings of the {I}nternational {C}ongress of {M}athematicians---{R}io de
  {J}aneiro 2018. {V}ol. {IV}. {I}nvited lectures, World Sci. Publ.,
  Hackensack, NJ, 2018, pp.~3093--3111.

\bibitem{kalai2022work}
Gil Kalai, \emph{The work of {J}une {H}uh}, Proceedings of the International
  Congress of Mathematicians, vol.~28, 2022.

\bibitem{P19}
Igor Pak, \emph{Combinatorial inequalities}, Notices Amer. Math. Soc.
  \textbf{66} (2019), 1109--1112.

\bibitem{P22survey}
Igor Pak, \emph{What is a combinatorial interpretation?}, Preprint
  arXiv:2209.06142 (2022).

\bibitem{SW14}
Adrien Saumard and Jon~A. Wellner, \emph{Log-concavity and strong
  log-concavity: a review}, Stat. Surv. \textbf{8} (2014), 45--114.

\bibitem{Sch14}
Rolf Schneider, \emph{Convex bodies: the {B}runn-{M}inkowski theory}, expanded
  ed., Cambridge University Press, 2014.

\bibitem{SvH20}
Yair Shenfeld and Ramon van Handel, \emph{The extremals of the
  {A}lexandrov-{F}enchel inequality for convex polytopes}, Acta Math.
  \textbf{231} (2023), 89--204.

\bibitem{Sta81}
Richard~P. Stanley, \emph{Two combinatorial applications of the
  {A}leksandrov-{F}enchel inequalities}, J. Combin. Theory Ser. A \textbf{31}
  (1981), 56--65.

\bibitem{Sta86}
Richard~P. Stanley, \emph{Two poset polytopes}, Discrete Comput. Geom.
  \textbf{1} (1986), 9--23.

\bibitem{Sta89}
Richard~P. Stanley, \emph{Log-concave and unimodal sequences in algebra,
  combinatorics, and geometry}, Graph theory and its applications: {E}ast and
  {W}est ({J}inan, 1986), Ann. New York Acad. Sci., vol. 576, New York Acad.
  Sci., New York, 1989, pp.~500--535.

\bibitem{Stanley00}
Richard~P. Stanley, \emph{Positivity problems and conjectures in algebraic
  combinatorics}, Mathematics: frontiers and perspectives, Amer. Math. Soc.,
  Providence, RI, 2000, pp.~295--319.

\end{thebibliography}

\end{document}